\documentclass[10pt]{article}
\usepackage{graphicx}
\usepackage{xcolor}
\usepackage{amssymb,amsfonts,amsthm,amsmath,calligra,tikz}
\usepackage[utf8]{inputenc}
\usepackage{accents}
\usepackage{slashed}
\usepackage{yfonts}
\usepackage{mathrsfs,pifont}
\theoremstyle{definition}
\usepackage{tikz}
\usepackage[all]{xy}

\usepackage[bookmarksnumbered=true, bookmarksopen=true]{hyperref}

\def\be{\begin{eqnarray}}
\def\ee{\end{eqnarray}}
\def\ben{\begin{eqnarray*}}
\def\een{\end{eqnarray*}}

\usepackage{geometry}
\allowdisplaybreaks

\def\matZ{{\mathbb{Z}}}
\def\matR{{\mathbb{R}}}

\def\matC{{\mathbb{C}}}

\newcommand{\cE}{\mathscr{E}}

\newcommand{\Or}{\textsf{O}}

\newcommand{\cD}{\mathcal{D}}

\newcommand{\cH}{\mathcal{H}}

\newcommand{\cL}{\mathcal{L}}

\newcommand{\cN}{\mathcal{N}}
\newcommand{\cO}{\mathcal{O}}

\newcommand{\cR}{\mathcal{R}}

\newcommand{\cW}{\mathcal{W}}

\newcommand{\bStab}{\mathbf{Stab}}

\newcommand{\End}{\operatorname{End}}
\newcommand{\Hom}{\operatorname{Hom}}
\newcommand{\inv}{\operatorname{inv}}
\newcommand{\Lie}{\operatorname{Lie}}
\newcommand{\Pic}{\operatorname{Pic}}
\newcommand{\pt}{\operatorname{pt}}
\newcommand{\Spec}{\operatorname{Spec}}
\newcommand{\Stab}{\operatorname{Stab}}
\newcommand{\Sym}{\operatorname{Sym}}

\newcommand{\calW}{\mathcal{W}}

\newcommand{\fS}{\mathfrak{S}}
\newcommand{\ft}{\mathfrak{t}}
\newcommand{\bA}{\mathsf{A}}

\newcommand{\bT}{\mathsf{T}}
\newcommand{\bK}{\mathsf{K}}

\newcommand{\fs}{\mathfrak{s}}

\theoremstyle{definition}
\newtheorem{Definition}{Definition}
\newtheorem{Proposition}{Proposition}
\newtheorem{Lemma}{Lemma}
\newtheorem{Corollary}{Corollary}

\newtheorem{Theorem}{Theorem}

\newtheorem{Example}{Example}

\newcommand{\ahilb}{\mathbf{AH}}
\newcommand{\fC}{\mathfrak{C}} 

\def\tb {{\cal{V}}}  
\def\fb {{\cal{W}}}

\makeatletter
\newcommand{\somespecialrotate}[3][]{%
\begingroup
\sbox\@tempboxa{#3}%
\@tempdima=.5\wd\@tempboxa
\sbox\@tempboxa{\rotatebox[#1]{#2}{\usebox\@tempboxa}}%
\advance\@tempdima by -.5\wd\@tempboxa
\mbox{\hskip\@tempdima\usebox\@tempboxa}%
\endgroup}
\linethickness{1 pt}
\newsavebox{\ybb}
\savebox{\ybb}(1,1){
	\put(0, 0){ \color{blue} \line(0, 1){ 36}} \put(2, 0){ \color{blue} \line(0, 1){ 36}} \put(4, 0){ \color{blue} \line(0, 1){ 30}} \put(6, 0){ \color{blue} \line(0, 1){ 28}} \put(8, 0){ \color{blue} \line(0, 1){ 26}} \put(10, 0){ \color{blue} \line(0, 1){ 24}} \put(12, 0){ \color{blue} \line(0, 1){ 20}} \put(14, 0){ \color{blue} \line(0, 1){ 16}} \put(16, 0){ \color{blue} \line(0, 1){ 16}} \put(18, 0){ \color{blue} \line(0, 1){ 14}} \put(20, 0){ \color{blue} \line(0, 1){ 12}} \put(22, 0){ \color{blue} \line(0, 1){ 10}} \put(24, 0){ \color{blue} \line(0, 1){ 10}} \put(26, 0){ \color{blue} \line(0, 1){ 8}} \put(28, 0){ \color{blue} \line(0, 1){ 6}} \put(30, 0){ \color{blue} \line(0, 1){ 4}} \put(32, 0){ \color{blue} \line(0, 1){ 2}} \put(34, 0){ \color{blue} \line(0, 1){ 2}} \put(36, 0){ \color{blue} \line(0, 1){ 2}} \put(0, 0){ \color{blue} \line(1, 0){ 36}} \put(0, 2){ \color{blue} \line(1, 0){ 36}} \put(0, 4){ \color{blue} \line(1, 0){ 30}} \put(0, 6){ \color{blue} \line(1, 0){ 28}} \put(0, 8){ \color{blue} \line(1, 0){ 26}} \put(0, 10){ \color{blue} \line(1, 0){ 24}} \put(0, 12){ \color{blue} \line(1, 0){ 20}} \put(0, 14){ \color{blue} \line(1, 0){ 18}} \put(0, 16){ \color{blue} \line(1, 0){ 16}} \put(0, 18){ \color{blue} \line(1, 0){ 12}} \put(0, 20){ \color{blue} \line(1, 0){ 12}} \put(0, 22){ \color{blue} \line(1, 0){ 10}} \put(0, 24){ \color{blue} \line(1, 0){ 10}} \put(0, 26){ \color{blue} \line(1, 0){ 8}} \put(0, 28){ \color{blue} \line(1, 0){ 6}} \put(0, 30){ \color{blue} \line(1, 0){ 4}} \put(0, 32){ \color{blue} \line(1, 0){ 2}} \put(0, 34){ \color{blue} \line(1, 0){ 2}} \put(0, 36){ \color{blue} \line(1, 0){ 2}}
	\put(1, 1){ \color{red} \line(0, 1){ 2}} \put(1, 3){ \color{red} \line(0, 1){ 2}} \put(1, 5){ \color{red} \line(0, 1){ 2}} \put(1, 7){ \color{red} \line(0, 1){ 2}} \put(1, 7){ \color{red} \line(1, 0){ 2}} \put(1, 9){ \color{red} \line(0, 1){ 2}} \put(1, 9){ \color{red} \line(1, 0){ 2}} \put(1, 11){ \color{red} \line(0, 1){ 2}} \put(1, 11){ \color{red} \line(1, 0){ 2}} \put(1, 13){ \color{red} \line(0, 1){ 2}} \put(1, 13){ \color{red} \line(1, 0){ 2}} \put(1, 15){ \color{red} \line(0, 1){ 2}} \put(1, 15){ \color{red} \line(1, 0){ 2}} \put(1, 17){ \color{red} \line(0, 1){ 2}} \put(1, 19){ \color{red} \line(0, 1){ 2}} \put(1, 19){ \color{red} \line(1, 0){ 2}} \put(1, 21){ \color{red} \line(0, 1){ 2}} \put(1, 21){ \color{red} \line(1, 0){ 2}} \put(1, 23){ \color{red} \line(0, 1){ 2}} \put(1, 23){ \color{red} \line(1, 0){ 2}} \put(1, 25){ \color{red} \line(0, 1){ 2}} \put(1, 25){ \color{red} \line(1, 0){ 2}} \put(1, 27){ \color{red} \line(0, 1){ 2}} \put(1, 27){ \color{red} \line(1, 0){ 2}} \put(1, 29){ \color{red} \line(0, 1){ 2}} \put(1, 29){ \color{red} \line(1, 0){ 2}} \put(1, 31){ \color{red} \line(0, 1){ 2}} \put(1, 33){ \color{red} \line(0, 1){ 2}} \put(3, 1){ \color{red} \line(0, 1){ 2}} \put(3, 1){ \color{red} \line(1, 0){ 2}} \put(3, 3){ \color{red} \line(0, 1){ 2}} \put(3, 3){ \color{red} \line(1, 0){ 2}} \put(3, 5){ \color{red} \line(0, 1){ 2}} \put(3, 5){ \color{red} \line(1, 0){ 2}} \put(3, 15){ \color{red} \line(1, 0){ 2}} \put(3, 17){ \color{red} \line(0, 1){ 2}} \put(3, 21){ \color{red} \line(1, 0){ 2}} \put(3, 23){ \color{red} \line(1, 0){ 2}} \put(3, 25){ \color{red} \line(1, 0){ 2}} \put(3, 27){ \color{red} \line(1, 0){ 2}} \put(5, 1){ \color{red} \line(1, 0){ 2}} \put(5, 3){ \color{red} \line(1, 0){ 2}} \put(5, 7){ \color{red} \line(0, 1){ 2}} \put(5, 7){ \color{red} \line(1, 0){ 2}} \put(5, 9){ \color{red} \line(0, 1){ 2}} \put(5, 11){ \color{red} \line(0, 1){ 2}} \put(5, 11){ \color{red} \line(1, 0){ 2}} \put(5, 13){ \color{red} \line(0, 1){ 2}} \put(5, 13){ \color{red} \line(1, 0){ 2}} \put(5, 17){ \color{red} \line(0, 1){ 2}} \put(5, 19){ \color{red} \line(0, 1){ 2}} \put(5, 19){ \color{red} \line(1, 0){ 2}} \put(5, 21){ \color{red} \line(1, 0){ 2}} \put(5, 25){ \color{red} \line(1, 0){ 2}} \put(7, 5){ \color{red} \line(0, 1){ 2}} \put(7, 7){ \color{red} \line(1, 0){ 2}} \put(7, 9){ \color{red} \line(0, 1){ 2}} \put(7, 11){ \color{red} \line(1, 0){ 2}} \put(7, 15){ \color{red} \line(0, 1){ 2}} \put(7, 17){ \color{red} \line(0, 1){ 2}} \put(7, 17){ \color{red} \line(1, 0){ 2}} \put(7, 23){ \color{red} \line(0, 1){ 2}} \put(7, 23){ \color{red} \line(1, 0){ 2}} \put(9, 1){ \color{red} \line(0, 1){ 2}} \put(9, 3){ \color{red} \line(0, 1){ 2}} \put(9, 3){ \color{red} \line(1, 0){ 2}} \put(9, 5){ \color{red} \line(0, 1){ 2}} \put(9, 9){ \color{red} \line(0, 1){ 2}} \put(9, 9){ \color{red} \line(1, 0){ 2}} \put(9, 11){ \color{red} \line(1, 0){ 2}} \put(9, 13){ \color{red} \line(0, 1){ 2}} \put(9, 15){ \color{red} \line(0, 1){ 2}} \put(9, 19){ \color{red} \line(0, 1){ 2}} \put(9, 19){ \color{red} \line(1, 0){ 2}} \put(9, 21){ \color{red} \line(0, 1){ 2}} \put(11, 1){ \color{red} \line(0, 1){ 2}} \put(11, 3){ \color{red} \line(1, 0){ 2}} \put(11, 5){ \color{red} \line(0, 1){ 2}} \put(11, 5){ \color{red} \line(1, 0){ 2}} \put(11, 7){ \color{red} \line(0, 1){ 2}} \put(11, 7){ \color{red} \line(1, 0){ 2}} \put(11, 9){ \color{red} \line(1, 0){ 2}} \put(11, 13){ \color{red} \line(0, 1){ 2}} \put(11, 13){ \color{red} \line(1, 0){ 2}} \put(11, 15){ \color{red} \line(0, 1){ 2}} \put(11, 15){ \color{red} \line(1, 0){ 2}} \put(11, 17){ \color{red} \line(0, 1){ 2}} \put(13, 1){ \color{red} \line(0, 1){ 2}} \put(13, 1){ \color{red} \line(1, 0){ 2}} \put(13, 5){ \color{red} \line(1, 0){ 2}} \put(13, 7){ \color{red} \line(1, 0){ 2}} \put(13, 9){ \color{red} \line(1, 0){ 2}} \put(13, 11){ \color{red} \line(0, 1){ 2}} \put(13, 13){ \color{red} \line(1, 0){ 2}} \put(13, 15){ \color{red} \line(1, 0){ 2}} \put(15, 1){ \color{red} \line(1, 0){ 2}} \put(15, 3){ \color{red} \line(0, 1){ 2}} \put(15, 5){ \color{red} \line(1, 0){ 2}} \put(15, 7){ \color{red} \line(1, 0){ 2}} \put(15, 11){ \color{red} \line(0, 1){ 2}} \put(15, 11){ \color{red} \line(1, 0){ 2}} \put(15, 13){ \color{red} \line(1, 0){ 2}} \put(17, 3){ \color{red} \line(0, 1){ 2}} \put(17, 7){ \color{red} \line(1, 0){ 2}} \put(17, 9){ \color{red} \line(0, 1){ 2}} \put(17, 11){ \color{red} \line(1, 0){ 2}} \put(19, 1){ \color{red} \line(0, 1){ 2}} \put(19, 1){ \color{red} \line(1, 0){ 2}} \put(19, 3){ \color{red} \line(0, 1){ 2}} \put(19, 3){ \color{red} \line(1, 0){ 2}} \put(19, 5){ \color{red} \line(0, 1){ 2}} \put(19, 5){ \color{red} \line(1, 0){ 2}} \put(19, 7){ \color{red} \line(1, 0){ 2}} \put(19, 9){ \color{red} \line(0, 1){ 2}} \put(19, 9){ \color{red} \line(1, 0){ 2}} \put(21, 3){ \color{red} \line(1, 0){ 2}} \put(21, 5){ \color{red} \line(1, 0){ 2}} \put(21, 7){ \color{red} \line(1, 0){ 2}} \put(21, 9){ \color{red} \line(1, 0){ 2}} \put(23, 1){ \color{red} \line(0, 1){ 2}} \put(23, 5){ \color{red} \line(1, 0){ 2}} \put(23, 7){ \color{red} \line(1, 0){ 2}} \put(25, 1){ \color{red} \line(0, 1){ 2}} \put(25, 1){ \color{red} \line(1, 0){ 2}} \put(25, 3){ \color{red} \line(0, 1){ 2}} \put(25, 5){ \color{red} \line(1, 0){ 2}} \put(27, 3){ \color{red} \line(0, 1){ 2}} \put(27, 3){ \color{red} \line(1, 0){ 2}} \put(29, 1){ \color{red} \line(0, 1){ 2}} \put(29, 1){ \color{red} \line(1, 0){ 2}} \put(31, 1){ \color{red} \line(1, 0){ 2}} \put(33, 1){ \color{red} \line(1, 0){ 2}}}

\newsavebox{\ydb}
\savebox{\ydb}(1,1){
	
	\put(0, 0){ \color{blue} \line(0, 1){ 6}} \put(2, 0){ \color{blue} \line(0, 1){ 6}} \put(4, 0){ \color{blue} \line(0, 1){ 6}} \put(6, 0){ \color{blue} \line(0,1){6}} \put(8, 0){ \color{blue} \line(0, 1){ 6}} \put(10, 0){ \color{blue} \line(0, 1){ 6}} \put(0, 0){ \color{blue} \line(1, 0){10}} \put(0, 2){ \color{blue} \line(1, 0){ 10}} \put(0, 4){ \color{blue} \line(1, 0){ 10}} \put(0, 6){ \color{blue} \line(1, 0){ 10}} \put(0, 6){ \color{green} \line(1, 0){ 2}} \put(2, 4){ \color{green} \line(1, 0){ 2}}\put(2, 4){ \color{green} \line(0, 1){ 2}} \put(4, 0){ \color{green} \line(0, 1){ 4}} 
	 	\put(1, 1){ \color{red} \line(0, 1){ 4}} 
	 	\put(1, 3){ \color{red} \line(1, 0){ 2}}
	 	\put(3, 1){ \color{red} \line(0, 1){ 2}} 
	 	\put(9, 1){ \color{red} \line(0, 1){ 4}}
	 	\put(3, 5){ \color{red} \line(1, 0){ 6}}
	 	 \put(5, 1){ \color{red} \line(0, 1){ 4}}
	 	\put(5, 1){ \color{red} \line(1, 0){ 2}} 
	 	\put(5, 3){ \color{red} \line(1, 0){ 2}}
}

\newcommand{\yd}{  \hskip 5pt  \usebox{\ydb} \vspace{ 11mm} \hskip 45pt   }

\newsavebox{\ydn}
\savebox{\ydn}(1,1){
	
	\put(0, 0){ \color{blue} \line(0, 1){ 6}} \put(2, 0){ \color{blue} \line(0, 1){ 6}} \put(4, 0){ \color{blue} \line(0, 1){ 6}} \put(6, 0){ \color{blue} \line(0,1){6}} \put(8, 0){ \color{blue} \line(0, 1){ 6}} \put(10, 0){ \color{blue} \line(0, 1){ 6}} \put(0, 0){ \color{blue} \line(1, 0){10}} \put(0, 2){ \color{blue} \line(1, 0){ 10}} \put(0, 4){ \color{blue} \line(1, 0){ 10}} \put(0, 6){ \color{blue} \line(1, 0){ 10}} \put(0, 6){ \color{green} \line(1, 0){ 2}} \put(2, 4){ \color{green} \line(1, 0){ 2}}\put(2, 4){ \color{green} \line(0, 1){ 2}} \put(4, 0){ \color{green} \line(0, 1){ 4}} 

}

\newcommand{\yn}{  \hskip 5pt  \usebox{\ydn} \vspace{ 11mm} \hskip 45pt   }

\newsavebox{\ydx}
\savebox{\ydx}(1,1){
	
	\put(0, 0){ \color{blue} \line(0, 1){9}}
	\put(3, 0){ \color{blue} \line(0, 1){ 9}} 
	\put(6, 0){ \color{blue} \line(0, 1){ 9}}
	\put(9, 0){ \color{blue} \line(0,1){9}} 
	\put(12, 0){ \color{blue} \line(0, 1){ 9}}
	\put(15, 0){ \color{blue} \line(0, 1){ 9}} 
	\put(0, 0){ \color{blue} \line(1, 0){15}} 
	\put(0, 3){ \color{blue} \line(1, 0){ 15}}
	\put(0, 6){ \color{blue} \line(1, 0){ 15}}
	\put(0, 9){ \color{blue} \line(1, 0){ 15}}
	\put(0, 9){ \color{green} \line(1, 0){ 3}}
	\put(3, 6){ \color{green} \line(1, 0){ 3}}
	\put(3, 6){ \color{green} \line(0, 1){ 3}}
	\put(6, 0){ \color{green} \line(0, 1){ 6}}   
	\put(1.5, 1.3){\footnotesize{$a_1$}}   
	\put(1.1, 4.3){\footnotesize{$a_1 \hbar$}}
	\put(1.0, 7.4){\footnotesize{$a_1 \hbar^2$}}
	\put(4.5, 1.3){\footnotesize{$a_1$}}
	\put(4.1, 4.3){\footnotesize{$a_1 \hbar$}}
	\put(3.9, 7.4){\footnotesize{$a_2 \hbar^4$}}
	
	\put(7, 1.3){\footnotesize{$a_2 \hbar^3$}}
	\put(7, 4.3){\footnotesize{$a_2 \hbar^3$}}
	\put(7, 7.4){\footnotesize{$a_2 \hbar^3$}}
	
	\put(10, 1.3){\footnotesize{$a_2 \hbar^2$}}
	\put(10, 4.3){\footnotesize{$a_2 \hbar^2$}}
	\put(10, 7.4){\footnotesize{$a_2 \hbar^2$}}

	\put(13, 1.3){\footnotesize{$a_2 \hbar$}}
	\put(13, 4.3){\footnotesize{$a_2 \hbar$}}
	\put(13, 7.4){\footnotesize{$a_2 \hbar$}}
	
}

\newcommand{\yx}{  \hskip 5pt  \usebox{\ydx} \vspace{ 11mm} \hskip 45pt   }

\newsavebox{\ydbl}
\savebox{\ydbl}(1,1){
	
	\put(0, 0){ \color{blue} \line(0, 1){ 6}} \put(2, 0){ \color{blue} \line(0, 1){ 6}} \put(4, 0){ \color{blue} \line(0, 1){ 6}} \put(6, 0){ \color{blue} \line(0,1){6}} \put(8, 0){ \color{blue} \line(0, 1){ 6}} \put(10, 0){ \color{blue} \line(0, 1){ 6}} \put(0, 0){ \color{blue} \line(1, 0){10}} \put(0, 2){ \color{blue} \line(1, 0){ 10}} \put(0, 4){ \color{blue} \line(1, 0){ 10}} \put(0, 6){ \color{blue} \line(1, 0){ 10}} \put(0, 6){ \color{green} \line(1, 0){ 2}} \put(2, 4){ \color{green} \line(1, 0){ 2}}\put(2, 4){ \color{green} \line(0, 1){ 2}} \put(4, 0){ \color{green} \line(0, 1){ 4}} 
	\put(1, 1){ \color{red} \line(0, 1){ 4}} 	\put(1, 1){ \color{red} \line(1, 0){ 2}}
	\put(3, 1){ \color{red} \line(0, 1){ 2}} \put(9, 1){ \color{red} \line(0, 1){ 4}}
	\put(3, 5){ \color{red} \line(1, 0){ 6}} \put(7, 1){ \color{red} \line(0, 1){ 4}}
	\put(5, 1){ \color{red} \line(1, 0){ 2}} \put(5, 3){ \color{red} \line(1, 0){ 2}}
}

\newsavebox{\exoneb}
\savebox{\exoneb}(1,1){
		\put(0, 0){ \color{blue} \line(0, 1){ 4} }
		\put(2, 0){ \color{blue} \line(0, 1){ 4} } 
		
		\put(4, 0){ \color{blue} \line(0, 1){ 4} }
		\put(6, 0){ \color{blue} \line(0, 1){ 4} }
		 \put(0, 0){ \color{blue} \line(1, 0){ 6} } \put(0, 2){ \color{blue} \line(1, 0){ 6} } \put(0, 4){ \color{blue} \line(1, 0){ 6} }
	\put(1, 1){ \color{red} \line(0, 1){ 2}}
	\put(5, 1){ \color{red} \line(0, 1){ 2}}
	\put(1, 1){ \color{red} \line(1, 0){ 2}}
	\put(1, 3){ \color{red} \line(1, 0){ 2}} 
	}
\newcommand{\exone}{  \hskip -6pt  \usebox{\exoneb} \vspace{11mm} \hskip 35pt   }

\newsavebox{\exonel}
\savebox{\exonel}(1,1){
	\put(0, 0){ \color{blue} \line(0, 1){ 4} }
	\put(2, 0){ \color{blue} \line(0, 1){ 4} } 
	
	\put(4, 0){ \color{blue} \line(0, 1){ 4} }
	\put(6, 0){ \color{blue} \line(0, 1){ 4} }
	\put(0, 0){ \color{blue} \line(1, 0){ 6} } \put(0, 2){ \color{blue} \line(1, 0){ 6} } \put(0, 4){ \color{blue} \line(1, 0){ 6} }
	\put(1, 1){ \color{red} \line(0, 1){ 2}}
	\put(1, 1){ \color{red} \line(1, 0){ 2}}
	\put(1, 3){ \color{red} \line(1, 0){ 2}} 
}
\newcommand{\exonl}{  \hskip -6pt  \usebox{\exonel} \vspace{11mm} \hskip 35pt   }

\newsavebox{\exonea}
\savebox{\exonea}(1,1){
	\put(0, 0){ \color{blue} \line(0, 1){ 2} }
	\put(2, 0){ \color{blue} \line(0, 1){ 2} } 
	\put(4, 0){ \color{blue} \line(0, 1){ 2} }
 
	\put(8, 0){ \color{blue} \line(0, 1){ 2} }
	\put(10, 0){ \color{blue} \line(0, 1){ 2} } 
	\put(12, 0){ \color{blue} \line(0, 1){ 2} } 
	\put(14, 0){ \color{blue} \line(0, 1){ 2} }

	\put(4, 0){ \color{blue} \line(0, 1){ 2} }
	\put(6, 0){ \color{green} \line(0, 1){ 2} }

	\put(0, 0){ \color{blue} \line(1, 0){ 14} } 
	\put(0, 2){ \color{blue} \line(1, 0){ 14} } 
	\put(0.8, 2.3){ \footnotesize{1} }
	\put(2.5, 2.3){ \footnotesize{...} }
	\put(4, 2.3){ \footnotesize{m-1}}
	\put(6.9, 2.3){ \footnotesize{m}}
	\put(9.9, 2.3){ \footnotesize{.....}}
	\put(12.9, 2.3){ \footnotesize{n-1}}

\put(1, 1){ \color{red} \line(1, 0){ 4}}

\put(7, 1){ \color{red} \line(1, 0){ 6}}
 
}
\newcommand{\exona}{  \hskip -6pt  \usebox{\exonea} \vspace{11mm} \hskip 35pt   }

\newsavebox{\exoneab}
\savebox{\exoneab}(1,1){
	\put(0, 0){ \color{blue} \line(0, 1){ 2} }
	\put(2, 0){ \color{blue} \line(0, 1){ 2} } 
	\put(4, 0){ \color{blue} \line(0, 1){ 2} }
	
	\put(8, 0){ \color{blue} \line(0, 1){ 2} }
	\put(10, 0){ \color{blue} \line(0, 1){ 2} } 
	\put(12, 0){ \color{blue} \line(0, 1){ 2} } 
	\put(14, 0){ \color{blue} \line(0, 1){ 2} }

	\put(4, 0){ \color{blue} \line(0, 1){ 2} }
	\put(6, 0){ \color{green} \line(0, 1){ 2} }

\put(0, 0){ \color{blue} \line(1, 0){ 14} } 
\put(0, 2){ \color{blue} \line(1, 0){ 14} }	
	
	\put(1, 1){ \color{red} \line(1, 0){ 4}}
	
	\put(7, 1){ \color{red} \line(1, 0){ 6}}
	
}
\newcommand{\exonab}{  \hskip -6pt  \usebox{\exoneab} \vspace{11mm} \hskip 35pt   }

\newsavebox{\exoneabl}
\savebox{\exoneabl}(1,1){
	\put(0, 0){ \color{blue} \line(0, 1){ 2} }
	\put(2, 0){ \color{blue} \line(0, 1){ 2} } 
	\put(4, 0){ \color{blue} \line(0, 1){ 2} }
	
	\put(8, 0){ \color{blue} \line(0, 1){ 2} }
	\put(10, 0){ \color{blue} \line(0, 1){ 2} } 
	\put(12, 0){ \color{blue} \line(0, 1){ 2} } 
	\put(14, 0){ \color{blue} \line(0, 1){ 2} }

	\put(4, 0){ \color{blue} \line(0, 1){ 2} }
	\put(6, 0){ \color{green} \line(0, 1){ 2} }

	\put(0, 0){ \color{blue} \line(1, 0){ 14} } 
	\put(0, 2){ \color{blue} \line(1, 0){ 14} }	
	
	\put(1, 1){ \color{red} \line(1, 0){ 4}}

}
\newcommand{\exonabl}{  \hskip -6pt  \usebox{\exoneabl} \vspace{11mm} \hskip 35pt   }

\newsavebox{\exoneabr}
\savebox{\exoneabr}(1,1){
	\put(0, 0){ \color{blue} \line(0, 1){ 2} }
	\put(2, 0){ \color{blue} \line(0, 1){ 2} } 
	\put(4, 0){ \color{blue} \line(0, 1){ 2} }
	
	\put(8, 0){ \color{blue} \line(0, 1){ 2} }
	\put(10, 0){ \color{blue} \line(0, 1){ 2} } 
	\put(12, 0){ \color{blue} \line(0, 1){ 2} } 
	\put(14, 0){ \color{blue} \line(0, 1){ 2} }

	\put(4, 0){ \color{blue} \line(0, 1){ 2} }
	\put(6, 0){ \color{green} \line(0, 1){ 2} }

	\put(0, 0){ \color{blue} \line(1, 0){ 14} } 
	\put(0, 2){ \color{blue} \line(1, 0){ 14} }	

	\put(7, 1){ \color{red} \line(1, 0){ 6}}
	
}
\newcommand{\exonabr}{  \hskip -6pt  \usebox{\exoneabr} \vspace{11mm} \hskip 35pt   }

\newsavebox{\exoner}
\savebox{\exoner}(1,1){
	\put(0, 0){ \color{blue} \line(0, 1){ 4} }
	\put(2, 0){ \color{blue} \line(0, 1){ 4} } 
	
	\put(4, 0){ \color{blue} \line(0, 1){ 4} }
	\put(6, 0){ \color{blue} \line(0, 1){ 4} }
	\put(0, 0){ \color{blue} \line(1, 0){ 6} } \put(0, 2){ \color{blue} \line(1, 0){ 6} } \put(0, 4){ \color{blue} \line(1, 0){ 6} }
		\put(5, 1){ \color{red} \line(0, 1){ 2}}
	
}
\newcommand{\exonr}{  \hskip -6pt  \usebox{\exoner} \vspace{11mm} \hskip 35pt   }

\newsavebox{\extwob}
\savebox{\extwob}(1,1){
		\put(0, 0){ \color{blue} \line(0, 1){ 4} } \put(2, 0){ \color{blue} \line(0, 1){ 4} } \put(4, 0){ \color{blue} \line(0, 1){ 4} } \put(0, 0){ \color{blue} \line(1, 0){ 4} } \put(0, 2){ \color{blue} \line(1, 0){ 4} } \put(0, 4){ \color{blue} \line(1, 0){ 4} }
	\put(1, 1){ \color{red} \line(0, 1){ 2}}
	\put(3, 1){ \color{red} \line(0, 1){ 2}}
	\put(1, 3){ \color{red} \line(1, 0){ 2}}}

\newsavebox{\rgb}

\savebox{\rgb}(1,1)
{   \put(4.5,2){\color{blue}\rule{16pt}{16pt}} 
	\put(6.5,4){\color{green}\rule{16pt}{16pt}}
	\put(8.5,6){\color{green}\rule{16pt}{16pt}}
	\put(10.5,6){\color{red}\rule{16pt}{16pt}}
	\put(8.5,4){\color{red}\rule{16pt}{16pt}}
	\put(6.5,2){\color{red}\rule{16pt}{16pt}}
	\put(2.5,0){\color{red}\rule{16pt}{16pt}}
	\put(0.5,0){\color{green}\rule{16pt}{16pt}}
	\put(2.5,2){\color{green}\rule{16pt}{16pt}}
	\put(0, 0) { \color{blue} \line(0, 1) { 12} } \put(2, 0) { \color{blue} \line(0, 1) { 12} } \put(4, 0) { \color{blue} \line(0, 1) { 10} } \put(6, 0) { \color{blue} \line(0, 1) { 10} } \put(8, 0) { \color{blue} \line(0, 1) { 10} } \put(10, 0) { \color{blue} \line(0, 1) { 8} } \put(12, 0) { \color{blue} \line(0, 1) { 8} } \put(14, 0) { \color{blue} \line(0, 1) { 4} } \put(0, 0) { \color{blue} \line(1, 0) { 14} } \put(0, 2) { \color{blue} \line(1, 0) { 14} } \put(0, 4) { \color{blue} \line(1, 0) { 14} } \put(0, 6) { \color{blue} \line(1, 0) { 12} } \put(0, 8) { \color{blue} \line(1, 0) { 12} } \put(0, 10) { \color{blue} \line(1, 0) { 8} } \put(0, 12) { \color{blue} \line(1, 0) { 2} }}

\newsavebox{\ppb}

\savebox{\ppb}(1,1)
{   \put(0.5,10){\color{gray}\rule{16pt}{16pt}} 
	\put(2.5,8){\color{gray}\rule{16pt}{16pt}}
	 \put(0.5,8){\color{gray}\rule{16pt}{16pt}}
	 \put(0.5,6){\color{gray}\rule{16pt}{16pt}} 
	\put(0, 0) { \color{blue} \line(0, 1) { 12} } \put(2, 0) { \color{blue} \line(0, 1) { 12} } \put(4, 0) { \color{blue} \line(0, 1) { 10} } \put(6, 0) { \color{blue} \line(0, 1) { 10} } \put(8, 0) { \color{blue} \line(0, 1) { 10} } \put(10, 0) { \color{blue} \line(0, 1) { 8} } \put(12, 0) { \color{blue} \line(0, 1) { 8} } \put(14, 0) { \color{blue} \line(0, 1) { 4} } \put(0, 0) { \color{blue} \line(1, 0) { 14} } \put(0, 2) { \color{blue} \line(1, 0) { 14} } \put(0, 4) { \color{blue} \line(1, 0) { 14} } \put(0, 6) { \color{blue} \line(1, 0) { 12} } \put(0, 8) { \color{blue} \line(1, 0) { 12} } \put(0, 10) { \color{blue} \line(1, 0) { 8} } \put(0, 12) { \color{blue} \line(1, 0) { 2} }}

\begin{document}
\title{3d Mirror Symmetry and Elliptic Stable Envelopes}
\author{Richárd Rimányi, Andrey Smirnov, Alexander Varchenko, Zijun Zhou}
\date{}
\maketitle
\thispagestyle{empty}
	
\begin{abstract}
We consider  a pair of quiver varieties $(X;X\rq{})$ related by 3d mirror symmetry, where
$X =T^*{Gr}(k,n)$ is the cotangent bundle of the Grassmannian of $k$-planes of $n$-dimensional space. We give
formulas  for the elliptic stable envelopes on both sides. We show an existence of an equivariant elliptic
cohomology class on $X \times X\rq$  (the {\it Mother function}) whose restrictions to $X$ and $X\rq{}$ are
the elliptic stable envelopes of those varieties. This implies, that the restriction
matrices of the elliptic stable envelopes for $X$ and $X\rq{}$
are equal after transposition and  identification of the equivariant parameters on one side  with the K\"ahler parameters on the dual side. 
  
\end{abstract}
	
\setcounter{tocdepth}{2}

\setcounter{tocdepth}{2}

\tableofcontents

\section{Introduction}

\subsection{Mirror symmetries}

Mirror symmetry is one of the most important physics structures that enters the world of mathematics and arouses lots of attention in the last several decades. Its general philosophy is that a space $X$ should come with a dual $X'$ which, though usually different from and unrelated to $X$ in the appearance, admits some deep connections with $X$ in geometry. Mirror symmetry in 2 dimensions turns out to be extremely enligntening in the study of algebraic geometry, symplectic geometry, and representation theory. In particular, originated from the 2d topological string theory, the Gromov--Witten theory has an intimate connection with 2d mirror symmetry; for an introduction, see \cite{CK, Hori}. 

Similar types of duality also exists in 3 dimensions. More precisely, as introduced in \cite{BDG, Ga-Wit, HW, BDGH, PhysMir1, PhysMir2, PhysMir3,gla3d}, the 3d mirror symmetry is constructed between certain pairs of 3d $\cN=4$ supersymmetric gauge theories, under which they exchanged their \emph{Higgs branches} and \emph{Coulomb branches}, as well as their FI parameters and mass parameters. In mathematics, the $\cN = 4$ supersymmetries imply that the corresponding geometric object of our interest should admit a hyperK\"ahler structure, or if one prefers to stay in the algebraic context, a holomorphic symplectic structure. In particular, for theories of the class as mentioned above, the Higgs branch, which is a certain branch of its moduli of vacua, can be interpreted as a holomorphic symplectic quotient in mathematics, where the prequotient and group action are determined by the data defining the physics theory. The FI parameters and mass parameters of the theory are interpreted as \emph{K\"ahler parameters} and \emph{equivariant parameters} respectively. 

The Coulomb branch, however, did not have such a clear mathematical construction until recently \cite{Coulomb1, Nak1, BFN2}. In this general setting it is not a holomorphic symplectic quotient, and it is difficult to study its geometry. Nevertheless, in many special cases e.g., already appearing in the physics literature \cite{BDG, GaKor}, the Coulomb branch might also be taken as some holomorphic symplectic quotient. Those special cases include hypertoric varieties, Hilbert schemes of points on $\matC^2$, the moduli space of instantons on the resolved $A_N$ surfaces, and so on. For a mathematical exposition, see \cite{BLPW, matdu}, where 3d mirror symmetry is refered to as \emph{symplectic duality}.

A typical mirror symmetry statement for a space $X$ and its mirror $X'$, is to relate certain geometrically defined invariants on both sides. For example, in the application of 2d mirror symmetry to genus-zero Gromow--Witten theory, the $J$-function counting rational curves in $X$ is related to the $I$-function, which arises from the mirror theory. 

In the 3d case, instead of cohomological counting, one should consider counting in the K-theory. One of the K-theoretic enumerative theory in this setting, which we are particularly interested in, is developed by A. Okounkov and his collaborators \cite{pcmilec, MO2, OS, AOelliptic}. The 3d mirror symmetry statement in this theory looks like
\begin{equation} \label{VV}
V(X) \cong V(X'),
\end{equation}
where $X$ and $X'$ is a 3d mirror pair of hypertoric or Nakajima quiver varieties, and $V(X)$, $V(X')$ are the so-called \emph{vertex functions}, defined via equivariant K-theoretic counting of  quasimaps into $X$ and $X'$ \cite{pcmilec}. 

On both sides, the vertex functions, which depend on K\"ahler parameters $z_i$ and equivariant parameters $a_i$, can be realized as solutions of certain geometrically defined $q$-difference equations. We call those solutions that are holomorphic in K\"ahler parameters and meromorphic in equivariant parameters the \emph{$z$-solutions}, and those in the other way the \emph{$a$-solutions}. In particular, vertex functions are by definition $z$-solutions. 

Under the correspondence (\ref{VV}), the K\"ahler and equivariant parameters on $X$ and $X'$ are exchanged with each other, and hence $z$-solutions of one side should be mapped to $a$-solutions of the other side and vice versa. In particular, for the correspondence to make sense, (\ref{VV}) should involve a transition between a basis of $z$-solutions, and a basis of $a$-solutions. In \cite{AOelliptic}, this transition matrix is introduced geometrically as the \emph{elliptic stable envelope}.

\subsection{Elliptic stable envelopes}

The notion of stable envelopes first appear in \cite{MO} to generate a basis for Nakajima quiver varieties which admits many good properties. Their definition depends on a choice of cocharacter, or equivalently, a chamber in the Lie algebra of the torus that acting on the space $X$. The transition matrices between stable envelopes defined for different chambers turn out to be certain $R$-matrices, satisfying the Yang--Baxter equation and hence defining quantum group structures. The stable envelopes are generalized to K-theory \cite{pcmilec, MO2, OS}, where they not only depend on the choice of cocharacter $\sigma$, but also depend piecewise linearly on the choice of slope $s$, which lives in the space of K\"ahler parameters. 

In \cite{AOelliptic}, stable envelopes are further generalized to the equivariant elliptic cohomology, where the piecewise linear dependence on the slope $s$ is replaced by the meromorphic dependence to a K\"ahler parameter $z$. In particular, the elliptic version of the stable envelope is the most general structure, K-theoretic and cohomological 
stable envelopes can be considered as limits of elliptic.  The elliptic stable envelopes depends on both, the equivariant and K\"ahler parameters which makes it a natural object for the study of 3d mirror symmetry.

In this paper, we will concentrate on a special case where $X = T^* Gr(k, n)$, the cotangent bundle of the Grassmannian of $k$-dimensional subspaces in $\matC^n$.  This variety is a simplest example of Nakajima quiver variety associated to the $A_1$-quiver, with dimension vector $\textsf{v} = k$ and framing vector $\textsf{w} = n$. We will always assume that $n\geq 2k$ \footnote{Only in the case $n\geq 2k$ the dual variety $X'$ can also be realized as quiver variety.}. Its mirror, which we denote by $X'$, can also be constructed as a Nakajima quiver variety, associated to the $A_{n-1}$-quiver. It has dimension vector 
$$
\textsf{v}=(1,2,\dots,k-1,\underbrace{k,\dots,k}_{(n-2 k+1) \text{-times}},k-1,\dots,2,1)
$$
and framing vector
$$
\textsf{w}_i = \delta_{i,k} + \delta_{i,n-k}. 
$$

For Nakajima quiver varieties, there is always a torus action induced by that on the framing spaces. Let $\bT$ and $\bT'$ be the tori on $X$ and $X'$ respectively. They both have $n! / (k! (n-k)!)$ fixed points, which admit nice combinatorial descriptions. Elements in $X^\bT$ can be interpreted as $k$-subsets ${\bf p} \subset {\bf n} := \{1, 2, \cdots, n\}$, while $(X')^{\bT'}$ is the set of Young diagrams $\lambda$ that fit into a $k \times (n-k)$ rectangle. There is a natural bijection (\ref{pointbij}) between those fixed points
$$
\textsf{bj}:(X')^{\bT'} \xrightarrow{\sim} X^{\bT}. 
$$

We will consider the extended equivariant elliptic cohomology of $X$ and $X'$ under the corresponding framing torus actions, denoted by $\textsf{E}_{\bT}(X)$ and $\textsf{E}_{\bT'}(X')$ respectively. By definition, they are certain schemes, associated with structure maps which are finite (and hence affine)
$$
\textsf{E}_{\bT}(X) \to \cE_\bT \times \cE_{\textrm{Pic}_{\bT}(X)}, \qquad \textsf{E}_{\bT'}(X') \to \cE_{\bT'} \times \cE_{\textrm{Pic}_{\bT'}(X')},
$$
where $\cE_\bT \times \cE_{\textrm{Pic}_{\bT}(X)}$ and $\cE_{\bT'} \times \cE_{\textrm{Pic}_{\bT'}(X')}$ are powers of an elliptic curve $E$, the coordinates on which are the K\"aher and equivariant parameters. There is a natural identification (\ref{parident})
$$
\kappa: \bK \to \bT', \qquad \bT \to \bK'
$$
between the K\"ahler and equivariant tori of the two sides, which induces an isomorphism between $\cE_\bT \times \cE_{\textrm{Pic}_{\bT}(X)}$ and $\cE_{\bT'} \times \cE_{\textrm{Pic}_{\bT'}(X')}$. 

By localization theorems, the equivariant elliptic cohomology of $X$ has the form
$$
\textsf{E}_{\bT}(X)=\Big(\coprod\limits_{p\in X^{\bT}}\, \widehat{\Or}_{p}\Big) /\Delta,
$$
where each $\widehat\Or_{\bf p}$ is isomorphic to the base $\cE_\bT \times \cE_{\textrm{Pic}_{\bT}(X)}$. The $\bT$-action on $X$ is good enough, in the sense that it is of the \emph{GKM type}, which means that it admits finitely many isolated fixed points, and finitely many 1-dimensional orbits. Due to this GKM property, the gluing data $\Delta$ of $X$ is easy to describe: it is simply the gluing of $\widehat\Or_{\bf p}$ and $\widehat\Or_{\bf q}$ for those fixed points ${\bf p}$ and ${\bf q}$ connected by 1-dimesional $\bT$-orbits. For $X'$, $\textsf{E}_{\bT'}(X')$ also has the form as above; however, the gluing data $\Delta'$ is more complicated. 

By definition, the elliptic stable envelope $\Stab_{\sigma} ({\bf p})$ for a given fixed point ${\bf p} \in X^\bT$ is the section of a certain line bundle ${\cal T} ({\bf p})$. We will describe this section in terms of its components 
$$
T_{\bf p, q} := \left. \Stab_\sigma ({\bf p}) \right|_{\widehat\Or_{\bf q}},
$$
which are written explicitly in terms of theta functions and satisfy prescribed quasiperiodics and compatibility conditions. Similar for $X'$, we will describe the components 
$$
T'_{\lambda, \mu} := \left. \Stab'_{\sigma'} (\lambda) \right|_\mu. 
$$

\subsection{Coincidence of stable envelopes for dual variates}

Our main result is that the restriction matrices for elliptic stable envelopes on the dual varieties coincide (up to transposition and normalization by the diagonal elements):
\begin{Corollary} \label{corth1}
	{ \it Restriction matrices of elliptic stable envelopes for $X$ and $X'$ are related by:
		\be \label{coinc1}
		T_{\bf p,p} T'_{\lambda, \mu} = T'_{\mu,\mu} T_{\bf q,p}
		\ee	
		where ${\bf p}=\mathsf{bj}(\lambda)$, ${\bf q}=\mathsf{bj}(\mu)$ and parameters are identified by (\ref{parident}). }
\end{Corollary}

In (\ref{coinc1}), the prefactors $T_{\bf p,p}$ and $T'_{\mu, \mu}$ have very simple expressions as product of theta functions. The explicit formula for matrix elements $T'_{\lambda, \mu}$ and $T_{\bf q,p}$, however, involves complicated summations. 

Explicit formulas (see Theorem \ref{ellxth} and \ref{mainth}) for elliptic stable envelopes are obtained by the \emph{abelianzation technique} \cite{Shenf,GenJacks,AOelliptic,EllHilb}. In the spirit of abelianization, the formula for $T_{\bf q,p}$ involves a symmetrization sum over the symmetric group $\fS_k$, the Weyl group of the gauge group $GL(k)$. However, the formula $T'_{\lambda, \mu}$ involves not only a symmetrization over $\fS_{n,k}$, the Weyl group of the corresponding gauge group, but also a \emph{sum over trees}. Similar phenomenon already appear in the abelianization formula for the elliptic stable envelopes of $\mathrm{Hilb} (\matC^2)$ \cite{EllHilb}. The reason for this sum over trees to occur is that in the abelianization for $X'$, the preimage of a point is no longer a point, as in the case of $X$. 

As a result, the correspondence (\ref{coinc}) we obtained here actually generates an infinite family of nontrivial identities among product of theta functions. See Section \ref{proofab} and \ref{proofnab} for examples in the simplest cases $k=1$ and $n=4, k=2$. In particular, in the $n=4, k=2$ case, we obtain the well-known 4-term theta identity.

Motivated by the correspondence (\ref{coinc}) and the Fourier--Mukai philosophy, a natural guess is that the identity might actually come from a universal \emph{``Mother function"} $\frak{m}$, living on the product $X \times X'$. Consider the following diagram of embeddings
$$
X=X\times\{\lambda\} \stackrel{i_{\lambda}}{ \longrightarrow} X\times X' 
\stackrel{i_{{\bf p}}}{ \longleftarrow} \{{\bf p}\}\times X' =X'. 
$$
Corollary \ref{corth1} then follows directly from our main theorem:
\begin{Theorem} \label{mothfun1}
	{\it 	
		There exists a holomorphic section ${\frak{m}}$ (\underline{the Mother function}) of a line bundle ${\frak{M}}$ on the~$\bT\times \bT^{'}$equivariant elliptic cohomology of $X\times X'$  such that
		$$
		i^{*}_{\lambda}({\frak{m}})={\bf Stab}({\bf p}), \qquad i^{*}_{{\bf p}}({\frak{m}})={\bf Stab}' (\lambda),
		$$
		where ${\bf p} = \mathsf{bj} (\lambda)$.}
\end{Theorem}

The existence of the Mother function was already predicted by Aganagic and Okounkov in the original paper \cite{AOelliptic}. This paper originated from our attempt to check their conjecture and construct the mother function for the simplest examples of dual quiver varieties.

\subsection{Relation to ($\frak{gl}_n$, $\frak{gl}_m$)-duality}
The 3D-mirror symmetry for $A$-type quiver varieties is closely related with the so called $(\frak{gl}_n$, $\frak{gl}_m$)-duality in representation theory. For the case of $X$, which is $A_{1}$-quiver variety and $X'$ which is $A_{n-1}$ quiver variety, we are dealing with a particular example of ($\frak{gl}_n$, $\frak{gl}_2$)-duality (i.e., $m=2$).

Let $\matC^{2}(u)$ be the fundamental evaluation module with evaluation parameter $u$ of the quantum affine algebra $\mathcal{U}_\hbar(\widehat{\frak{gl}}_2)$. Similarly, let $\bigwedge^{\!k}_{\hbar}\matC^{n}(a)$ be the $k$-th fundamental evaluation module with the evaluation parameter $a$ of quantum affine algebra 
$\mathcal{U}_\hbar(\widehat{\frak{gl}}_n)$. Recall that the equivariant K-theory of quiver varieties are naturally equipped with an action of quantum affine algebras \cite{Nakfd}. In particular, for $X=T^{*}Gr(k,n)$ we have isomorphism of weight subspaces in $\mathcal{U}_\hbar(\widehat{\frak{gl}}_2)$-modules:
\be \label{sl2mods}
K_{\bT}(X)\cong \textrm{weight $k$ subspace in}\ \ \matC^2(u_1)\otimes \cdots \otimes \matC^2(u_n)
\ee
In geometry, the evaluation parameters $u_i$ are identified with equivariant parameters of torus $\bT$. Similarly, the dual variety $X^{'}$ is related to representation theory of $\mathcal{U}_\hbar(\widehat{\frak{gl}}_n)$:
\be \label{slnmods}
K_{\bT^{'}}(X')\subset \bigwedge_{\hbar}\nolimits^{\!\!k} \matC^{n}(a_1) \otimes \bigwedge_{\hbar}\nolimits^{\!\!n-k} \matC^{n}(a_2)
\ee
the corresponding weight subspace is spanned by the following vectors
$$
K_{\bT^{'}}(X')= Span \{ (e_{i_1}\wedge\dots \wedge e_{i_k}) \otimes (e_{j_1}\wedge\dots \wedge e_{j_{n-k}}), \{i_1,\dots i_k,j_1,\dots,j_{n-k}\}={\bf{n}}\},
$$
where $e_i$ is the canonical basis in $\matC^n$. So that both spaces have dimension $n!/(k! (n-k)!)$.

Let us recall that  the elliptic stable envelopes features in the representation theory as a building block for solutions of quantum Knizhnik-Zamilodchikov equations and quantum dynamical equations associated to affine quantum groups \cite{EV}.
The integral solutions of these equations have the form \cite{AOBethe,konno2,Push1,Push2}:
$$
\Psi_{p,q} \sim \int\limits_{C} \prod\limits_{i} dx_i\,  \Phi_p (x_1,\dots,x_n) \bStab_q (x_1,\dots,x_n)
$$
Here $\Psi_{p,q}$ represents the matrix of fundamental solution of these equations in some basis. The functions $\Phi_p (x_1,\dots,x_n)$ 
are the so called master functions 
and $\bStab_q (x_1,\dots,x_n)$ denotes the elliptic stable envelope of the fixed point (elliptic weight function). The variables of integration $x_i$ correspond to the  Chern roots of tautological bundles. 

Theorem \ref{mothfun1} implies, in particular, that 3D mirror symmetry for the pair $(X, X^{'})$ identifies $\mathcal{U}_\hbar(\widehat{\frak{gl}}_2)$ 
solutions in (\ref{sl2mods}) with $\mathcal{U}_\hbar(\widehat{\frak{gl}}_n)$ solutions in (\ref{slnmods}). Under this identification the evaluation parameters
turn to dynamical parameters of the dual side, so that the quantum Knizhnik-Zamolodchikov equations and dynamical equations change their roles. This way, our results suggest a new geometric explanation of 
($\frak{gl}_n$, $\frak{gl}_m$) and bispectral  dualities \cite{bisp,zotovzen,toled}.

\subsection{Further progress}
In this final section we would like to overview recent progress in the study of $3D$-mirror symmetry and elliptic stable envelopes made since the first release of this paper. 

In his last two papers \cite{Ok1,Ok2} A. Okounkov proves that the elliptic stable envelopes exist for very general examples of symplectic varieties, 
improving the results of the original paper \cite{AOelliptic} dealing only with quiver varieties. Applications of the elliptic stable envelopes to problems in enumerative geometry, such us constructing integral solutions of the quantum differential and difference equations, description of monodromies of these equations, etc., are the central topics of these papers.    

In particular, an interesting class of varieties for which the stable envelopes exists (by \cite{Ok1}) is given by the Cherkis-Nakajima-Takayama bow varieties \cite{NakajimaBows}. Unlike quiver varieties, the bow varieties are closed under $3D$-mirror symmetry, i.e., $3D$-mirror of a bow variety is a bow variety again. For instance, the mirror $X'$ for $X=T^{*}Gr(k,n)$ is a bow variety for every value $0\leq k \leq n$ (note that $X'$ is a bow variety but not a quiver variety if $n<2k$, which is why we consider only the ``quiver'' case $n\geq 2k$ in this paper).  
It is thus very natural to study the elliptic stable envelope classes  and the corresponding Mother functions for the bow varieties. This investigation is currently pursued in \cite{RS}. 

The results obtained in our paper were further generalized to the case of $X$ given by the cotangent bundles over complete flag varieties of type  $A_n$ in \cite{RSVZ2}. This result is further generalized to flag varieties of arbitrary type in \cite{RW}.  In \cite{ZS}  Theorem \ref{mothfun1} was proved for the hypertoric varieties, see also \cite{Zijtor} for the toric case. In particular, the Mother function for the hypertoric varieties can be written very explicitly, see  Theorem 6.4 in \cite{ZS}. The categorical generalization of Theorem \ref{mothfun1} for hypertoric varieties is recently proposed in  \cite{McBSheshYau}. In this case the elliptic cohomology of $X$ is substituted by the category of coherent sheaves on the spaces of loops in $X$ and  $\frak{m}$ is substituted by the kernel of a Fourier-Mukai transform describing the mirror symmetry. This leads to a possible categorification of the elliptic stable envelopes.

Alternative proofs of our results, based on analysis of the vertex functions and $q$-difference equations was given by H. Dinkins \cite{Dink5,Dink4}. Applications of $3D$-mirror symmetry in enumerative geometry of threefolds were also considered in \cite{HL}.   An approach to $3D$-mirror symmetry based on the theory of quantum opers is investigated by Koroteev-Zeitlin see \cite{KZop} for the current progress. 

The $3D$-mirror symmetry for the K-theoretic stable envelope (which are limits of the elliptic ones) is investigated in the ongoing project \cite{KS1,KS2}. We expect that this work results in new geometric theory of the quantum differential and quantum difference equations associated with symplectic varieties.

\section*{Acknowledgements}
First and foremost we are grateful to M.Aganagic and A.Okounkov for sharing their ideas with us. During the 2018 MSRI program ``Enumerative Geometry Beyond Numbers'' the authors learned from  A.Okounkov about his explicit formula for the Mother function in the hypertoric case. This very formula triggered our curiosity and encouraged us to look for non-abelian examples of these functions.  

The work of R.Rimányi is supported by the Simons Foundation grant 523882.
The work of A.Smirnov is supported by RFBR grant 18-01-00926 and by AMS travel grant. The work of A.Varchenko is supported by
NSF grant DMS-1665239. The work of Z.Zhou is supported by FRG grant 1564500.

\section{Overview of equivariant elliptic cohomology}
We start with a pedestrian exposition of the equivariant elliptic cohomology. For more detailed discussions we refer to \cite{ell1,ell2,ell3,ell4,ell5,ell6}. 

\subsection{Elliptic cohomology functor}

Let $X$ be a smooth variety endowed with an action of torus $\bT\cong (\matC^{\times})^r$. We say $X$ is a $\bT$-variety. Recall that taking spectrums of the equivariant cohomolory and K-theory, $\Spec H^*_\bT(X)$ can be viewed as an affine scheme over the Lie algebra of the torus $\Spec H^*_\bT ( \pt ) \cong \matC^r$, and $\Spec K_\bT(X)$ is an affine scheme over the algebraic torus $\Spec K_\bT (\pt) \cong  (\matC^\times)^r$. Equivariant elliptic cohomology is an elliptic analogue of this viewpoint. 

Let us fix an elliptic curve
$$
E=\matC^{\times}/q^{\matZ},
$$
i.e., fix the modular parameter $q$. The equivariant elliptic cohomology is a covariant functor:
$$
\textrm{Ell}_{\bT}: \{ \bT \textrm{-varieties} \} \rightarrow \{ \textrm{schemes}\},
$$
which assigns to a $\bT$-variety $X$ certain scheme $\textrm{Ell}_{\bT}(X)$.
For example, the equivariant elliptic cohomology of a point is
$$
\textrm{Ell}_{\bT}(\pt)=\bT/q^{\textrm{cochar}(\bT)} \cong E^{\dim(\bT)}.
$$ 
We denote this abelian variety by $\cE_{\bT}:=\textrm{Ell}_{\bT}(\pt)$.
We will refer to the coordinates on $\cE_{\bT}$ (same as coordinates on $\bT$) as  {\it equivariant parameters}.

Let $\pi: X\to \pt$ be the canonical projection to a point. The functoriality of the elliptic cohomology provides the map $\pi^{*}: \textrm{Ell}_{\bT}(X)\rightarrow \cE_{\bT}$. For each point $t\in \cE_\bT$, we take a small anallytic neighborhoods $U_t$, which is isomorphic via the exponential map to a small analytic neighborhood in $\matC^r$. Consider the sheaf of algebras
$$
\mathscr{H}_{U_t} := H^\bullet_\bT (X^{\bT_t}) \otimes_{H^\bullet_\bT (\pt)} \cO^{\mathrm{an}}_{U_t},
$$
where
$$
T_t := \bigcap_{\substack{\chi \in \mathrm{char} (\bT), \chi(t) = 0 }} \ker \chi \subset \bT. 
$$
Those algebras glue to a sheaf $\mathscr{H}$ over $\cE_\bT$, and we define $\textrm{Ell}_\bT (X) := \Spec_{\cE_\bT} \mathscr{H}$. The fiber of $\textrm{Ell}_\bT (X)$ over $t$ is obtained by setting local coordinates to $0$, as described in the following diagram \cite{AOelliptic}:

\be
\label{abdiag}
\ee
\vspace{-1.3cm}
\[
\xymatrix{
	\Spec H^{\bullet}(X^{\bT_{t}})  \ar@{^{(}->}[r] \ar[d]^{\pi^* } & \Spec H_\bT^{\bullet}(X^{\bT_{t}}) \ar[d]  & (\pi^*)^{-1} (U_t) \ar[l] \ar[r] \ar[d] &  \ar[d]^{\pi^* } \textrm{Ell}_{\bT}(X)  \\
	\{t\}   
	\ar@{^{(}->}[r]^{} & \matC^r  & U_t \ar[l] \ar[r] & 
	\cE_{\bT}. }
\]
This diagram describes a structure of the scheme $\textrm{Ell}_\bT (X)$
and gives one of several definitions of elliptic cohomology.

\begin{Example} \label{exam1}
	Let us consider a two-dimensional vector space $V=\matC^{2}$ with coordinates $(z_1,z_2)$, and a torus $\bT=(\matC^{\times})^2$ acting on it by scaling the coordinates: $(z_1,z_2)\to (u_1 z_1, u_2 z_2)$. Let us set $X=\mathbb{P}(V)$. The action of $\bT$ on $V$ induces a structure of  $\bT$-space on $X$. We have $\cE_{\bT}=E\times E$ and the equivariant parameters $u_1$ and $u_2$ represent the coordinates on the first and the second factor. Note that for a generic point $t=(u_1,u_2)\in \cE_{\bT}$ the fixed set $X^{\bT_t}$ consists of two points, which in homogeneous coordinates of $\mathbb{P}(V)$ are:
	$$
	p=[1:0], \qquad q=[0:1]. 
	$$  
	The stalk of $\mathscr{H}$ at $t$ is $H_\bT^\bullet (p \cup q) \otimes_{H_\bT^\bullet} \cO_{\cE_{\bT}, t}$, and the fiber is $H^{\bullet}(p \cup q)$. We conclude, that over a general point $t \in \cE_{\bT}$ the fiber of $\pi^*$ in (\ref{abdiag}) consists of two points. 
	
	At the points $t=(u_1,u_2)$ with $u_1=u_2$ the torus $\bT_{t}$ acts trivially on $X$, thus locally the sheaf $\mathscr{H}$ looks like
	$$
	H^\bullet_\bT (X^{\bT_t}) = H^\bullet_\bT (\mathbb{P}^1) = \matC [\delta u_1, \delta u_2, z] / (z - \delta u_1) (z - \delta u_2),
	$$
	where $\delta u_1$, $\delta u_2$ are local coordinates centered at $x$. Taking $\Spec$, this is the gluing of two copies of $\matC^2$ along the diagonal. Overall we obtain that
	$$
	\textrm{Ell}_{\bT}(X)= \Or_{p} \cup_\Delta \Or_{q} ,
	$$
	where $\Or_{p}\cong \Or_{q}\cong \cE_{\bT}$, and $\Delta$ denotes the gluing of these two abelian varieties along the diagonal
	$$
	\Delta=\{(u_1,u_2) \mid u_1=u_2\} \subset \cE_{\bT}.
	$$
\end{Example}

\subsection{GKM varieties}
We assume further that the set of fixed points $X^{\bT}$ is a finite set of isolated points. We will only encounter varieties of this type in our paper. In this case, for a generic one-parametric subgroup $\bT_{t}\subset \bT$ we have
$$
X^{\bT_{t}}=X^{\bT}. 
$$  
By the localization theorem, we know that the irreducible components of $\textrm{Ell}_\bT(X)$ are parameterized by fixed points $p\in X^\bT$, and each isomorphic to the base $\cE_\bT$. Therefore, similarly to Example \ref{exam1} we conclude that set-theoretically,  $\textrm{Ell}_{\bT}(X)$ is union of $|X^{\bT}|$ copies of $\cE_{\bT}$:
\be \label{elcoh}
\textrm{Ell}_{\bT}(X)=\Big(\coprod\limits_{p\in X^{\bT}}\, \Or_{p}\Big) /\Delta,
\ee
where $\Or_{p}\cong \cE_{\bT}$ and $/\Delta$ denotes the gluing of these abelian varieties along the subschemes $\Spec H^{\bullet}(X^{\bT_{t}}) $ corresponding to substori $\bT_{t}$ for which the fixed sets $X^{\bT_{t}}$ are larger than $X^{\bT}$. We call $\Or_{p}$ the $\bT$-orbit associated to the fixed point $p$ in $\textrm{Ell}_{\bT}(X)$.

In general, to describe the scheme structure of $\textrm{Ell}_\bT (X)$ in terms of the gluing data (\ref{elcoh}) can be quite involved. There is, however, a special case when it is relatively simple.

\begin{Definition} { \it 
		We say that $\bT$-variety $X$ is a GKM variety if it satisfies the following conditions:	
		\begin{itemize}
			\item $X^{\bT}$ is finite,
			\item for every two fixed points $p,q\in X^{\bT}$ there is no more than one $\bT$-equivariant curve connecting them. 
		\end{itemize}
	}	
\end{Definition}

Note that by definition, a GKM variety contains finitely many $\bT$-equivariant compact curves (i.e., curves starting and ending at fixed points). We note also that all these curves are rational $C\cong \mathbb{P}^1$ because $\bT$-action on $C$  exists only in this case.

For a compact curve $C$ connecting fixed points 
$p$ and $q$, let $\chi_{C} \in \textrm{Char}(\bT)= \textrm{Hom}(\cE_{\bT},E)$ be the character of the tangent space 
$T_p C$. For all points $t$ on the hyperplane $\chi^{\perp}_{C} \subset \cE_{\bT}$ we thus have $p,q,\in C\subset X^{\bT_{t}} $. As in example \ref{exam1}, this means that in (\ref{elcoh})
the $\bT$-orbits $\Or_{p}$ and $\Or_{q}$ in the scheme $\textrm{Ell}_{\bT}(X)$ are glued along the common hyperplane
$$
\Or_{p} \supset \chi^{\perp}_{C} \subset \Or_{q}.
$$
Note that the character of $T_{q} C$ is $-\chi_{C}$ 
so it does not matter which end point we choose as the first. In sum, we have:
\begin{Proposition} \label{progkm}
	{ \it If $X$ is a GKM variety, then}
	$$
	\textrm{Ell}_{\bT}(X)=\Big(\coprod\limits_{p\in X^{\bT}}\, \Or_{p}\Big) /\Delta, 
	$$	
	{\it where $/\Delta$ denotes the intersections of $\bT$-orbits $\Or_{p}$ and $\Or_{q}$ 
		along the hyperplanes}
	$$
	\Or_{p} \supset \chi^{\perp}_{C} \subset \Or_{q}, 
	$$
	{\it for all $p$ and $q$ connected by an equivariant curve $C$ where $\chi_{C}$ is the $\bT$-character of the tangent space $T_{p} C$. 
	The intersections of orbits $\Or_p$ and $\Or_q$ are transversal and hence the scheme $\mathrm{Ell}_\bT (X)$ is a variety with simple normal crossing singularities.  }
\end{Proposition}

\begin{proof}
	Locally around $t\in \cE_\bT$, the stalk of $\mathscr{H}$ is given by $H_\bT^\bullet (X^{\bT_t}) \otimes \cO_{\cE_\bT, t}$. Let $s \in \cE_\bT$ be another point, such that $T_s \supset T_t$. We have by $\bT_s$-equivariant localization,
	\begin{equation} \label{localization}
	H^\bullet_\bT (X^{\bT_t}) \otimes \mathrm{Frac} (H^\bullet_{\bT_s} (\pt) ) \cong H^\bullet_\bT (X^{\bT_s}) \otimes \mathrm{Frac} (H^\bullet_{\bT_s} (\pt) ).
	\end{equation}
	In other words, if $U_t$ and $U_s$ are small analytic neighborhoods around $t$ and $s$, such that $U_s \subset U_t$, then by definition of the elliptic cohomlogy, the restriction map of $\cH$ from $U_t$ to $U_s$ is equivalent to the isomorphism given by the $\bT_s$-equivariant localization. 
	
	By the property of equivariant cohomology of GKM varieties \cite{GKM}, the variety $\Spec H_\bT^\bullet (X^{\bT_t})$ is the union of $\ft_p$'s along hyperplanes $\chi_C^\perp$, where $\ft_p \cong \matC^r$ are Lie algebras of the torus, associated to fixed points. Moreover, the intersection of $\ft_p$'s for $p\in X^\bT$ is transversal. More precisely, we have the exact sequence
	$$
	\xymatrix{
    0 \ar[r] &	H^\bullet_\bT (X^{\bT_t}) \ar[r] &  H^\bullet_\bT (X^\bT) \ar[r] & H^\bullet_\bT (X_1^{\bT_t}, X^\bT),
}
	$$
	where $X_1^{\bT_t}$ is the 1-skeleton of $X^{\bT_t}$ under the $\bT$-action, and the last map is given by $\chi_C$ for all 1-dimensional orbits in $X^{\bT_t}$. 
	
	We see that the exact sequence is compatible with the localization isomorphism (\ref{localization}), which means that analytically, the local descriptions of $\mathrm{Ell}_\bT (X)$ glue over $\cE_\bT$, and globally $\mathrm{Ell}_\bT (X)$ can be described exactly as in the proposition. 
\end{proof}

Here by ``gluing", we mean the pushout in the category of schemes, in the sense of \cite{Schw}. 

The classical examples of GKM varieties include Grassmannians or more generally, partial flag varieties. For non-GKM varieties the structure of subschemes 
$\Spec H^{\bullet}(X^{\bT_{t}})$ and intersection of orbits in (\ref{elcoh}) can be more complicated. In particular, more than two orbits can intersect along the same hyperplane.

\subsection{Extended elliptic cohomology}
We define 
\be
\label{Kah}
\cE_{\textrm{Pic}(X)} :=\textrm{Pic}(X)\otimes_{\matZ} E \cong E^{\,\dim(\textrm{Pic}(X))}.
\ee 
For Nakajima quiver varieties $\textrm{Pic}(X)\cong \matZ^{|Q|}$ and thus $\cE_{\textrm{Pic}(X)}\cong E^{|Q|}$, where $|Q|$ denotes the number of vertices in the quiver. We will refer to the coordinates in this abelian variety as {\it K\"ahler parameters}. We will usually denote the K\"ahler parameters by the symbol $z_i$, $i=1,\dots,|Q|$.

The extended $\bT$-orbits are defined by
$$
\widehat{\Or}_{p}:={\Or}_{p}\times \cE_{\textrm{Pic}(X)} ,
$$
and the extended elliptic cohomology by
$$
\textsf{E}_{\bT}(X):=\textrm{Ell}_{\bT}(X)\times \cE_{\textrm{Pic}(X)}.
$$
In particular, if $X$ is GKM, $\textsf{E}_{\bT}(X)$ is a bouquet of extended orbits:
$$
\textsf{E}_{\bT}(X)=\Big(\coprod\limits_{p\in X^{\bT}}\, \widehat{\Or}_{p}\Big) /\Delta
$$
where $\Delta$ denotes the same gluing of orbits as in (\ref{elcoh}),
i.e., the extended orbits are glued only along the equivariant directions.

\subsection{Line bundles on elliptic cohomology} 

We have the following description of a line bundle on the variety $\textsf{E}_\bT (X)$. 

\begin{Proposition} \label{defel}
	
	Let $X$ be a GKM variety. 
	\indent  {\it 
		\begin{itemize}
			\item A line bundle ${\cal{T}}$ on the scheme $\textsf{E}_{\bT}(X)$ is a collection of line bundles ${\cal{T}}_p$ on extended orbits $\widehat{\Or}_{p}$, $p\in X^{\bT}$, which coincide on the intersections:
			$$
			\left.{\cal{T}}_p\right|_{\widehat{\Or}_{p}\cap \widehat{\Or}_{q}}= \left.{\cal{T}}_q\right|_{\widehat{\Or}_{p}\cap \widehat{\Or}_{q}},
			$$
			
			\item A meromorphic (holomorphic) section $s$ of a line bundle ${\cal{T}}$ is the collection of meromorphic (holomorphic) sections $s_p$ of ${\cal{T}}_p$ which agree on intersections:
			\be \label{eqseq}
			\left. s_p \right|_{\widehat{\Or}_{p}\cap \widehat{\Or}_{q}} =  \left. s_q \right|_{\widehat{\Or}_{p}\cap \widehat{\Or}_{q}}.
			\ee
	\end{itemize} }
\end{Proposition}

Since each orbit $\widehat \Or_p$ is isomorphic to the base $\cE_\bT \times \cE_{\textrm{Pic}(X)}$, each ${\cal{T}}_p$ is isomorphic via the pull-back along $\pi^*$ to a line bundle on the base. In practice, we often use the coordinates on the base to describe ${\cal{T}}_p$'s. 

\begin{Example}
	Characterization of line bundles and sections is more complicated for non-GKM varieties. Let $X = \mathbb{P}^1 \times \mathbb{P}^1$, with homogeneous coordinates $([x:y], [z:w])$, and  let $\matC^*$ acts on it by 
	$$
	t \cdot ([x:y], [z:w]) = ([x:ty], [z:tw]).
	$$
	There are four fixed points, but infinitely many $\matC^*$-invariant curves: the closure of $\{([x:y], [x:\lambda y])\}$ for any $\lambda \in \matC^*$ is a $\matC^*$-invariant curve, connecting the points $([1:0], [1:0])$ and $([0:1], [0:1])$. Locally near the identity $1\in \cE_{\matC^*}$, the elliptic cohomology $\mathrm{Ell}_{\matC^*} (X)$ looks like
	$$
	\Spec H_\bT^\bullet (X) = \Spec \matC [ H_1, H_2, u] / (H_1^2-u^2, H_2^2-u^2) \quad \to \quad \Spec \matC [u],
	$$
	where $u$ is the local coordinate near $1\in \cE_{\matC^*}$. $\Spec H_\bT^\bullet (X)$ consists of 4 lines in $\matC^3$ intersecting at the origin. 
	
	The gluing of 4 affine lines along the origin, as abstract schemes, would no longer be a subscheme in $\matC^3$, and hence not isomorphic to $\Spec H_\bT^\bullet (X)$. To express $\Spec H_\bT^\bullet (X)$ still as a gluing, one has to allow each orbit $\Or_p$ to have certain embedded non-reduced point at the origin. For an example of this type, see \cite{Schw}. 
\end{Example}

\subsection{Theta functions \label{thetasec}}
By Proposition \ref{defel}, to specify a line bundle ${\cal{T}}$ on $\textsf{E}_{\bT}(X)$ one needs to define line bundles ${\cal{T}}_p$
on each orbit $\widehat{\Or}_{p}$. As $\widehat{\Or}_{p}$ is an abelian variety, to fix ${\cal{T}}_p$ it suffices to describe the transformation properties of sections as we go around periods of $\widehat{\Or}_{p}$. 
In other words, to define ${\cal{T}}_p$ one needs to fix quasiperiods $w_i$ of sections
$$
s(x_i q) =  w_i s(x_i),
$$
for all coordinates $x_i$ on $\widehat{\Or}_{p}$, i.e., for all equivariant and K\"ahler parameters. 

The abelian variates $\widehat{\Or}_{p}$ are all some powers of $E$, which implies that sections of a line bundle on $\textsf{E}_{\bT}(X)$ 
can be expressed explicitly through the Jacobi theta function  associated with $E$:
$$
\theta(x):=(q x)_{\infty} (x^{1/2}-x^{-1/2}) (q/x)_{\infty}, \qquad (x)_{\infty}=\prod\limits_{i=0}^{\infty} (1-x q^i ).
$$ 
The elementary transformation properties of this function are:
\be \label{thettrans}
\theta(x q) =-\dfrac{1}{x \sqrt{q}} \theta(x), \qquad \theta(1/x)=-\theta(x).
\ee
We also extend it by linearity and define 
\be \label{thetathom}
\Theta\Big( {\sum}_i\, a_i -{\sum}_j b_j \Big) := \dfrac{\prod_i \theta(a_i)}{\prod_j \, \theta(b_j)}.
\ee
By definition, the elliptic stable envelope associated with a $\bT$-variety $X$
is a section of certain line bundle on $\textsf{E}_{\bT}(X)$ \cite{AOelliptic}. 
Thus, one can use theta-functions to give explicit formulas for stable envelopes, see Theorem~\ref{ellxth} for an example of such formulas. 

It will also be convenient to introduce the following combination:
$$
\phi(x,y)=\dfrac{\theta(x y)}{\theta(x) \theta(y)}.
$$
This function has the following quasiperiods:
$$
\phi(x q,y)=y^{-1} \phi(x q,y), \qquad \phi(x,y q) =x^{-1} \phi(x,y). 
$$
These transformation properties define the so-called Poincaré line bundle 
on the product of dual elliptic curves $E\times E^{\vee}$ with coordinates $x$ and $y$ and $\phi(x,y)$ is a meromorphic section of this bundle.


\section{Elliptic Stable Envelope for $X$ \label{elx}}

In this section, we discuss algebraic variety 
$X=T^{*}Gr(k,n)$ -- the cotangent bundle over the Grassmannian of $k$-dimensional subspaces in an $n$-dimensional complex space.

\subsection{$X$ as a Nakajima quiver variety} 

We consider a Nakajima quiver variety $X$ defined by the $A_1$-quiver, with dimension ${\bf v}=k$ and framing ${\bf w}=n$. Explicitly, this variety has the following construction. Let $R= \Hom(\matC^{k},\matC^{n})$ be a vector space of complex $k\times n$ matrices. There is an obvious action of $GL(k)$ on this space, which extends to a Hamiltonian action on its cotangent bundle: 
$$
T^*R = R\oplus R^{*} \cong \Hom(\matC^{k},\matC^{n})\oplus \Hom(\matC^{n},\matC^{k}),
$$
with the Hamiltonian moment map 
$$
\mu: T^*R \to \frak{gl}(k)^*, \qquad  \mu(\textbf{j}, \textbf{i})=\textbf{i} \textbf{j}. 
$$
Then $X$ is defined as
$$
X := \mu^{-1}(0) \cap \{  \theta \textrm{-semistable points} \}/GL(k),
$$
where $\textbf{j}\in R$ and $\textbf{i}\in R^{*}$ are $n\times k$ and $k\times n$ matrices respectively. There are two choices of stability conditions $\theta<0$ and $\theta>0$. In the first case the semistable points are those pairs $(\textbf{j},\textbf{i})$ with injective $\textbf{j}$:
\be \label{thch}
\{  \theta \textrm{-semistable points} \}=\{(\textbf{j},\textbf{i})\in T^{*}R \mid \textrm{rank}(\textbf{j})=k\},
\ee
In the case $\theta>0$ the semistable points are $(\textbf{j},\textbf{i})$ with $\textbf{i}$ surjective \cite{ginzb}:
$$
\{  \theta \textrm{-semistable points} \}=\{(\textbf{j},\textbf{i})\in T^{*}R \mid \textrm{rank}(\textbf{i})=k\}.
$$
By construction $X$ is a smooth holomorphic symplectic variety.  In this paper, we choose  
$$\theta = (-1) \in \Lie_{\matR}(\bK),$$
where $\bK := U(1)$, as the stability condition defining $X$, in which case it is isomorphic to the cotangent bundle of the Grassmannian of complex $k$-dimensional vector subspaces in an $n$-dimensional space.

\subsection{Torus action on $X$}
Let $\bA = (\matC^{\times})^n$ be a torus acting on $\matC^{n}$  by scaling the coordinates:
\be
\label{corch}
(z_1,\dots, z_n) \quad \mapsto \quad (z_1 u_1^{-1},\dots, z_n u_n^{-1}),
\ee
which induces an action of $\bA$ on $T^*R$. We denote by $\matC^{\times}_{\hbar}$ the torus acting on $T^*R$ by scaling the second component:
$$
(\textbf{j},\textbf{i})\rightarrow (\textbf{j},\textbf{i}\,\hbar^{-1})
$$
We denote the whole torus $\bT=\bA \times \matC^{\times}_{\hbar}$. The action of $\bT$ preserves semistable locus of $\mu^{-1}(0)$ and thus descends to action on $X$. Simple check shows that the action of $\bA$ preserves the symplectic form on $X$, while $\matC^{\times}_{\hbar}$ scales it by $\hbar$.

Note that the action (\ref{corch}) leaves invariant $k$-dimensional 
subspaces spanned by arbitrary $k$ coordinate vectors. Thus, the set of $\bT$-fixed points $X^{\bT}$ consists of $n!/((n-k)! k!)$ points corresponding to $k$-dimensional coordinate subspaces in $\matC^n$. In other words, a fixed point ${\bf p} \in X^{\bT}$ is described by a $k$-subset in the set $\{1,2,\dots,n\}$.

\subsection{$\bT$-equivariant $K$-theory of $X$} 
Let us denote the tautological bundles on $X$ associated to $\matC^{k}$ and 
$\matC^{n}$ by $\tb$ and $\fb$ respectively. The bundle $\fb$ is a topologically trivial rank-$n$ vector bundle, because $\matC^n$ is a trivial representation of $GL(k)$. In contrast, $\tb$ is a nontrivial rank-$k$ subbundle of $\fb$. One can easily see that $\tb$ is the standard tautological bundle of $k$-subspaces on the Grassmannian. We assume that the tautological bundle splits in $K$-theory into
a sum of virtual line bundles,
\be
\label{tbun}
\tb=y_1+\cdots+y_k.
\ee
In other words, $y_i$ denote the Chern roots of $\tb$. The $\bT$-equivariant
$K$-theory of $X$ has the form:
$$
K_{\bT}(X)=\left.\matC[y^{\pm 1}_1,\dots,y^{\pm 1}_k]^{\frak{S}_{k}} \otimes \matC[u_i^{\pm 1},\hbar^{\pm 1}]\right/I
$$
where $\fS_k$ is the symmetric group of $k$ elements, and $I$ denotes the ideal of Laurent polynomials  vanishing at the fixed points, i.e., at (\ref{eval}). For our choice of stability condition,
the matrix $\textbf{j}$ representing a fixed point is of rank $k$, thus if ${\bf p}$ is a fixed point corresponding to the $k$-subset $\{{\bf p}_1,\dots,{\bf p}_k\} \subset \{1,2,\dots,n\}$, then 
\be
\label{eval}
\left. y_i \right|_{\bf p} = u^{-1}_{{\bf p}_i}, \ \ \ i=1,\dots,k.
\ee
This means that if a K-theory class is represented by a Laurent polynomial  
$f(y_i)$ then its restriction to a fixed point is given by the substitution
$\left.f(y_i)\right|_{{\bf p}}=f(u^{-1}_{{\bf p}_i})$. 

We note that for the opposite choice of the stability parameter $\theta>0$
the restriction would take the form $\left.f(y_i)\right|_{{\bf p}}=f(u^{-1}_{{\bf p}_i} \hbar^{-1})$, where the extra factor $\hbar^{-1}$ comes from the action of $\matC^{\times}_{\hbar}$ on the matrix ${\bf i}$, which is of rank $k$ for this choice of stability condition.

\subsection{Tangent and polarization bundles}
The definition of the elliptic stable envelope requires the choice of a  polarization and a chamber \cite{AOelliptic}. The polarization $T^{1/2}X$, as a virtual bundle, is a choice of the half of the tangent space. In other words,
$$
T X = T^{1/2} X + \hbar^{-1} T^{1/2} X^{*}.
$$
We choose the polarization dual to the canonical polarization (which is defined for all Nakajima varieties, see Example 3.3.3. in \cite{MO}):
\be
\label{polargr}
T^{1/2} X= \hbar^{-1} \fb^{*} \otimes \tb-\hbar^{-1} \tb^{*} \otimes  \tb.
\ee
Expressing $TX$ through the Chern roots by (\ref{tbun}) and restricting it to a fixed point ${\bf p}$ by (\ref{eval}), we find the $\bT$-character of the tangent space at ${\bf p}$ equals:
\be \label{tangr}
T_{{\bf p}} X=\sum\limits_{{i\in {\bf p}} \atop {j\in {\bf n}\setminus {\bf p}}}\Big( \dfrac{u_i}{u_j} +\hbar^{-1} \dfrac{u_j}{u_i} \Big),
\ee
where ${\bf p}$ denotes the $k$-subset in ${\bf n}=\{1,\dots,n\}$.

The definition of the stable envelope also requires the choice of a chamber, or equivalently, a cocharactor of the torus $\bA$. We choose $\sigma$ explicitly as
\be
\label{chamgr}
\sigma=(1,2,\dots,n) \in \Lie_{\matR}(\bA).
\ee 
The choice of $\sigma$ fixes the decomposition $T_{\bf p} X=N^{+}_{\bf p}\oplus N^{-}_{\bf p}$, where $N^{\pm}_{\bf p}$ are the subspaces whose $\bA$-characters take positive or negative values on $\sigma$. 
From (\ref{tangr}) we obtain:
\be \label{repat}
N^{-}_{\bf p}=\sum\limits_{{{i\in {\bf p},} \atop {j\in {\bf n}\setminus {\bf p}} ,}\atop {i<j}}\dfrac{u_i}{u_j} + \hbar^{-1} \sum\limits_{{{i\in {\bf p},} \atop {j\in {\bf n}\setminus {\bf p}} ,}\atop {i>j}}\dfrac{u_j}{u_i}, \qquad N^{+}_{\bf p}=\sum\limits_{{{i\in {\bf p},} \atop {j\in {\bf n}\setminus {\bf p}} ,}\atop {i>j}}\dfrac{u_i}{u_j} + \hbar^{-1} \sum\limits_{{{i\in {\bf p},} \atop {j\in {\bf n}\setminus {\bf p}} ,}\atop {i<j}}\dfrac{u_j}{u_i}.
\ee

\subsection{Elliptic cohomology of $X$ \label{elhohX}}
Let us first note that $X$ is a GKM variety. Two fixed points 
${\bf p}$, ${\bf q}$ are connected by an equivariant curve $C$ if and only if the corresponding $k$-subsets differ by one index ${\bf p}={\bf q}\setminus \{i\} \cup \{j\}$. In this case the $\bT$-character of the tangent space equals:
$$
T_{\bf p} C = u_j / u_i.
$$
By Proposition \ref{progkm} we conclude that the extended elliptic cohomology scheme equals:
\be \label{extcoh}
\textsf{E}_{\bT}(X)= \Big(\coprod\limits_{{\bf p} \in X^{\bT}} \, \widehat{\Or}_{{\bf p}} \Big)  /\Delta 
\ee 
with $\widehat{\Or}_{{\bf p}}\cong \cE_{\bT}\times \cE_{\textrm{Pic}(X)}$ and $/\Delta$ denotes gluing of abelian varieties $\widehat{\Or}_{{\bf p}}$ and $\widehat{\Or}_{{\bf q}}$ with 
${\bf p}={\bf q}\setminus \{i\} \cup \{j\}$ along the hyperplanes $u_i=u_j$. 

By definition, the elliptic stable envelope of a fixed point ${\bf p}$ is a section of the twisted Thom class of the polarization: 
\be \label{tlinebun}
{\cal{T}}({\bf p})  = \Theta(T^{1/2} X) \otimes \dots
\ee
which is a line bundle over scheme (\ref{extcoh}). Here $\Theta: K_{\bT}(X) \to \Pic(\textsf{E}_{\bT}(X))$ is the elliptic Thom class and $\dots$ denotes the twist by a certain explicit section, which depends on  ${\bf p}$. We refer to  Sections 2.5-2.8 of \cite{AOelliptic} for the details of this construction.

Sections of $\left.{\cal{T}}({\bf p})\right|_{\widehat{\Or}_{{\bf q}}}$ transform as the following explicit function 
\footnote{The variables $z_{u_i}$ correspond to K\"ahler variables of 
the $\bT$-{\it equivariant} Picard group. One checks that all quasiperiods of line bundles in these directions are trivial and the elliptic stable envelopes are actually independent on these variables, see discussion in Section 
3.3.7 of \cite{AOelliptic}. 
It is, however, convenient to keep these directions to describe shifts of stable envelopes by the index.}:
\be \label{univ}
{\mathcal{U}}_{{\bf p,q}}(X)=\Theta \Big(\left.T^{1/2} X\right|_{\bf q}\Big) \prod\limits_{i=1}^{k} \dfrac{\phi(u_{{\bf p}_i}^{-1},z^{-1})}{\phi(u_{{\bf q}_i}^{-1},z^{-1})}\prod\limits_{i=1}^{n}\,
\dfrac{\phi(u_i, z^{-1}_{u_i} \hbar^{D^{\bf p}_{i}  })}{
	\phi(u_i, z^{-1}_{u_i})}.
\ee
Here $\Theta\left(\left.T^{1/2} X\right|_{\bf q}\right)$ for the Laurent polynomial $\left.T^{1/2} X\right|_{\bf q}$ is given by a product of theta functions via (\ref{thetathom}). It has the same transformation properties as the elliptic Thom class  $\left.\Theta\Big(T^{1/2} X\Big)\right|_{\widehat{\Or}_{{\bf p}}}$. Similarly, other terms given by products in (\ref{univ}) describe  the transformation properties of the term denoted by $\dots$ in (\ref{tlinebun}).

 The powers ${D^{\bf p}_{i}}$ come from the {\it index} of the polarization bundle. They are computed as follows:
for our choice of polarization (\ref{polargr}) and chamber (\ref{chamgr}) the index of a fixed point ${\bf p}$ equals:
$$
\textrm{ind}_{{\bf p}} = \left.T^{1/2}X\right|_{{\bf p},>} =  \sum\limits_{{{i \in {\bf p}} \atop {j \notin {\bf p}}} \atop {j>i}} \frac{u_j}{u_i \hbar},
$$
and the integers $D^{\bf p}_{i}$ are the degrees of the index bundle, i.e., the degree in variable $u_i$ of the monomial:
$$
\det \textrm{ind}_{{\bf p}}=\prod\limits_{{{i \in {\bf p}} \atop {j \notin {\bf p}}} \atop {j>i}} \frac{u_j}{u_i \hbar}.
$$
Note that ${\cal U}_{{\bf p}, {\bf q}}$ are certain explicit products of the theta functions and their quasiperiods in all variables are easily determined from (\ref{thettrans}). In particular (\ref{univ}) conveniently packages the information about quasiperiods of the elliptic stable envelopes: the matrices of restrictions (\ref{resmat}) transform in all variables, under shifts by $q$, as~${\cal U}_{{\bf p}, {\bf q}}$.


The elliptic stable envelope $\Stab_{\sigma}({\bf p})$ of a fixed point ${\bf p}$ (corresponding to the choice of chamber $\sigma$ and polarization $T^{1/2}X$) is a section of ${\cal{T}}({\bf p})$ fixed uniquely by a list of properties \cite{AOelliptic}.
Alternative version of the elliptic stable envelope for cotangent bundles to partial flag variates was defined in \cite{Rim,FRV}.
Comparing explicit formulas for elliptic stable envelopes in the case of the variety $X$ from \cite{AOelliptic} and from \cite{Rim,FRV}
one observes that they differ by  a multiple.  The definition of \cite{Rim,FRV} is based on the fact that $X$ is a GKM variety, while definition of \cite{AOelliptic} is more general and is not restricted to GKM varieties. In fact, the Nakajima varieties are almost never GKM varieties.  In this paper we choose the approach of \cite{Rim,FRV}, 
because GKM structure of $X$ will simplify the computations. As we mentioned already, in the case of variety $X$ both approaches lead to the same explicit formulas, thus there is no ambiguity in this choice. 

\begin{Definition} \label{dfelgr}
	{ \it
		The elliptic stable envelope of a fixed point $\Stab_{\sigma}({\bf p})$
		is the unique section of ${\cal{T}}(\bf{p})$, such that its components 
		\be \label{resmat}
		T_{\bf p,q} := \left. \Stab_{\sigma}({\bf p}) \right|_{\widehat{\Or}_{\bf q}}
		\ee
		satisfy the following properties
		
		1)  $T_{{\bf p},{\bf p}}=\prod\limits_{{{i\in {\bf p},} \atop {j\in {\bf n}\setminus {\bf p}} ,}\atop {i<j}} \theta \Big(\dfrac{u_j}{u_i}\Big) \prod\limits_{{{i\in {\bf p},} \atop {j\in {\bf n}\setminus {\bf p}} ,}\atop {i>j}}\theta\Big(\dfrac{u_j}{u_i \hbar}  \Big)$.
		
		2) $T_{{\bf p},{\bf q}}=f_{\bf p,q}\, \prod\limits_{{{i\in {\bf q},} \atop {j\in {\bf n}\setminus {\bf q}} ,}\atop {i>j}}\theta\Big(\dfrac{u_j}{u_i} \hbar^{-1}\Big)$,  where
		$f_{\bf p,q}$ is holomorphic in parameters $u_i$.   } 
\end{Definition}
Let us note that the fact that $\Stab_{\sigma}({\bf p})$
is a section of ${\cal{T}}(\bf{p})$ implies that its restrictions 
$T_{\bf p,q}$ are sections of line bundles on abelian varieties 
$\widehat{\Or}_{\bf q}$ which have the same transformation properties in all variables as ${\mathcal{U}}_{{\bf p,q}}(X)$.

\subsection{Uniqueness of stable envelope for $X$} 
To justify the last definition, we need the following uniqueness theorem.  
\begin{Theorem} \label{uqth} [Appendix A, \cite{FRV}]
	{ \it	The matrix $T_{{\bf p}, {\bf q}}$ satisfying:
		
		1) 	For a given fixed ${\bf p}$,  the collection $\{T_{{\bf p}, {\bf q}} \mid {\bf q} \in X^\bT \}$ form a section of the line bundle $\cal{T}({\bf p})$ (as defined by (\ref{univ})). 
		
		2)  $T_{{\bf p},{\bf p}}=\prod\limits_{{{i\in {\bf p},} \atop {j\in {\bf n}\setminus {\bf p}} ,}\atop {i<j}} \theta \Big(\dfrac{u_j}{u_i}\Big) \prod\limits_{{{i\in {\bf p},} \atop {j\in {\bf n}\setminus {\bf p}} ,}\atop {i>j}}\theta\Big(\dfrac{u_j}{u_i \hbar}  \Big)$.
		
		3) $T_{{\bf p},{\bf q}}=f_{\bf p,q}\, \prod\limits_{{{i\in {\bf q},} \atop {j\in {\bf n}\setminus {\bf q}} ,}\atop {i>j}}\theta\Big(\dfrac{u_j}{u_i} \hbar^{-1}\Big)$, where
		$f_{\bf p,q}$ is holomorphic in parameters $u_i$.   
	
is unique. }

\end{Theorem} 

\begin{proof}
	Assume that we have two matrices which satisfy 1),2),3) and let $\kappa_{\bf p,q}$ be their difference. Assume that $\kappa_{\bf p,q}\neq 0$ for some ${\bf p}$. Let ${\bf q}$ be a maximal (in the partial order defined by the chamber\footnote{The partial order defined by a chamber $\sigma$ is
		$$
		{\bf p}\succ {\bf q}, \ \ \iff {\bf q} \in \textsf{Attr}^{f}_{\sigma}({\bf p})
		$$
		where $\textsf{Attr}^{f}_{\sigma}({\bf p})$ is the full attracting set of a fixed point ${\bf p}$, see Section 3.1 in \cite{AOelliptic}. 
		For $X=T^{*}Gr(k,n)$ and the chamber (\ref{chamgr}) this is the standard Bruhat order on $S_n/(S_k\times S_{n-k})$. For fixed points ${\bf p}=\{{\bf p}_1,\dots,{\bf p}_k\}$ and ${\bf q}=\{{\bf q}_1,\dots,{\bf q}_k\}$ with ${\bf p}_1<\dots<{\bf p}_k$, ${\bf q}_1<\dots<{\bf q}_k$ we have
	$$
	{\bf p} \succ {\bf q} \ \ \iff \ \ {\bf p}_i \geq {\bf q}_i, \ \  i=1,\dots,k.
$$ }) fixed point such that $\kappa_{\bf p,q}\neq 0$. By 3) we know that
	\be \label{kapp}
	\kappa_{{\bf p},{\bf q}}=f_{\bf p,q}\, \prod\limits_{{{i\in {\bf q},} \atop {j\in {\bf n}\setminus {\bf q}} ,}\atop {i>j}}\theta\Big(\dfrac{u_j}{u_i} \hbar^{-1}\Big),
	\ee
	where $f_{\bf p,q}$ is a holomorphic function of $u_i$.
	
	For $i \in {\bf q}$ and $j\in {\bf n}\setminus {\bf q} $ with $i<j$, consider the point ${\bf q}' ={\bf q}\setminus \{i\} \cup \{ j\}$. 
	By construction, ${\bf q}$ and ${\bf q}^{'}$ are connected by an equivariant curve with character $u_i/u_j$. The condition 1) means:
	$$
	\left.\Big(\kappa_{{\bf p},{\bf q}}-\kappa_{{\bf p},{\bf q}^{'}}\Big)\right|_{u_i=u_j}=0.
	$$
	By construction ${\bf q}^{'}\succ {\bf q}$ (in the order on fixed points) and thus $\kappa_{{\bf p},{\bf q}^{'}}=0$, which implies 
	$
	\left.\kappa_{{\bf p},{\bf q}} \right|_{u_i=u_j}=0.
	$
	Comparing with (\ref{kapp}) we conclude that $f_{\bf p,q}$ is divisible by $\theta(u_i/u_j)$. Going over all such pairs of $i,j$ we find:
	$$
	\kappa_{{\bf p},{\bf q}}=f^{'}_{\bf p,q}\,
	\prod\limits_{{{i\in {\bf q},} \atop {j\in {\bf n}\setminus {\bf q}} ,}\atop {i<j}}\theta\Big(\dfrac{u_i}{u_j}\Big)
	\prod\limits_{{{i\in {\bf q},} \atop {j\in {\bf n}\setminus {\bf q}} ,}\atop {i>j}}\theta\Big(\dfrac{u_j}{u_i} \hbar^{-1}\Big)=f^{'}_{\bf p,q}\,T_{{\bf q},{\bf q}},
	$$
	where $f^{'}_{\bf p,q}$ is holomorphic in $u_i$. As a holomorphic function in $u_i\in \matC^*$, it can be expanded as the Laurent series
	$f^{'}_{\bf p,q}=\sum\limits_{k\in \matZ} c_{k} u_i^k$ with nonzero radius of convergence. 
	
	The quasiperiods of functions $T_{{\bf p},{\bf q}}$ are the same as those of the functions  ${{\cal U}}_{{\bf p},{\bf q}}(X)$. In particular, for all $i \not \in {\bf p} \cap {\bf q}$ from (\ref{univ}) we find:
	$$
	f^{'}_{\bf p,q}(u_i q)=f^{'}_{\bf p,q}(u_i) z^{\pm 1} \hbar^{m} 
	$$
	for some integer $m$. We obtain:
	$$
	\sum\limits_{k\in \matZ} c_{k} (z^{\pm 1} \hbar^{m} - q^k ) u_i^k = 0
	$$
	and thus $c_k=0$ for all $k$, i.e., $f^{'}_{\bf p,q}=0$. 
\end{proof}

\subsection{Existence of elliptic stable envelope for $X$}
The following result is proven in \cite{AOelliptic,FRV,konno1}:
\begin{Theorem} \label{ellxth}
	{ \it For canonical polarization (\ref{polargr}) and chamber (\ref{chamgr}) the elliptic stable envelope of a fixed point ${\bf p} \in X^{\bT}$  has the following explicit form:
		
		\be 
		\label{stabgr}
		\begin{array}{|c|}
			\hline\\
			\Stab_{\sigma}({\bf p})=\Sym \left(
			\dfrac{\prod\limits_{l=1}^{k}\left(
				\prod\limits_{i=1}^{{\bf p}_l-1} \theta(y_l u_i\hbar^{-1}) \dfrac{\theta(y_l u_{{\bf p}_l} z^{-1} \hbar^{k-n+{\bf p}_l -2 l})}{\theta(z^{-1} \hbar^{k-n+{\bf p}_l -2 l}) } \prod\limits_{i={\bf p}_l+1}^{n} \theta(y_l u_i)	
				\right)}{\prod\limits_{1\leq i<j \leq k} \theta \Big( \dfrac{y_i}{y_j} \Big) \theta \Big( \dfrac{y_j}{y_i \hbar} \Big) }\right)\\
			\\
			\hline
		\end{array}
		\ee
		where the symbol $\Sym$ stands for the symmetrization over all Chern roots $y_1,\dots,y_k$.  \it}
\end{Theorem}
Note that the components $T_{\bf p,q}$ are defined by this explicit formula as restriction $T_{\bf p,q}=\left. \Stab_{\theta,\sigma}({\bf p})\right|_{{\bf q}}=\left. \Stab_{\theta,\sigma}({\bf p})\right|_{y_i=u^{-1}_{{\bf q}_i}}$. The proof of this theorem is by checking the properties 1)-3) from Theorem \ref{uqth} explicitly, 
details can be found in \cite{FRV}. 

\subsection{Holomorphic normalization}
Note that the stable envelope (\ref{stabgr}) has poles in the K\"ahler parameter $z$.  
It will be more convenient to work with a different normalization of the stable envelope in which it is holomorphic in $z$:
\be \label{hnormx}
{\bf Stab}({\bf p}):=\Theta_{\bf p} \,  \Stab_{\sigma}({\bf p}), 
\ee
where $\Theta_{\bf p}$ is the section of a line bundle on the K\"ahler part 
$\cE_{\textrm{Pic}(X^{'})}$ defined explicitly by
\be \label{prefxor}
\Theta_{\bf p}=\prod_{m=1}^{k}\,\theta(z^{-1} \hbar^{k-n+{\bf p}_m }). 
\ee
(For $X'$ and $\bT'$, see Section \ref{elxd}.) Similarly to Theorem \ref{uqth}, the stable envelope ${\bf Stab}(\bf p)$ can be defined as a unique section of the twisted line bundle on $\textsf{E}_{\bT}(X)$:
\be \label{twlin}
\frak{M}({\bf p}) ={\cal{T}}({\bf p})\otimes \Theta_{{\bf p}}.
\ee
with diagonal restrictions (Property 2  in Theorem \ref{uqth}) given by $T_{\bf p,p} \Theta_{\bf p}$.
Note that the function $\Theta_{\bf p}$ only depends on K\"ahler variables. Thus, the twist of line bundle (\ref{twlin}) does not affect quasiperiods of stable envelopes in the equivariant parameters.

We will see that the section $\Theta_{{\bf p}}$ has the following geometric meaning: it represents the elliptic Thom class of the repelling normal bundle on the dual variety $X'$ (see (\ref{nchar})):
$$
\Theta_{\bf p}=\Theta(N^{'-}_{\bf \lambda}),
$$ 
where $\lambda$ is related to ${\bf p}$ by (\ref{pointbij}), with parameter $a_1 / a_2$ related to K\"ahler parameter $z$ by (\ref{parident}).


\section{Elliptic Stable Envelope for $X'$ \label{elxd}} 

\subsection{$X'$ as a Nakajima quiver variety}

From now on we always assume that $n\geq 2k$. In this section we consider the variety $X'$ which is a Nakajima quiver variety
associated to the $A_{n-1}$ quiver. This variety is defined by the framing dimension vector:
$$
\textsf{w}_i=\delta_{k,i}+\delta_{n- k,i},
$$  
i.e., all framing spaces are trivial except those at position $k$ and $n-k$.
Both non-trivial framing spaces are one-dimensional. The dimension vector has the form $$\textsf{v}=(1,2,\dots,k-1,\underbrace{k,\dots,k}_{(n-2 k+1) \text{-times}},k-1,\dots,2,1).$$
By definition, this variety is given by the following symplectic reduction. Let us consider the vector space:
\be \label{repr}
R=\bigoplus_{i=1}^{n-2} \Hom(\matC^{\textsf{v}_i},\matC^{\textsf{v}_{i+1}})
\bigoplus \Hom(\matC,\matC^{\textsf{v}_k}) \bigoplus \Hom(\matC^{\textsf{v}_{n-k}}, \matC),
\ee
and denote the representatives by $({\bf a}_{l},{\bf i}_{k},{\bf j}_{n-k})$, $l=1,\dots,n-2$. Similarly, the dual vector space:
$$
R^{*}=\bigoplus_{i=1}^{n-2} \Hom(\matC^{\textsf{v}_{i+1}},\matC^{\textsf{v}_i} )
\bigoplus \Hom(\matC^{\textsf{v}_k},\matC) \bigoplus \Hom(\matC, \matC^{\textsf{v}_{n-k}})
$$
with representatives by $({\bf b}_{l},{\bf j}_{k},{\bf i}_{n-k})$. We consider the symplectic space $T^*R=R \oplus R^*$ and the moment map
$$
\mu: T^{*}R \rightarrow \bigoplus_{i=1}^{n-1} \frak{gl}(\textsf{v}_{i})^*.
$$
Denote ${\bf a}=\oplus_i {\bf a}_i$,  ${\bf b}=\oplus_i {\bf b}_i$, ${\bf i}=\oplus_i {\bf i}_i$ and ${\bf j}=\oplus_i {\bf j}_i$, then the moment takes the explicit form 
$\mu(  ( {\bf a}, {\bf b}, {\bf i}, {\bf j}  ) = [{\bf b}, {\bf a}] + {\bf i} \circ {\bf j}.$
With this notation $X'$ is defined as the quotient:
$$
X' :=\mu^{-1}(0)\cap \{ \theta' \textrm{-semistable points} \}/\prod\limits_{i=1}^{n-1} GL(\textsf{v}_{i}).
$$
We will use the canonical choice of the stability parameter~\footnote{We use the same notations for stability condition as in the Maulik-Okounkov \cite{MO}. In particular, for us the stability parameter 
	$\theta=(\theta_i)$ corresponds to a character $\chi: \prod_{i} GL(v_i) \rightarrow \matC^{\times}$ given by
	$$
	\chi : (g_{i}) \mapsto \prod_{i} (\det g_i)^{\theta_i}.
	$$
	This notation is opposite to one used in Ginzburg's lectures \cite{ginzb}, where $\theta$ corresponds to the character $\prod_{i} (\det g_i)^{-\theta_i}$.
}  
\be \label{stabch}
\theta' = (1,1,\dots,1) \in \Lie_\matR (\bK'),
\ee
where $\bK' := U(1)^{n-1}$. The set of the $\theta'$-semistable points in $T^*R$ is described as follows : a point $( ( {\bf a}, {\bf i}_{k},{\bf j}_{n-k}),( {\bf b} ,{\bf j}_{k},{\bf i}_{n-k}) ) \in \mu^{-1}(0) $ is $\theta'$-semistable, if and only if the image of ${\bf i}_k \oplus {\bf i}_{n-k}$ under the actions of $\{ {\bf a}_l, {\bf b}_l , 1 \leq l \leq n-2 \}$ generate the entire space $\bigoplus_{i=1}^{n-1} \matC^{\textsf{v}_i}. $

\subsection{Tautological bundles over $X'$ \label{tbxp}}
We denote  by $\tb_i$ the rank $\textsf{v}_i$ tautological vector bundle on $X'$ associated to $\matC^{\textsf{v}_i}$. It will be convenient to represent the dimension vector and associated tautological bundles using the following combinatorial description. Let us consider a rectangle $\textsf{R}_{n,k}$ with dimensions $k\times (n-k)$. We turn $\textsf{R}_{n,k}$ by $45^{\circ}$ as in the Fig.\ref{examdia}. We will denote by $\Box=(i,j)\in \textsf{R}_{n,k}$ a box in $\textsf{R}_{n,k}$ with coordinates $(i,j)$, $i=1,\dots, n-k$
and $j=1,\dots,k$. We define a function of \emph{diagonal number} on boxes: 
$$
c_{\Box}=i-j+k.
$$ 
Note that $1 \leq c_{\Box} \leq n-1$.  It may be convenient to visualize $c_{\Box}$ as the horizontal coordinate of a box $\Box$ as in Fig.\ref{bij} in Section \ref{bijsec}. The total number of boxes with $c_{\Box}=i$ is $\textsf{v}_i=\dim \tb_i$. With a box $(i,j)$ we associate a variable $x_{i j}$.  It will be convenient to think about the set of $x_{i j}$ with the same $c_{\Box}$ as  Chern roots of tautological bundles, such that in $K$-theory we have:
$$
\tb_m=\sum\limits_{c_{\Box}=m } x_{\Box}.
$$
The tautological bundles $\tb_i$ generate the equivariant $K$-theory of $X'$. The $K$-theory classes are represented by Laurent polynomials in $x_{\Box}$:

$$
K_{\bT'}(X')=\matC[x^{\pm 1}_{i j}]^{\frak{S}_{n,k}}\otimes \matC[a_1^{\pm},a_2^{\pm},\hbar^{\pm}]/I,
$$
where $\bT'$ is the torus described in the next subsection. These are the Laurent polynomials symmetric with respect to each group of Chern roots, i.e., invariant under the group:
\be
\label{symg}
\frak{S}_{n,k}=\prod\limits_{i=1}^{n-1} \frak{S}_{\textsf{v}_i}.
\ee
where $\frak{S}_{\textsf{v}_l}$ acts by permutations on $x_{i j}$ with
$c_{\Box}=l$. The ideal $I$ is the ideal of polynomials which restricts to zero at every fixed point:
$$
I=\{ f(x_{i,j}): \left.f(x_{i,j})\right|_{x_{ij}=\varphi^{\lambda}_{ij}}, \forall \lambda \in (X')^{\bT^{'}} \}.
$$
see (\ref{fpsubs}) below. 

\subsection{Torus action on $X'$}

Let $\bA' = (\matC^{\times})^2$ be a 2-dimensional torus acting on the framing space $\matC \oplus \matC$ by 
$$
(z_1, z_2) \mapsto (z_1 a_1, z_2 a_2 ). 
$$
Let $\matC_\hbar^\times$ be the 1-dimensional torus acting on $T^*R$ by scaling the cotangent fiber
$$
(  ( {\bf a}, {\bf i}_{k},{\bf j}_{n-2k}),( {\bf b} ,{\bf j}_{k},{\bf i}_{n-2k})  ) \mapsto (  ( {\bf a}, {\bf i}_{k},{\bf j}_{n-2k}), \hbar ( {\bf b} ,{\bf j}_{k},{\bf i}_{n-2k})  ). 
$$
Denote their product by $\bT' = \bA' \times \matC_\hbar^\times$.  The fixed loci in $X'$ under the $\bA'$-action admit a \emph{tensor product} decomposition:
$$
(X')^{\bA'} = \coprod_{\textsf{v}^{(1)} + \textsf{v}^{(2)} = \textsf{v} } \mathcal{M} ( \textsf{v}^{(1)}, \delta_k) \times \mathcal{M} (\textsf{v}^{(2)}, \delta_{n-2k} ), 
$$
where $\mathcal{M} ( \textsf{v}^{(1)}, \delta_k)$ is the quiver variety associated with the $A_{n-1}$   quiver with dimension vector $\textsf{v}^{(1)}$, framing vector $\delta_k$ and the same stability condition $\theta'$; similar with $\mathcal{M} (\textsf{v}^{(2)}, \delta_{n-2k} )$. 

We now give a combinatorial description of the quiver variety $\mathcal{M} ( \textsf{v}^{(1)}, \delta_k)$. By definition, a representative of a point in $\mathcal{M} ( \textsf{v}^{(1)}, \delta_k)$ takes the form 
$
( {\bf a}, {\bf i}, {\bf b}, {\bf j} ).
$
It is $\theta'$-semistable, if and only if the image of ${\bf i}$ under the actions of all ${\bf a}$ and ${\bf b}$'s generate the space
$$
\textsf{V}^{(1)} := \bigoplus_{i=1}^{n-1} \matC^{\textsf{v}_i^{(1)}}. 
$$  
One can show that in this case, as an analogue of Lemma 2.8 in \cite{NakLec}, we must have ${\bf j} = 0$.  The moment map equation, together with ${\bf j}_k = 0$ implies that ${\bf a}$ commutes with ${\bf b}$, as operaters on $\textsf{V}^{(1)}$. Therefore, we see that $V^{(1)}$ is spanned by vectors ${\bf a}^i {\bf b}^j {\bf i}_k (1)$, which if nonzero, lie in $\matC^{\textsf{v}_{i-j+k}^{(1)}}$. The stability condition implies that the set $\{ (i,j) \in \matZ^2_{> 0} \mid {\bf a}^{i-1} {\bf b}^{j-1} {\bf i}_k (1) \neq 0 \}$ form a Young diagram, which corresponds to a partition $\lambda$. 

In summary, the quiver variety $\mathcal{M} ( \textsf{v}^{(1)}, \delta_k)$ is either empty or a single point, where the latter case only happens when there exists a partition $\lambda$, whose number of boxes in the $m$-th diagonal is ${\bf v}_{m+k}^{(1)}$. The quiver variety $\mathcal{M} ({\bf v}^{(2)}, \delta_{n-2k} )$ can be described in exactly the same way. 

The restriction of Chern roots to the fixed point can be determined as follows. Consider
	$$
	{\bf a}^{i-1} {\bf b}^{j-1} {\bf i}_k : \matC \to \textsf{V}_{i-j+k}. 
	$$ 
	The action of the group $GL(\textsf{v}^{(1)})$ on ${\bf a}^{i-1} {\bf b}^{j-1} {\bf i}_k$ is
	$$
	{\bf a} \mapsto g {\bf a} g^{-1}, \qquad {\bf b} \mapsto g {\bf b} g^{-1} , \qquad {\bf i}_k \mapsto g_k {\bf i}_k,
	$$
	where $g = (g_1, \cdots, g_{n-1}) \in \prod_i\, GL(\textsf{v}^{(1)}_i)$. So
	$$
	{\bf a}^{i-1} {\bf b}^{j-1} {\bf i}_k \mapsto g {\bf a}^{i-1} {\bf b}^{j-1} {\bf i}_k, 
	$$
	and the action of $\bA'$ on the framing space $\matC$, $z \mapsto a_1 z$, induces the action
	$$
	{\bf a}^{i-1} {\bf b}^{j-1} {\bf i}_k \mapsto  a_1^{-1}  {\bf a}^{i-1} {\bf b}^{j-1} {\bf i}_k. 
	$$
	Here $a_1$ becomes $a_1^{-1}$ because the framing $\matC$ is the domain space of ${\bf i}_k$. To determine the restriction of the Chern root $\varphi_{ij}$, we need $g$ to compensate the action of $\bT'$, i.e.
	$$
	g_i = a_1, \qquad \forall i.
	$$
	So the ($\bA'$-equivariant) restriction is 
	$ \varphi_{ij} = a_1$.
	For the $\hbar$-weight, $\matC_\hbar^\times$ acts on ${\bf b}$ directly by $\hbar$. So  the $\bT'$-equivariant restriction is
	$\varphi_{ij} = a_1 \hbar^{j-1}$. Exactly same consideration applies to the second part $\mathcal{M} ({\bf v}^{(2)}, \delta_{n-2k} )$.

Let us summarize the above discussion.  The set of fixed points $(X')^{\bT'}$ is a finite set labeled by Young diagrams which fit into rectangle $\textsf{R}_{n,k}$. If $\lambda$ is such a diagram we denote its complement in the $(n-k)\times k$ rectangle $\textsf{R}_{n,k}$  by $\bar{\lambda}$. It is easy to see  that $\bar{\lambda}$ is also a Young diagram. The Young diagrams $\lambda$ and $\bar{\lambda}$ divide the rectangle $\textsf{R}_{n,k}$ into two non-intersecting set of boxes.

If $\lambda \in (X')^{\bT'}$ is a fixed point, then the restriction of the Chern roots $x_{\Box}$ of the tautological bundles are given by the following formula:
\be \label{fpsubs}
\left.x_{\Box}\right|_{\lambda}=\varphi^{\lambda}_{\Box}:=\left\{ 
\begin{array}{ll}
	a_1 \hbar^{j-1},& \textrm{if} \ \  (i,j) \in \lambda,\\
	a_2 \hbar^{n-k-i+1},& \textrm{if} \ \  (i,j) \in \bar{\lambda}
\end{array}\right.
\ee
Our notations should be clear from the following example:

\noindent
\begin{Example}
Let us fix $n=8, k=3$ and consider a Young diagram $\lambda=[3,2]$,
then $\bar{\lambda}=[4,3,3]$. The union of $\lambda$ and $\bar{\lambda}$ 
is the rectangular of dimensions $5\times 3$:

\vspace{7mm}

\begin{figure}[h!]
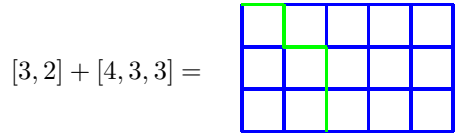

	$$
	[3,2]+[4,3,3]=\somespecialrotate[origin=c]{0}{\hskip 0cm  \yn}
	$$
	\vspace{2mm}
	\caption{\label{fry} An example of a fixed point represented by $[3,2]\in \textsf{R}_{8,3}$ }
\end{figure}
The values of Chern roots (which correspond to boxes of $\textsf{R}_{8,3}$) are given in Fig.\ref{fval}:
\vspace{6mm}
\begin{figure}[h!]
	$$
	\somespecialrotate[origin=c]{0}{\hskip -2cm \scalebox{0.8}{\yx}}
	$$
	\vspace{2mm}
	\caption{\label{fval} The values of $\varphi^{\lambda}_{\Box}$ for $\lambda=[3,2]$ and $n=8,k=3$.}
\end{figure}

\end{Example}

\subsection{Tangent and polarization bundles for $X'$}
To define the elliptic stable envelope we need to specify a polarization and a chamber. We choose the canonical polarization:
\be
\label{dualpol}
T^{1/2}X' =a_1^{-1} \tb_{k}+a_2 \tb_{n-k}^{*} + \sum\limits_{i=1}^{n-2} \tb_{i+1} \tb_{i}^{*} -\sum\limits_{i=1}^{n-1} \tb_i^{*} \tb_i ,
\ee
such that the virtual tangent space takes the form:
$$
TX'=T^{1/2}X'+(T^{1/2}X')^{*} \otimes\hbar^{-1}.
$$
We choose a chamber in the following form:
\be \label{dualco}
\sigma' : (0,1) \in \Lie_{\matR}(\bA').
\ee
The character of the tangent space at a fixed point $\lambda\in (X')^{\bT}$ can be computed by restriction (\ref{fpsubs}):
$$
T_{\lambda} X'=\left.TX' \right|_{\lambda}.
$$
The tangent space at a fixed point decomposes into attracting and repelling parts:
$$
T_{\lambda} X'=N^{'+}_{\lambda}\oplus N^{'-}_{\lambda}, 
$$
where $N^{'\pm}_{\lambda}$ are the subspaces with $\bA$-characters which take positive and negative values on the cocharacter (\ref{dualco}) respectively. Explicitly these characters equal:
\be \label{nchar}
N^{'-}_{\lambda}=\sum\limits_{m=1}^{k} \dfrac{a_1}{a_2} \hbar^{2 k -n +{\bf{p}}_m -2m-1},  \qquad N^{'+}_{\lambda}=\sum\limits_{m=1}^{k} \dfrac{a_2}{a_1} \hbar^{-2 k +n -{\bf{p}}_m +2m}
\ee
where ${\bf p}=\{{\bf p}_1,\dots,{\bf p}_k\}=\textsf{bj}(\lambda)$, for $\textsf{bj}$ described in (\ref{pointbij}). 
\subsection{Elliptic cohomology of $X'$}
The extended elliptic cohomology scheme of $X'$ is a bouquet of $\bT'$ orbits (as a set)
\be \label{elschememirr}
\textsf{E}_{\bT'}(X'):=\coprod_{\lambda\in (X')^{\bT'}}\, \widehat{\Or}'_{\lambda}/\Delta',
\ee
where $\widehat{\Or}'_{\lambda}\cong \cE_{\bT'}\times \cE_{{\textrm{Pic}} (X')}$. The equivariant parameters and K\"ahler parameters of $X'$ are identified with the coordinates in the first and second factor of $\widehat{\Or}'_{\lambda}$ respectively.

By definition, the elliptic stable envelope classes are sections of the twisted elliptic Thom class of the polarization (see discussion in Section \ref{elhohX}):
$$
{\cal{T}}'(\lambda)= \Theta(T^{1/2} X') \otimes \dots
$$
which is a line bundle over the scheme (\ref{elschememirr}) \footnote{${\cal{T}}'(\lambda)$ depends on $\lambda$ via twist terms denoted by $\dots$}.
Sections of the line bundles $\left.{\cal{T}}'(\lambda)\right|_{\widehat{\Or}'_{\mu}}$ over abelian varieties $\widehat{\Or}'_{\mu}$ have the same transformation properties as the following function:
\be \label{ufund}
{\cal{U}}_{\lambda,\mu}(X')=\Theta\Big(\left.T^{1/2} X' \right|_{\mu}\Big) \prod\limits_{\Box\in \textsf{R}_{n,k}} \dfrac{\phi(\varphi^{\lambda}_{\Box},z^{-1}_{c_\Box})}{\phi(\varphi^{\mu}_{\Box},z^{-1}_{c_\Box})}\prod\limits_{i=1}^{2}\dfrac{\phi(a_i,z_{a_i}^{-1}\hbar^{D^{\lambda}_{i}}  )}{\phi(a_i,z_{a_i}^{-1}  )}.
\ee
The powers $D^{\lambda}_{i}$ are determined as follows: let us consider the index of the fixed point
$$
\textrm{ind}_{\lambda}=\left.T^{1/2} X'\right|_{\lambda,>}
$$
The symbol $>$ means that we choose only the $\bT'$-weights of polarization 
$\left.T^{1/2} X'\right|_{\lambda}$  which are positive at~$\sigma'$. Let $\textrm{det}(\textrm{ind}_{\lambda})$ denote the product of all these weights, then $D_i^{\lambda}$ is a degree of this monomial in variable~$a_i$. 


The elliptic stable envelope $\Stab'_{\sigma'} ({\lambda})$ of a fixed point $\lambda$ is a section of ${\cal{T}}'(\lambda)$, which is specified by a list of conditions similar to those of Definition \ref{dfelgr}. The quiver variety $X'$ is not of GKM type. In particular, for $k>1$ it contains families of curves connecting two fixed points. This means that the gluing procedure of orbits and the condition of agreement for sections on different components are more complicated. 


\subsection{Holomorphic normalization}
It will be convenient to work with stable envelopes which differ from one defined in \cite{AOelliptic} by normalization
\be \label{holnd}
{\bf Stab}' (\lambda)=\Theta'_\lambda \, \Stab'_{\sigma'} ({\lambda})
\ee
with prefactor $\Theta'_\lambda$ given by
\be \label{thptefdual}
\Theta'_\lambda=\prod\limits_{{{i\in {\bf p},} \atop {j\in {\bf n}\setminus {\bf p}} ,}\atop {i<j}} \theta \Big(\dfrac{u_j}{u_i}\Big) \prod\limits_{{{i\in {\bf p},} \atop {j\in {\bf n}\setminus {\bf p}} ,}\atop {i>j}}\theta\Big(\dfrac{u_j}{u_i \hbar}  \Big)
\ee
where ${\bf p}=\textsf{bj}(\lambda)$ (see (\ref{pointbij}) below) and 
variables $u_i$ are related to K\"ahler parameters $z_i$ through (\ref{parident}). The stable envelope ${\bf Stab}' ({\lambda})$ is a section of the twisted line bundle on $\textsf{E}_{\bT^{'}}(X^{'})$
\be \label{mdu}
\frak{M}' (\lambda) ={\cal{T}}' (\lambda)\otimes \Theta'_\lambda.
\ee
As the function $\Theta'_{\lambda}$ only depends on K\"ahler variables this twist does not affect quasiperiods of stable envelopes in equivariant parameters. Note that the section $\Theta'_{\lambda}$ 
coincides with the diagonal elements $T_{\bf p,p}$ of the restriction matrix of stable envelopes (see Definition \ref{defel}). Up to a sign, 
is coincides with $\Theta(N^{-}_{\bf p})$ for the repelling part of the normal bundle (\ref{repat}).
We will see that in this normalization the stable envelopes are {\it holomorphic sections} of $\frak{M}'(\lambda)$.

\section{\label{abform} Abelianization formula for elliptic stable envelope for $X'$}

\subsection{Non-K\"ahler part of stable envelope} 
Define a function in the boxes of the rectangle $\textsf{R}_{n,k}$ by:
$$
\rho^{\lambda}_{\Box}=\left\{\begin{array}{ll}
i+j, & \textrm{if} \ \  \Box \in \lambda\\
-i-j, & \textrm{if} \ \ \Box \not\in \lambda
\end{array}\right. 
$$
The following function describes the part of elliptic stable envelope of a fixed point $\lambda$ which is independent on K\"ahler parameters:
\be \label{shenpart}
\textbf{S}^{n,k}_{\lambda}=(-1)^{k(n-k)} \frac{\prod\limits_{{c_{I}=k} \atop {I \in \lambda}} \theta\Big(\dfrac{x_{I}}{a_1}\Big)
	\prod\limits_{{c_{I}=k} \atop {I \not \in \lambda}} \theta\Big(\dfrac{a_1}{x_{I} \hbar}\Big)
	\prod\limits_{c_{I}=n-k} \theta\Big(\dfrac{a_2 \hbar}{x_{I}}\Big)
	\prod\limits_{{c_{I}+1=c_{J}} \atop {\rho^{\lambda}_{I}>\rho^{\lambda}_{J}}} \theta\Big( \dfrac{x_{J} \hbar}{x_I} \Big) \prod\limits_{{c_{I}+1=c_{J}} \atop {\rho^{\lambda}_{I}<\rho^{\lambda}_{J}}} \theta\Big( \dfrac{x_{I}}{x_{J}} \Big)}{\prod\limits_{{c_{I}=c_{J}} \atop { \rho^{\lambda}_{I}>\rho^{\lambda}_{J} }   }  \theta \Big( \dfrac{x_I}{x_J} \Big) \theta \Big( \dfrac{x_I}{x_J \hbar} \Big) }
\ee
where all products run over boxes in $\textsf{R}_{n,k}$ which satisfy the specified conditions. For example, $\prod\limits_{{c_{I}=k} \atop {I \not \in \lambda}}$ denotes a product over all boxes $I\in \lambda$ 
and projection $c_I=k$. Similarly, $\prod\limits_{{c_{I}=c_{J}} \atop { \rho^{\lambda}_{I}>\rho^{\lambda}_{J} }   }$ denotes double product over all boxes $I,J\in \textsf{R}_{n,k}$ with $c_{I}=c_{J}$ and $\rho^{\lambda}_{I}>\rho^{\lambda}_{J}$. 

\vspace{2mm}
\noindent 
\begin{Example}
$$
\begin{array}{l}
\textbf{S}^{3,1}_{[1]}=\theta \left( {\dfrac {x_{{1,1}}}{a_{{1}}}} \right) \theta \left( {
	\dfrac {a_{{2}} \hbar}{x_{{2,1}}}} \right) \theta \left( {\dfrac {x_{{2,1}}\hbar
	}{x_{{1,1}}}} \right)
,
\\
\\
\textbf{S}^{4,2}_{[1,1]}=\frac{\theta \left( {\dfrac {x_{{1,1}}}{a_{{1}}}} \right) \theta \left( {
		\dfrac {a_{{2}}\hbar}{x_{{1,1}}}} \right) \theta \left( {\dfrac {a_{{1}}}{\hbar x
			_{{2,2}}}} \right) \theta \left( {\dfrac {a_{{2}}\hbar}{x_{{2,2}}}}
	\right) \theta \left( {\dfrac {x_{{1,1}}}{x_{{2,1}}}} \right) \theta
	\left( {\dfrac {x_{{1,2}}}{x_{{1,1}}}} \right) \theta \left( {\dfrac {\hbar
			x_{{2,2}}}{x_{{1,2}}}} \right) \theta \left( {\dfrac {x_{{2,2}}}{x_{{2,
					1}}}} \right)
}{\theta \left( {\dfrac {x_{{1,1}}}{x_{{2,2}}}} \right) \theta \left( {
		\dfrac {x_{{1,1}}}{\hbar x_{{2,2}}}} \right)  },
\\
\\
\textbf{S}^{4,2}_{[2]}=\dfrac{\theta \left( {\frac {x_{{1,1}}}{a_{{1}}}} \right) \theta \left( {
		\dfrac {a_{{2}} \hbar}{x_{{1,1}}}} \right) \theta \left( {\dfrac {a_{{1}}}{\hbar x
			_{{2,2}}}} \right) \theta \left( {\dfrac {a_{{2}}\hbar}{x_{{2,2}}}}
	\right) \theta \left( {\dfrac {x_{{2,1}} \hbar}{x_{{1,1}}}} \right) \theta
	\left( {\dfrac {x_{{1,1}} \hbar}{x_{{1,2}}}} \right) \theta \left( {\dfrac {
			\hbar x_{{2,2}}}{x_{{1,2}}}} \right) \theta \left( {\dfrac {x_{{2,2}}}{x_{{2
					,1}}}} \right)
}{\theta \left( {\dfrac {x_{{1,1}}}{x_{{2,2}}}} \right) \theta \left( {
		\dfrac {x_{{1,1}}}{\hbar x_{{2,2}}}} \right) }.
\end{array} 
$$

\end{Example}

\subsection{Trees in Young diagrams}
Let us consider a Young diagram $\lambda$. We will say that two boxes $\Box_1 = (i_1, j_1) , \Box_2 = (i_2, j_2)  \in \lambda$ are \textit{adjacent} if
$$
i_{1}=i_{2}, \quad |j_{1}-j_{2}|=1  \qquad  \textrm{ or } \qquad  j_{1}=j_{2}, \quad  |i_{1}-i_{2}|=1.$$
\begin{Definition} {\it
 	A $\lambda$-tree is a rooted tree with:
	
	($\star$) a set of vertices given by the boxes of a partition $\lambda$,

	($\star,\star$) a root at the box $r=(1,1)$,
	
	($\star,\star,\star$) edges connecting only the adjacent boxes.}
\end{Definition}
Note that the number of $\lambda$-trees depends on the shape of $\lambda$. In particular, there is exactly one
tree for ``hooks''  $\lambda=(\lambda_1,1,\cdots,1)$.

We assume that each edge of a $\lambda$-tree is oriented in a certain way. In particular, on a set of edges we have two well-defined functions
$$ 
h,t : \{\textrm{edges of a tree}\} \longrightarrow \{\textrm{boxes of} \   \lambda \}, 
$$ which for an edge $e$ return its head $h(e)\in \lambda$ and tail $t(e)\in \lambda$  boxes respectively. In this paper we will work with a distinguished \textit{canonical orientation} on $\lambda$-trees.  
\begin{Definition}{\it
	We say that a $\lambda$-tree has canonical orientation if all edges
	are oriented from the root to the end points of the tree.}	
\end{Definition}

For a box $\Box \in \lambda$ and a canonically oriented $\lambda$-tree
$\ft$ we have a well-defined canonically oriented subtree $[\Box, \ft] \subset \ft$ with root at $\Box$.  In particular, $[r,\ft] = \ft$ for a root $r$ of $\ft$.

We rotate the rectangle $\textsf{R}_{n,k}$ by  $45^{\circ}$ as in the Fig.\ref{examdia}, such that the horizontal coordinate of the box is equal to $c_\Box$.
The boundary of a Young diagram $\lambda\subset \textsf{R}_{n,k}$ is a graph $\Gamma$ of a piecewise linear function. We define a function on boxes in $\textsf{R}_{n,k}$ by:
\be \label{betefun}
\beta^{(1)}_{\lambda}(\Box)=\left\{\begin{array}{ll}
	+1&\textrm{if $\Box \in \lambda$ and $\Gamma$ has maximum above $\Box$ }\\
	-1&\textrm{if $\Box \in \lambda$ and $\Gamma$  has minimum above $\Box$}\\
	0& \textrm{else} \end{array} \right.
\ee
Note that $\beta^{(1)}_{\lambda}(\Box)=0$ for all $\Box\in \bar{\lambda}$. For example, the Fig.\ref{examdia} gives the values of $\beta^{(1)}_{\lambda}(\Box)$ for $\lambda=(4,4,4,3,3,2) \in \textsf{R}_{10,4}$.
\begin{figure}[h]
	\hskip 50mm
	\begin{tikzpicture}[draw=blue]
	\draw [line width=1pt] (-7,0) -- (-9,2);
	\node [left] at (-6.75,0.5) {$-1$};
	\node [left] at (-6.75,1.5) {$-1$};
	\node [left] at (-6.75,2.5) {$-1$};
	\node [left] at (-6.75,3.5) {$0$};
	\node [left] at (-5.75,1.5) {$1$};
	\node [left] at (-5.75,2.5) {$1$};
	\node [left] at (-5.75,3.5) {$1$};
	\node [left] at (-5.75,4.5) {$0$};
	\node [left] at (-6.25,1) {$0$};
	\node [left] at (-6.25,2) {$0$};
	\node [left] at (-6.25,3) {$0$};
	\node [left] at (-6.25,4) {$0$};
	\node [left] at (-7.25,1) {$1$};
	\node [left] at (-7.25,2) {$1$};
	\node [left] at (-7.25,3) {$1$};
	\node [left] at (-7.75,1.5) {$0$};
	\node [left] at (-7.75,2.5) {$0$};
	\node [left] at (-5.25,3) {$-1$};
	\node [left] at (-5.25,2) {$-1$};
	\node [left] at (-5.25,4) {$0$};
	\node [left] at (-4.75,3.5) {$1$};
	\node [left] at (-4.75,2.5) {$1$};
	\node [left] at (-4.25,3) {$0$};

	\node [left] at (-8.25,2) {$0$};
	\draw [line width=1pt] (-9,2) -- (-6,5);
	\draw [line width=1pt] (-8,1) -- (-5,4);
	\draw [line width=1pt] (-8.5,1.5) -- (-5.5,4.5);

	\draw [line width=1pt,green] (-7,3) -- (-6,4);
	\draw [line width=1pt,green] (-9,2) -- (-7.5,3.5);
	\draw [line width=1pt,green] (-5.5,3.5) -- (-5,4);
	
	\draw [line width=1pt] (-7.5,0.5) -- (-4.5,3.5);
	\draw [line width=1pt] (-6.5,0.5) -- (-4,3);
	\draw [line width=1pt] (-7,0) -- (-5.5,1.5);

	\draw [line width=1pt] (-6.5,0.5) -- (-8.5,2.5);
	\draw [line width=1pt] (-6,1) -- (-8,3);
	\draw [line width=1pt] (-5.5,1.5) -- (-7.5,3.5);
	\draw [line width=1pt] (-5,2) -- (-7,4);
	
	\draw [line width=1pt] (-4,3) -- (-6,5);
	\draw [line width=1pt] (-4.5,2.5) -- (-6.5,4.5);
	
	\draw [line width=1pt, green] (-6,4) -- (-5.5,3.5);
	\draw [line width=1pt,green] (-4,3) -- (-5,4);
	\draw [line width=1pt,green] (-7,3) -- (-7.5,3.5);
	
	\end{tikzpicture}
	\caption{Values of function $\beta^{(1)}_{\lambda}(\Box)$ for the diagram $\lambda=(4,4,4,3,3,2) \in \textsf{R}_{10,4}$.
		The boundary $\Gamma$ of the Young diagram $\lambda$ is denoted by green color. \label{examdia} }
\end{figure}
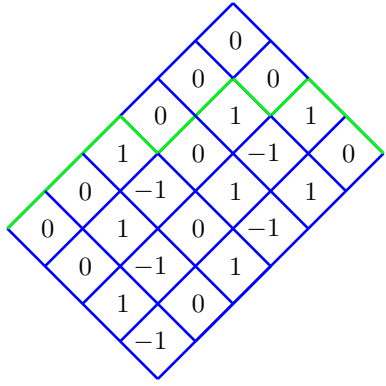 

\noindent
We also define 
$$
\beta^{(2)}_{\lambda}(\Box) =\left\{\begin{array}{ll}
+1 & \textrm{if}\ \  c_\Box <k \\
-1 & \textrm{if} \ \ c_\Box > n-k \\
0 & \textrm{else}
\end{array}\right.  
$$
and we set 
\be
\label{betafun}
\textsf{v}(\Box)=\beta^{(1)}_{\lambda}(\Box)+\beta^{(2)}_{\lambda}(\Box).
\ee

\subsection{K\"ahler part of the stable envelope} 
Let $\lambda\subset \textsf{R}_{n,k}$ be a Young diagram and $\bar{\lambda}=\textsf{R}_{n,k}\setminus \lambda$ is the complement Young diagram as above. 
Let $\ft \cup \bar\ft$ be the (disjoint) union of $\lambda$-tree $\ft$ and $\bar{\lambda}$-tree $\bar\ft$. We define a function:
$$
\textbf{W}^{Ell}(\ft \cup \bar\ft;x_i,z_i):=
\textbf{W}^{Ell}(\ft;x_i,z_i) \textbf{W}^{Ell}(\bar\ft, x_i,z_i),
$$
for the elliptic weight of a tree, where
\be \label{wpartell}
\textbf{W}^{Ell}(\ft;x_i,z_i):=(-1)^{\kappa (\ft)} \phi \Big( \frac{\varphi^{\lambda}_{r}}{x_r}, \prod\limits_{\Box \in [r, \ft]} z^{-1}_{c_\Box} \hbar^{-\textsf{v}(\Box)} \Big) \prod\limits_{e\in \ft} \phi\Big(\dfrac{x_{t(e)} \varphi^{\lambda}_{h(e)}}{\varphi^{\lambda}_{t(e)} x_{h(e)}}, \prod\limits_{\Box \in [h(e), \ft]} z^{-1}_{c_\Box} \hbar^{-\textsf{v}(\Box)}\Big),
\ee
and similarly for $\textbf{W}^{Ell}(\bar\ft, x_i,z_i)$. 

Here $\Box\in \ft$ or $e\in \ft$ means the box or edge belongs to the tree. The sign of a tree depends on the number $\kappa (\ft)$ which is equal to the number of edges in the tree with wrong orientation. In other words, $\kappa(\ft)$ is the number of edges in $\ft$ directed down or to the left, while $\kappa(\bar\ft)$ is the number of edges in $\bar\ft$ directed up or to the right. To avoid ambiguity, we also define
$\textbf{W}^{Ell}(\ft;x_i,z_i):=1$ for a tree in the empty Young diagram.  

\vspace{4mm}
\noindent
\begin{Example} 
Let us consider a Young diagram $[2,2]\subset  \textsf{R}_{5,2}$
with trees
$\exone \hspace{5mm}$. 

\vspace{-5mm} \noindent
By definition we have:
$$\textbf{W}^{Ell}\Big(\exone \hspace{5mm} \Big)=\textbf{W}^{Ell}\Big(\exonl \hspace{5mm} \Big) \textbf{W}^{Ell}\Big(\exonr \hspace{5mm} \Big).
$$

\vspace{-30mm} \noindent
In this case we have six boxes with the following characters:
$$
\varphi^{\lambda}_{1 1}= a_1,\ \  \varphi^{\lambda}_{2 1}= a_1, \ \ \varphi^{\lambda}_{3 1}= a_2 \hbar, \ \ \varphi^{\lambda}_{1 2}= a_1 \hbar, \ \ \varphi^{\lambda}_{2 2}= a_1 \hbar, \ \ \varphi^{\lambda}_{3 2}= a_2 \hbar.
$$
Similarly for the $\hbar$-weights of boxes (\ref{betafun}) we obtain:
$$
\begin{array}{l}
\beta(1,1)=\beta^{(1)}(1,1)+\beta^{(2)}(1,1)=1+0=1,\\
\beta(1,2)=\beta^{(1)}(1,2)+\beta^{(2)}(1,2)=0+1=1,\\
\beta(2,1)=\beta^{(1)}(2,1)+\beta^{(2)}(2,1)=0+0=0,\\
\beta(2,2)=\beta^{(1)}(2,2)+\beta^{(2)}(2,2)=1+0=1,\\
\beta(3,1)=\beta^{(1)}(3,1)+\beta^{(2)}(3,1)=0-1=-1,\\
\beta(3,2)=\beta^{(1)}(3,2)+\beta^{(2)}(3,2)=0+0=0.
\end{array}
$$

\vspace{3mm}

\noindent
First, let us consider $\textbf{W}^{Ell}\Big(\exonl \hspace{5mm} \Big)$. \vspace{-6mm} In this case we have a tree with the root at $r=(1,1)$ and three edges with the following heard and tails: 
$$
t(e_1)=(1,1), \ h(e_1)=(1,2), \  t(e_2)=(1,1), \ h(e_2)=(2,1), \ t(e_3)=(1,2), \ h(e_3)=(2,2).
$$ 
For the first factor in (\ref{wpartell}) we obtain:
$$
\phi \Big( \frac{\varphi^{\lambda}_r}{x_r},\prod\limits_{\Box \in [r, \ft]} z^{-1}_{c_\Box} \hbar^{-\textsf{v}(\Box)} \Big) = \phi \Big( \frac{a_1}{x_{1,1}}, z_{1}^{-1} z_{2}^{-2} z_{3}^{-1} \hbar^{-3} \Big)
$$
For the edges in the product (\ref{wpartell}) we obtain:
$$
\phi\Big(\dfrac{x_{t(e_1)} \varphi^{\lambda}_{h(e_1)}}{\varphi^{\lambda}_{t(e_1)} x_{h(e_1)}}, \prod\limits_{\Box \in [h(e_1), \ft]} z^{-1}_{c_\Box} \hbar^{-\textsf{v}(\Box)}\Big)=\phi \Big( \dfrac{x_{1 1}}{x_{1 2} } \hbar, z^{-1}_{1} z^{-1}_{2} \hbar^{-2}  \Big),
$$
$$
\phi\Big(\dfrac{x_{t(e_2)} \varphi^{\lambda}_{h(e_2)}}{\varphi^{\lambda}_{t(e_2)} x_{h(e_2)}}, \prod\limits_{\Box \in [h(e_2), \ft]} z^{-1}_{c_\Box} \hbar^{-\textsf{v}(\Box)}\Big)=\phi \Big( \dfrac{x_{1 1}}{x_{2 1} }, z_{3}^{-1}  \Big),
$$

$$
\phi\Big(\dfrac{x_{t(e_3)} \varphi^{\lambda}_{h(e_3)}}{\varphi^{\lambda}_{t(e_3)} x_{h(e_3)}}, \prod\limits_{\Box \in [h(e_3), \ft]} z^{-1}_{c_\Box} \hbar^{-\textsf{v}(\Box)}\Big)=\phi \Big( \dfrac{x_{1 2}}{x_{2 2} } , z_{2}^{-1} \hbar^{-1}  \Big).
$$
Thus, overall we obtain:
$$
\textbf{W}^{Ell}\Big(\exonl \hspace{5mm} \Big)=
\phi \Big( \frac{a_1}{x_{1,1}}, \dfrac{1}{z_{1} z_{2}^{2} z_{3} \hbar^{3}} \Big) \phi \Big( \dfrac{x_{1 1}\hbar}{x_{1 2} }, \dfrac{1}{z_{1} z_{2} \hbar^{2}}  \Big)
\phi \Big( \dfrac{x_{1 1} }{x_{2 1}} , \dfrac{1}{ z_{3}}  \Big)
\phi \Big( \dfrac{x_{1 2}}{x_{2 2} } , \dfrac{1}{ z_{2} \hbar}  \Big).
$$
Similarly, for the second multiple we obtain:
$$
\textbf{W}^{Ell}\Big(\exonr \hspace{5mm} \Big)=\phi \Big( \frac{a_2 \hbar}{x_{3 2}},\dfrac{\hbar}{z_3 z_4} \Big) \phi \Big( \dfrac{x_{3 2}}{x_{3 1}}, \dfrac{\hbar}{z_{4}} \Big).
$$
\end{Example}

\subsection{Formula for elliptic stable envelope \label{upsion}} 
\begin{Definition} {\it
	The skeleton $\Gamma_{\lambda}$ of a partition $\lambda$ is the graph, whose vertices are given by the set of boxes
	of $\lambda$ and whose edges connect all adjacent boxes in $\lambda$.}
\end{Definition}

\begin{Definition} \label{lsdef}{\it
	A} $\reflectbox{\textsf{L}}${\it-shaped subgraph in $\lambda$ is a subgraph $\gamma\subset \Gamma_{\lambda}$ consisting of two edges
	$\gamma=\{\delta_1,\delta_2\}$ with the following end boxes:
	\be 
	\label{gshapped} \delta_{1,1}=(i,j), \qquad \delta_{2,1}=\delta_{1,2}=(i+1,j), \qquad \delta_{2,2}=(i+1,j+1).
	\ee}
\end{Definition} \noindent
It is easy to see that the total number of  \reflectbox{\textsf{L}}-shaped subgraphs in $\lambda$ is equal to 
\be
\label{nofsh}
m=\sum\limits_{l \in \matZ}( \textsf{d}_{l}(\lambda)-1),
\ee
where $\textsf{d}_{l}(\lambda)$ is the number of boxes in the $l$-diagonal of $\lambda$ 
\be
\label{dfun}
\textsf{d}_{l}(\lambda)=\# \{\Box \in \lambda \mid c_\Box =l \}.
\ee 
There is a special set of $\lambda$-trees, constructed as follows.  For each $\reflectbox{\textsf{L}}$-shaped subgraph $\gamma_i$ in $\lambda$ we choose one of its two edges. We have $2^m$ of such choices. For each such choice the set  of edges $\Gamma_\lambda \setminus \{ \delta_i \}$ is a $\lambda$-tree. We denote the set of $2^m$ $\lambda$-trees which appear this way by $\Upsilon_\lambda$.

Now let us define $\Upsilon_{n,k}= \Upsilon_\lambda \times \Upsilon_{\bar{\lambda}}$, whose elements of are pairs of trees $(\ft ,\bar\ft)$, where $\ft$ is a $\lambda$-tree with root $(1,1)$, $\bar\ft$ is a $\bar\lambda$-tree with root $(n-k,k)$. Both trees are constructed in the way described as above, and they are disjoint, i.e., do not have common vertices. 

\vspace{2mm}
\noindent
\begin{Example} \vspace{2mm}
Let us consider $\lambda=[3,2]\in \textsf{R}_{8,3}$ and $\bar \lambda=[4,3,3]$. 
A typical element of $\Upsilon_{8,3}$ looks like:

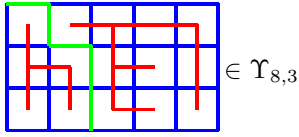
\vspace{5mm}

$$
\somespecialrotate[origin=c]{0}{\hskip -1cm\yd  $\ \ \ \ \  \ \ \ \ \ \in \Upsilon_{8,3}$}  
$$

\vspace{5mm}

\end{Example}

\noindent
The following theorem can be proved using the same arguments as in \cite{EllHilb}.
\begin{Theorem} \label{mainth} {\it
	The elliptic stable envelope of a fixed point $\lambda$ for the chamber $\sigma'$ defined by (\ref{dualco})
	and polarization (\ref{dualpol}) has the following form:
	\be \label{ellipticenvelope}
	\begin{array}{|c|}\hline
		\\
		\ \ \ \Stab'_{\sigma'}(\lambda) = \Sym_{\frak{S}_{n,k}}\Big( \mathbf{S}^{n,k}_{\lambda} \sum\limits_{(\ft, \bar\ft) \in \Upsilon_{n,k} } \mathbf{W}^{Ell}(\ft \cup \bar\ft) \Big) \ \ \\\
		\\
		\hline
	\end{array}
	\ee
	where the symbol $\Sym_{\frak{S}_{n,k}}$ denotes a sum over all permutations in the group (\ref{symg}). }
\end{Theorem}
\begin{proof}
The proof of this theorem is bases on the abelianization of elliptic stable envelope 
developed in Section 4.3 of \cite{AOelliptic} which, in turn, is inspired
by the abelianization of stable envelopes in cohomology \cite{Shenf}.  The proof follows closely the proof of the main result of \cite{EllHilb}. To keep the presentation short we will refer to  the corresponding results in these papers when possible. We also refer to \cite{AOelliptic,EllHilb} for definitions of all maps and objects appearing here. 

Let us denote by  $\textbf{AX}'$ the abelianization of the Nakajima variety $X'$.
This is a hypertoric variety defined by the following symplectic reduction
$$
\textbf{AX}':=T^*R/\!\!/\!\!/\!\!/ S
$$ 
where $R$ is given by (\ref{repr}) and $S$ is the maximal torus of $\prod_i^n GL(\textsf{v}_i)$. The stability condition for this symplectic reduction is defined by 
(\ref{stabch}). Let ($\lambda,\bar{\lambda}$) be a $\bT'$ fixed points in
$X'$. By definition ($\lambda,\bar{\lambda}$) is a zero-dimensional Nakajima quiver variety of type $A_{n-1}$.  We will denote by $\textbf{AX}'_{(\lambda,\bar{\lambda})}$ the abelianization of this Nakajima variety. It is a hypertoric subvariety $\textbf{AX}'_{(\lambda,\bar{\lambda})}\subset \textbf{AX}'$ fixed by the action of torus $\bA$. We denote by $\textrm{Stab}^{'}_{\fC}$ the elliptic stable envelope
map for these hypertoric varieties. The chamber $\fC$ here is the chamber of $\bA$  defined by cocharacter (\ref{dualco}). 

The abelianization diagram for Nakajima varieties (see (74) in \cite{AOelliptic}) expresses the elliptic stable envelope of the fixed point ($\lambda,\bar{\lambda}$) in $X'$ as the following composition:
\be \label{abla}
\textrm{Stab}_{\fC}(\lambda,\bar{\lambda})=\pi_{*}\circ \j^{*}_{+}\circ (\j_{- *})^{-1}\circ \textrm{Stab}_{\fC}'  \circ \j^{'}_{- *}\circ(\j^{'*}_{+})^{-1}\circ\pi^{'-1}_{*}
\ee
For the definition of all maps here we refer to Section 4.3 in \cite{AOelliptic}.   

\begin{Lemma} \label{faclem}
	{\it The Nakajima quiver variety $(\lambda,\bar{\lambda})$ is a direct product of two
		zero-dimensional Nakajima varieties of $A_{\infty}$-type corresponding to dimension vectors given by $\lambda$ and $\bar{\lambda}$. The abelianization 
		$\mathbf{AX'}_{(\lambda,\bar{\lambda})}=\ahilb_{\lambda} \times \ahilb_{\bar{\lambda}}$ where $\ahilb_{\lambda}$ denotes the abelianization of 
		$A_{\infty}$ Nakajima variety corresponding to $\lambda$.}
\end{Lemma}
\begin{proof}
	The fixed point set of a Nakajima quiver variety with respect to action of the framing torus is  isomorphic to the direct product of Nakajima varieties for the same quiver and one-dimensional framings (this property of quiver varieties is known as tensor product structure). Non-empty $A_{n-1}$ quiver varieties with one-dimensional
	framing are all zero-dimensional $A_{\infty}$ quiver varieties and have dimension vectors corresponding to 
	Young diagrams \cite{dinksmir2}.  	
\end{proof}

\begin{Corollary} \label{corfac}
	{ \it The abelianization maps $\j^{'}_{- *}, \j^{'*}_{+}, \pi^{'}_{*}$ factor into direct products of maps:
		$$
		\j^{'}_{- *}=\j^{'}_{1,- *}\times \j^{'}_{2,- *}, \ \ \  \j^{'*}_{+}= \j^{'*}_{1,+}\times  \j^{'*}_{2,+}, \ \ \ \pi^{'}_{*}=\pi^{'}_{1,*}\times \pi^{'}_{2,*}
		$$	
		where $(\j^{'}_{1,- *},\j^{'*}_{1,+},\pi^{'}_{1,*})$ are maps for zero-dimensional Nakajima quiver variety 
		$\lambda$ (i.e. $A_{\infty}$ quiver variety with one-dimensional framing and the dimension vector given by $\lambda$) and its abelianization $\ahilb_{\lambda}$;  similarly $(\j^{'}_{2,- *},\j^{'*}_{2,+},\pi^{'}_{2,*})$
		are the abelianization maps for $\bar\lambda$ and $\ahilb_{\bar\lambda}$. }
\end{Corollary}

The hypertoric varieties $\textbf{AH}_{\lambda}$ were considered in Section 6 of \cite{EllHilb}. In particular, it was shown that $\textbf{AH}_{\lambda}$ contains fixed points (of a maximal torus acting on $\textbf{AH}_{\lambda}$ by automorphisms) labeled by $\lambda$-trees. For trees $\ft$, $\bar\ft$ in $\lambda$ and $\bar\lambda$ we denote by the same symbols the corresponding fixed points in $\textbf{AH}_{\lambda}$ and $\textbf{AH}_{\bar\lambda}$.

Let $\matC^{\times}_{\ft}, \matC^{\times}_{\bar\ft}$ be one-dimensional tori acting on $\ahilb_{\lambda}$ and $\ahilb_{\bar\lambda}$ respectively with chambers $\fC^{''}_1,\bar{\fC}^{''}_2$ as defined in Section 6.3 of \cite{EllHilb}. We denote the corresponding chamber in 
$\matC^{\times}_{\ft}\times \matC^{\times}_{\bar\ft}$ by $\fC^{''}$ (such that $\fC^{''}_1,\bar{\fC}^{''}_2$ are one-dimensional faces of $\fC^{''}$). 
Finally, we denote by $\fC^{'}$ the chamber of the torus $\bA\times \matC^{\times}_{\ft}\times \matC^{\times}_{\bar\ft}$ corresponding to the chambers $\fC$ and $\fC^{''}$ (i.e., such that $\fC$ and $\fC{''}$ are faces of the chamber $\fC^{'}$).  

The following is a version of Proposition 6 from \cite{EllHilb} for the case of $X'$:

\begin{Proposition} \label{facprop} {\it
		Up to a shift of K\"ahler parameters $z_i\to z_i \hbar^{m_i}$, $m_i \in \matZ$, the elliptic stable of $\matC^{\times}_{\ft}\times \matC^{\times}_{\bar\ft}$-fixed point $(\ft,\bar{\ft})$ 	in $\mathbf{AX}'_{(\lambda,\bar{\lambda})}$ corresponding to the chamber $\fC^{''}$ equals:
		\be \label{facstab}
		\Stab_{\fC^{''}}(\ft,\bar{\ft})=\Stab_{\fC^{''}_1}(\ft) \Stab_{{\fC}^{''}_2}(\bar{\ft}),
		\ee
		where 
		$$
		\begin{array}{l}
		\Stab_{\fC^{''}_1}(\ft)=\prod\limits_{{c_{I}=k}\atop{I\in \lambda}} \theta\Big(\dfrac{x_I}{a_1}\Big) 
		\prod\limits_{{{c_{I}+1=c_{J}} \atop {\rho^{\lambda}_{I}>\rho^{\lambda}_{J}}} \atop {I,J \in \lambda}} \theta\Big( \dfrac{x_{J} \hbar}{x_I} \Big) \prod\limits_{{{c_{I}+1=c_{J}} \atop {\rho^{\lambda}_{I}<\rho^{\lambda}_{J}}}
			\atop {I,J \in \lambda}} \theta\Big( \dfrac{x_{I}}{x_{J}} \Big)\, W_\ft (z_i), \\
		\Stab_{{\fC}^{''}_2}(\ft)=\prod\limits_{{c_{I}=n-k}\atop{I\in \bar{\lambda}}} \theta\Big(\dfrac{a_2 \hbar}{x_I}\Big) 
		\prod\limits_{{{c_{I}+1=c_{J}} \atop {\rho^{\lambda}_{I}>\rho^{\lambda}_{J}}} \atop {I,J \in \bar{\lambda}}} \theta\Big( \dfrac{x_{J} \hbar}{x_I} \Big) \prod\limits_{{{c_{I}+1=c_{J}} \atop {\rho^{\lambda}_{I}<\rho^{\lambda}_{J}}}
			\atop {I,J \in \bar{\lambda}}} \theta\Big( \dfrac{x_{I}}{x_{J}} \Big) \, W_{\bar\ft}(z_i),
		\end{array}
		$$
		the elliptic stable envelope of $\bA\times\matC^{\times}_{\ft}\times \matC^{\times}_{\bar\ft}$-fixed point $(\ft,\bar{\ft})$ 	in $\mathbf{AX}^{'}$ 
		is given by
		\be \label{stabprime}
		\Stab_{\fC^{'}}(\ft,\bar{\ft})=\prod\limits_{{c_{I}=k} \atop {I \in \lambda}} \theta\Big(\dfrac{x_{I}}{a_1}\Big)
		\prod\limits_{{c_{I}=k} \atop {I \not \in \lambda}} \theta\Big(\dfrac{a_1}{x_{I} \hbar}\Big)
		\prod\limits_{c_{I}=n-k} \theta\Big(\dfrac{a_2 \hbar}{x_{I}}\Big)
		\prod\limits_{{c_{I}+1=c_{J}} \atop {\rho^{\lambda}_{I}>\rho^{\lambda}_{J}}} \theta\Big( \dfrac{x_{J} \hbar}{x_I} \Big) \prod\limits_{{c_{I}+1=c_{J}} \atop {\rho^{\lambda}_{I}<\rho^{\lambda}_{J}}} \theta\Big( \dfrac{x_{I}}{x_{J}} \Big)
		W_{\ft}(z_i) W_{\bar\ft}(z_i)
		\ee
		with
		\be
		\label{Wpart}
		W_{\ft}(z_i) =(-1)^{\kappa_{\textbf{t}}} \phi \Big( x_r, \prod\limits_{i=1}^{n}  z_i \Big) \prod\limits_{e\in \textbf{t}} \phi\Big(\dfrac{x_{h(e)} \varphi^{\lambda}_{t(e)}}{x_{t(e)}\varphi^{\lambda}_{h(e)}}, \prod\limits_{i \in [h(e),\ft]} z_i\Big)
		\ee}
\end{Proposition}
\begin{proof}
	By Lemma \ref{faclem}, elliptic stable envelope of a fixed point factors 
	to a product of elliptic stable envelopes (\ref{facstab}). 
	The explicit formulas for elliptic stable envelopes of $\ft$ in $\ahilb_{\lambda}$
	and are given by Proposition 6 in \cite{EllHilb} which gives the above explicit formulas.    
\end{proof}
We note that the variables $z_i$, $i\in\textsf{R}_{k,n}$ in (\ref{Wpart}) denote the K\"ahler parameters  associated to the line bundles $x_i$ on the abelianization $\textbf{AX}'$. 

\begin{Proposition}
	$$	
	\pi^{'}_{*} \circ \j^{'*}_{+}\circ (\j^{'}_{- *})^{-1} \Big(\sum\limits_{(\ft, \bar\ft) \in \Upsilon_{n,k} } \textrm{Stab}_{\fC^{''}}(\ft,\bar{\ft})\Big) =1.
	$$
\end{Proposition} 
\begin{proof}
	By Corollary \ref{corfac} all  $\j^{'}_{- *}, \j^{'*}_{+}, \pi^{'}_{*}$ factor in into the direct product of maps. By Proposition \ref{facprop} the stable envelope of the fixed point $(\ft,\bar{\ft})$ also factorizes into a product of stable envelopes. This gives:
	$$
	\begin{array}{l}
	\pi^{'}_{*} \circ \j^{'*}_{+}\circ (\j^{'}_{- *})^{-1} \Big(\sum\limits_{(\ft, \bar\ft) \in \Upsilon_{n,k} } \textrm{Stab}_{\fC^{''}}(\ft,\bar{\ft})\Big)=\\
	\left(\pi^{'}_{1,*} \circ \j^{'*}_{1,+}\circ (\j^{'}_{1,- *})^{-1} \Big(\sum\limits_{\ft} \textrm{Stab}_{\fC^{''}_1}(\ft)\Big)\right) 
	\left(\pi^{'}_{2,*} \circ \j^{'*}_{2,+}\circ (\j^{'}_{2,- *})^{-1} \Big(\sum\limits_{\bar \ft} \textrm{Stab}_{\fC^{''}_2}(\bar \ft)\Big)\right)
	\end{array}
	$$	
	Each factor here is equal to $1$ by Theorem 5 in \cite{EllHilb}.   
\end{proof} 
The last proposition implies that the abelianization formula (\ref{abla}) can be written in the form:
$$
\textrm{Stab}(\lambda,\bar{\lambda})=\pi_{*}\circ \j^{*}_{+}\circ (\j_{- *})^{-1}\circ \textrm{Stab}_{\fC}^{'}\Big(\sum\limits_{(\ft, \bar\ft) \in \Upsilon_{n,k} } \textrm{Stab}_{\fC^{''}}(\ft,\bar{\ft})\Big)=\pi_{*}\circ \j^{*}_{+}\circ (\j_{- *})^{-1} \Big(\sum\limits_{(\ft, \bar\ft) \in \Upsilon_{n,k} } \textrm{Stab}_{\fC^{'}}(\ft,\bar{\ft})\Big),
$$
where the second identity $\textrm{Stab}_{\fC^{'}}=\textrm{Stab}_{\fC}^{'} \circ \textrm{Stab}_{\fC^{''}}$ is the triangle lemma for elliptic stable envelope,
see Section 3.6 in \cite{AOelliptic}. We now see that the last expression coincides with (\ref{ellipticenvelope}). Indeed, the numerator of (\ref{ellipticenvelope}) 
is given by (\ref{stabprime}), the product $\prod\limits_{{c_{I}=c_{J}} \atop { \rho^{\lambda}_{I}>\rho^{\lambda}_{J} }   }  \theta \Big( \dfrac{x_I}{x_J} \Big) $ in the denominator of (\ref{shenpart}) comes from pushforward $\pi_{*}$ computed by localization, similarly $\prod\limits_{{c_{I}=c_{J}} \atop { \rho^{\lambda}_{I}>\rho^{\lambda}_{J} }   }  \theta \Big( \dfrac{x_I}{x_J \hbar} \Big) $ comes from the pushforward $(\j_{- *})^{-1}$. We refer to Section 4.3 of \cite{AOBethe} for
computations of the corresponding normal bundles to maps $\pi$ and $\j_{-}$. 

By definition, the K\"ahler parameters $z_l$, $l=1,\dots,n-1$ of the Nakajima variety $X^{'}$ are parameters associated to tautological line bundles $\cL_l=\det \tb_l$. Expressed in the corresponding Chern roots these line bundles have the form $\cL_m=\prod\limits_{{i \in \textsf{R}_{n,k}} \atop {c_{i}=m}} x_{\Box}$. This means that the  K\"ahler parameters $z_i$, $i \in \textsf{R}_{k,n}$  corresponding to the line bundles $x_i$ on the abelianization $\textbf{AX}'$  restrict to the K\"ahler parameters  of $X'$ by $z_{i}\to z_{c_{i}}$. This substitution gives desired dependence of stable envelope on K\"ahler variables. 

The last step is to find correct shifts of the K\"ahler variables by powers of $\hbar$.
Indeed, the proposition (\ref{facprop}) provides the explicit formulas for elliptic stable envelopes up to shifts $z_i \to z_i \hbar^{m_i}$ for some integers $m_i$. 
The values of $m_i$ are uniquely determined by the condition that quasi-periods
$x_{i}\to x_i q$ of (\ref{ellipticenvelope}) coincide with the quasi-periods of the section (\ref{ufund}). A calculation repeating the last part of Section 8.3 in \cite{EllHilb} gives exactly the combinatorial formula for the $\hbar$-powers (\ref{betafun}). 
\end{proof}

\subsection{Refined formula}
In this subsection, we prove a refined version of formula (\ref{ellipticenvelope}), in the sense that when restricted to another fixed point $\mu$, the summation will be rewritten as depending on the trees $\bar\ft$ only, but not on the trees $\ft$. The refined formula will be of crucial use to us in the proof of the main theorem.

Given a fixed point $\lambda$,  the original formula (\ref{ellipticenvelope}) has the following structure (for simplicity we omit the chamber subscript $\sigma'$):
$$
\Stab' (\lambda)  = \sum_{\sigma\in \fS_{n,k}, \ft, 
	\bar\ft} \frac{\cN^\sigma}{\cD^\sigma}  \cR^\sigma (\ft, \bar\ft) \cW^\sigma (\ft, \bar\ft) ,
$$
where we denote
\be  \nonumber
\cN &:=& (-1)^{k (n-k)-1} \prod_{\substack{c_I = k \\ I \in \lambda \\  I\neq (1,1) }} \theta \Big( \dfrac{x_I}{a_1} \Big) \prod_{\substack{c_I = k \\ I \not\in \lambda }} \theta \Big( \frac{a_1}{x_I \hbar} \Big) \prod_{ \substack{ c_I = n-k \\ I\neq (n-k, k) }} \theta \Big( \frac{a_2 \hbar}{x_I} \Big) \prod_{\substack{c_I + 1 = c_J \\ \rho^\lambda_I > \rho^\lambda_J \\ ( I \leftrightarrow J ) \not\in \Gamma_\lambda \cup \Gamma_{\bar\lambda} }} \theta \Big( \frac{x_J \hbar}{x_I} \Big) \prod_{\substack{ c_I + 1 = c_J \\ \rho^\lambda_I < \rho^\lambda_J \\ ( I \leftrightarrow J ) \not\in \Gamma_\lambda \cup \Gamma_{\bar\lambda} }} \theta \Big( \dfrac{x_I}{x_J} \Big), \\ \nonumber
\\  \nonumber
\cD &:=& \prod_{c_I = c_J, \, \rho^\lambda_I > \rho^\lambda_J} \theta \Big( \frac{x_I}{x_J} \Big) \prod_{c_I = c_J, \, \rho^\lambda_I > \rho^\lambda_J +2} \theta \Big( \frac{x_I}{x_J \hbar} \Big), \\ \nonumber
\\ \nonumber
\cR (\ft, \bar\ft) &:=& \dfrac{ \prod_{\substack{c_I + 1 = c_J, \, \rho^\lambda_I = \rho^\lambda_J+1 \\ ( I \leftrightarrow J ) \in \Gamma_\lambda \backslash \ft \cup \Gamma_{\bar\lambda} \backslash \bar\ft }} \theta \Big( \dfrac{x_J \hbar}{x_I} \Big) \prod_{\substack{ c_I + 1 = c_J, \, \rho^\lambda_I + 1 = \rho^\lambda_J \\ ( I \leftrightarrow J ) \in \Gamma_\lambda \backslash \ft \cup \Gamma_{\bar\lambda} \backslash \bar\ft }} \theta \Big( \dfrac{x_I}{x_J} \Big) }{ \prod_{c_I = c_J, \, \rho^\lambda_i = \rho^\lambda_j +2} \theta \Big( \dfrac{x_I}{x_J \hbar} \Big) } \\ \nonumber
\\ \nonumber
\cW (\ft, \bar\ft) &:=&  \dfrac{ \theta \Big( \dfrac{a_1}{x_r}  \prod\limits_{I\in [r, \ft]} z_{c_I}^{-1} \hbar^{-\textsf{v}(I)} \Big) }{\theta \Big( \prod\limits_{I \in [r, \ft]} z_{c_I}^{-1} \hbar^{-\textsf{v}(I)} )}  \prod_{e\in \ft} \dfrac{ \theta \Big( \dfrac{x_{t(e)} \varphi^\lambda_{h(e)} }{x_{h(e)} \varphi^\lambda_{t(e)} } \prod\limits_{I\in [h(e), \ft]} z_{c_I}^{-1} \hbar^{-\textsf{v}(I)} \Big) }{ \theta \Big( \prod\limits_{I\in [h(e), \ft]} z_{c_I}^{-1} \hbar^{-\textsf{v}(I)} \Big)  } \\ \nonumber
&& \cdot  \dfrac{\theta \Big( \dfrac{a_2 \hbar }{x_{\bar r}} \prod\limits_{I \in [\bar r, \bar\ft]} z_{c_I}^{-1} \hbar^{-\textsf{v} (I)} \Big) }{ \theta \Big( \prod\limits_{I\in [\bar r, \bar\ft]} z_{c_I}^{-1} \hbar^{-\textsf{v}(I)}  \Big) }  \prod_{e\in \bar\ft} \dfrac{ \theta \Big( \dfrac{x_{t(e)} \varphi^\lambda_{h(e)} }{x_{h(e)} \varphi^\lambda_{t(e)} } \prod\limits_{I\in [h(e), \bar\ft]} z_{c_I}^{-1} \hbar^{-\textsf{v}(I)} \Big) }{ \theta \Big( \prod\limits_{I\in [h(e), \bar\ft]} z_{c_I}^{-1} \hbar^{-\textsf{v}(I)} \Big)} ,
\ee
and $\cN^\sigma$, $\cD^\sigma$, $\cR^\sigma (\ft, \bar\ft)$, $\cW^\sigma (\ft, \bar\ft)$ are the functions obtained by permuting $x_i$'s via $\sigma\in \fS_{n,k}$ in $\cN$, $\cD$, $\cR$, $\cW$. 

We would like to consider its restriction to a fixed point $\nu \supset \lambda$; in other words, to evaluate $x_I = \varphi_I^\nu$.
The symmetrization ensures that $\Stab' (\lambda)$ does not have poles for those values of $x_I$'s, and hence $\left. \Stab' (\lambda) \right|_\nu$ is well-defined. 

For an individual term such as
$$
\frac{\cN^\sigma}{\cD^\sigma}  \cR^\sigma (\ft, \bar\ft) \cW^\sigma (\ft, \bar\ft),
$$
however, its restriction to $\nu$ is not well-defined; in other words, it may depend on the order we approach the limit $x_I = \varphi_I^\nu$. We discuss these properties  in more details here. 

\begin{Lemma} \label{ND-well-defined}
{\it	The restriction to $\nu$ of
	$$
	\frac{\cN^\sigma}{\cD^\sigma}
	$$
	is well-defined, i.e., does not depend on the ordering of evaluation.} 
\end{Lemma}

\begin{proof}
	The proof is the same as Proposition 9 of \cite{EllHilb}.
\end{proof}

\begin{Lemma} \label{fix-nubar}
{\it	If
	$$
	\left. \frac{\cN^\sigma}{\cD^\sigma} \right|_\nu \neq 0,
	$$
	then $\sigma$ fixes every box in $\bar\nu$.}
\end{Lemma}

\begin{proof}
	Suppose that $\left. \dfrac{\cN^\sigma}{\cD^\sigma} \right|_\nu \neq 0$. Then by Lemma \ref{ND-well-defined}, $\cN^\sigma$ contains no factors that vanish when restricted to $\nu$. First note that $\cN^\sigma$ contains
	$$
	\prod_{c_i = n-k, \, i\neq (n-k, k)} \theta \Big( \frac{a_2 \hbar}{x_{\sigma(i)}} \Big),
	$$
	which vanishes unless $\sigma(i) \neq (n-k, k)$ for any $i\neq (n-k, k)$. Hence $\sigma (n-k, k) = (n-k, k)$.
	
	We proceed by induction on the $\rho$-values of boxes in $\bar\nu$. Assume that $\sigma$ fixes every box with $\rho \leq \rho_0$. Consider a box $(a,b)$ with $\rho(a,b) = \rho_0 +2$. Then either $(a+1, b)$ or $(a, b+1)$ lies in $\bar\nu$, and both of them have $\rho = \rho_0$. Suppose $\sigma^{-1} (a,b) \neq (a,b)$, then it is adjacent to neither $(a+1, b)$ nor $(a, b+1)$, and by induction hypothesis, $\rho_{\sigma^{-1} (a,b)} > \rho_{a+1, b}, \rho_{a, b+1}$. We see that $\cN^\sigma$ contains the factor
	$$
	\theta \Big( \frac{x_{\sigma (a+1, b)}}{x_{\sigma (\sigma^{-1} (a,b))}} \Big) = \theta \Big( \frac{x_{a+1, b}}{x_{ab}} \Big)  \qquad \text{ or } \qquad \theta \Big( \frac{x_{\sigma (\sigma^{-1} (a,b))}}{x_{\sigma (a, b+1)} \hbar} \Big) = \theta \Big( \frac{x_{ab}}{x_{a, b+1} \hbar} \Big) ,
	$$
	which vanishes at $\nu$. Hence $\sigma$ must fix $(a,b)$ and the lemma holds. 
\end{proof}

\begin{Lemma}
	{\it If
	$$
	\frac{\cN^\sigma}{\cD^\sigma} \bigg|_\nu \neq 0
	$$
	then $\sigma$ preserves the set of boxes of  $\lambda$.}
\end{Lemma}

\begin{proof}
	We proceed by induction on the diagonals. For the initial step, we need to show that the box with least content in $\lambda$, denoted by $(1,b)$, is fixed by $\sigma$. If $(1,b+1)\in \bar\nu$, then $(2,b+1) \in \bar\nu$, and $\sigma$ fixes $(1,b)$ by Lemma \ref{fix-nubar}. Now assume that $(1,b+1)\in \nu\backslash \lambda$.  Let $X_1 = (1,b+1), X_2, \cdots$ be the boxes in the diagonal of $\nu\backslash \lambda$ with one less content than $(1,b)$. Since $\rho_{X_i} < 0 < \rho_{1,b}$, by Lemma \ref{fix-nubar} we always have in $\cN^\sigma$ the factor
	$$
	\prod_{m \geq 1} \theta \Big( \frac{x_{\sigma(X_m) }}{x_{\sigma(1,b)}} \Big) = \prod_{m \geq 1} \theta \Big( \frac{x_{X_m}}{x_{\sigma(1,b)}} \Big),
	$$
	which vanishes at $\nu$ unless $\sigma (1,b)$ has no box to the left of it. This implies $\sigma (1,b) = (1,b)$. 
	
	Now assume that $\sigma$ preserves the $l$-th diagonal of $\lambda$. Consider the $(l+1)$-th diagonal. There are several cases.
	
	\begin{itemize}
		
		\item Both the $l$-th and $(l+1)$-th diagonals of $\nu\backslash\lambda$ are empty. The lemma holds trivially for $l+1$.
		
		\item $\nu\backslash \lambda$ is empty in the $l$-th diagonal, but has one box $X_1^{l+1}$ in the $(l+1)$-th diagonal.
		
		In this case, let $Y_1^l, Y_2^l, \cdots$ be boxes in the $l$-th diagonal of $\lambda$. In $\cN^\sigma$, there is the theta factor 
		$$
		\prod_{m \geq 1} \theta \Big( \frac{x_{\sigma( Y_m^l )}}{x_{\sigma (X_1^{l+1})} \hbar } \Big) = \prod_{m \geq 1} \theta \Big( \frac{x_{Y_m^l}}{x_{\sigma (X_1^{l+1})} \hbar } \Big) ,
		$$
		which vanishes at $\nu$ unless $\sigma (X_1^{l+1}) = X_1^{l+1}$. Hence $\sigma$ preserves the $(l+1)$-th diagonal of $\lambda$.
		
		\item The $l$-th diagonal of $\nu\backslash\lambda$ is nonempty.
		
		In this case, let $X_1^l, X_2^l, \cdots$ be the boxes in the $l$-th diagonal of $\nu\backslash \lambda$, and consider a general box $Y$ in the $(l+1)$-th diagonal of $\lambda$. We have in $\cN^\sigma$ the factor
		$$
		\prod_{m \geq 1} \theta \Big( \frac{x_{\sigma(X_m^l)}}{x_{\sigma (Y)}} \Big) = \prod_{m\geq 1} \theta \Big( \frac{x_{X_m^l}}{x_{\sigma (Y)}} \Big) .
		$$
		If $\sigma(Y) \not\in \lambda$, then it must be in $\nu\backslash \lambda$. Let $Z$ be the box to the left of $\sigma(Y)$, which must either also lie in $\nu\backslash \lambda$  and has to be one of those $X_i^l$'s, or lie in $\lambda$. In the former case the product vanishes at $\nu$; in the latter case we have another factor $\theta \Big( \dfrac{x_Z}{x_{\sigma(Y)}} \Big)$, which also vanishes at $\nu$.	
	\end{itemize}
	
	The lemma holds by induction.
\end{proof}

Consider the subgroups in $\fS_{n,k}$ defined as
$$
\fS_{\nu\backslash \lambda} := \{ \sigma \mid \sigma \text{ fixes each box in } \lambda \cup \bar\nu \}, \qquad \fS_{\bar\lambda} := \{ \sigma \mid \sigma \text{ fixes each box in } \lambda \}.
$$

\begin{Lemma}
{\it	If
	$$
	\frac{\cN^\sigma}{\cD^\sigma} \bigg|_\nu  \neq 0,
	$$
	then $\sigma \in \fS_{\nu\backslash \lambda}$.}
\end{Lemma}

\begin{proof}
	The proof is exactly the same as Lemma \ref{fix-nubar}, by induction on the $\rho$-values of boxes. 
\end{proof}

Now we would like to restrict the formula to the fixed point $\nu$, in a specific choice of limit. We call the following the \emph{row limit} for $\lambda$: first take 
$$
x_I = x_J
$$
for each pair of boxes $I, J\in \lambda$; then take any limit $x_I \to \varphi_I^\nu$ of the remaining variables. 

By previous lemmas, we see that only $\sigma \in \fS_{\nu \backslash\lambda}$ survives. Moreover, under the row limit, one can see that only one tree $\ft$ (which contains all rows of $\lambda$) survives, and one can write all terms independent of trees in $\lambda$:
$$
\cR^\sigma (\ft, \bar\ft) = (-1)^{m(\lambda)} \cR^\sigma (\bar\ft), \qquad \cW^\sigma (\ft, \bar\ft) = \cW^\sigma (\bar\ft),
$$
where $m(\lambda) =\sum\limits_{l \in \matZ}( \textsf{d}_{l}(\lambda)-1)$, and
\be \nonumber
\cR (\bar\ft) &:=& \dfrac{ \prod\limits_{\substack{c_I + 1 = c_J, \, \rho_I = \rho_J+1 \\ ( I \leftrightarrow J ) \in  \Gamma_{\bar\lambda} \backslash \bar\ft }} \theta \Big( \dfrac{x_J \hbar}{x_I} \Big) \prod\limits_{\substack{ c_I + 1 = c_J, \, \rho_I + 1 = \rho_J \\ ( I \leftrightarrow J ) \in  \Gamma_{\bar\lambda} \backslash \bar\ft }} \theta \Big( \dfrac{x_I}{x_J} \Big) }{ \prod\limits_{\substack{ c_I = c_J, \, \rho_I = \rho_J +2 \\ I, J \in \bar\lambda}} \theta \Big( \dfrac{x_I}{x_J \hbar} \Big) },  \\
\nonumber
\\ \nonumber
\cW ( \bar\ft) &:=&  \dfrac{\theta \Big( \dfrac{a_2 \hbar }{x_{\bar r}} \prod\limits_{I \in [\bar r, \bar\ft]} z_{c_I}^{-1} \hbar^{- \textsf{v} (I)} \Big) }{ \theta \Big( \prod\limits_{I\in [\bar r, \bar\ft]} z_{c_I}^{-1} \hbar^{-\textsf{v} (I)}  \Big)}  \prod_{e\in \bar\ft} \dfrac{ \theta \Big( \dfrac{x_{t(e)} \varphi^\lambda_{h(e)} }{x_{h(e)} \varphi^\lambda_{t(e)} } \prod\limits_{I \in [h(e), \bar\ft]} z_{c_I}^{-1} \hbar^{-\textsf{v} (I)} \Big) }{ \theta \Big( \prod\limits_{I \in [h(e), \bar\ft]} z_{c_I}^{-1} \hbar^{-\textsf{v}(I)} \Big)}. \nonumber
\ee
For $\cN^\sigma$, $\cD^\sigma$ and $\sigma \in \fS_{\bar \lambda}$, we have the factorization
$$
\frac{\cN^\sigma}{\cD^\sigma} = \frac{\cN_\lambda}{\cD_\lambda} \cdot \widetilde N^{', -}_\lambda \cdot \frac{\cN^\sigma_{\bar\lambda}}{\cD^\sigma_{\bar\lambda}},
$$
where
\be \nonumber
\Theta (\widetilde N^{', -}_\lambda) &=& (-1)^{k(n-k)-1} 
\dfrac{\prod\limits_{\substack{c_I = k \\ I \not\in \lambda}} \theta \Big( \dfrac{a_1}{x_I \hbar} \Big) \prod\limits_{\substack{c_I = n-k \\ I\in \lambda}} \theta \Big( \dfrac{a_2 \hbar}{x_I} \Big) \prod\limits_{\substack{c_I + 1 = c_J \\  I \in \lambda, \,  J \in \bar\lambda }} \theta \Big( \dfrac{x_J \hbar}{x_I} \Big)  \prod\limits_{\substack{ c_I + 1 = c_J \\  I \in \bar\lambda , \ J \in \lambda }} \theta (\dfrac{x_I}{x_J})  
}{
	 \prod\limits_{\substack{ c_I = c_J \\ I \in \lambda, \  J \in \bar\lambda }} \theta \Big( \dfrac{x_I}{x_J} \Big)  \theta \Big( \dfrac{x_I}{x_J \hbar} \Big) },  \\ \nonumber
\\ \nonumber
\cN^\sigma_{\bar\lambda} &=&  \prod_{\substack{c_I = n-k, \, i\in \bar\lambda \\ i\neq (n-k, k) } } \theta \Big( \dfrac{a_2 \hbar}{x_I} \Big) \prod_{\substack{c_I + 1 = c_J, \,  \rho_I > \rho_J \\ ( I \leftrightarrow J ) \not\in \Gamma_{\bar\lambda}, \,  I,  J \in \bar\lambda }} \theta \Big( \dfrac{x_{\sigma(J) \hbar}}{x_{\sigma(I)}} \Big) \prod_{\substack{ c_I + 1 = c_J, \,  \rho_I < \rho_J \\ ( I \leftrightarrow J ) \not\in  \Gamma_{\bar\lambda}, \, I, J \in \bar\lambda  }} \theta \Big( \frac{x_{\sigma(I)}}{x_{\sigma(J)}} \Big) \\ \nonumber
\\ \nonumber
\cN_{\lambda} &=& \prod_{\substack{ c_I = k, \,  I \in \lambda \\ I \neq (1,1) } } \theta \Big( \frac{x_I}{a_1} \Big) \prod_{\substack{c_I + 1 = c_J, \,  \rho_I > \rho_J \\ ( I \leftrightarrow J ) \not\in \Gamma_{\lambda}, \,  I,  J \in \lambda }} \theta \Big( \frac{x_J \hbar}{x_I} \Big) \prod_{\substack{ c_I + 1 = c_J, \,  \rho_I < \rho_J \\ ( I \leftrightarrow J ) \not\in  \Gamma_{\lambda}, \, I, J \in \lambda  }} \theta \Big( \frac{x_I}{x_J} \Big),
\ee
$$
\cD^\sigma_{\bar\lambda} = \prod_{\substack{ c_I = c_J, \,  \rho_I > \rho_J \\ I, J \in \bar\lambda }} \theta \Big( \frac{x_{\sigma(I)}}{x_{\sigma(J)}} \Big) \prod_{\substack{ c_I = c_J, \, \rho_I > \rho_J +2 \\ I, J \in \bar\lambda }} \theta \Big( \frac{x_{\sigma(I)}}{x_{\sigma(J)} \hbar} \Big), \qquad  \cD_{\lambda} = \prod_{\substack{ c_I = c_J, \,  \rho_I > \rho_J \\ I, J \in \lambda }} \theta \Big( \frac{x_I}{x_J} \Big) \prod_{\substack{ c_I = c_J, \, \rho_I > \rho_J +2 \\ I, J \in \lambda }} \theta \Big( \frac{x_I}{x_J \hbar} \Big). 
$$
In summary, we have the following refined formula:
\begin{Proposition} {\it \label{refin-formula}
	For any choice of limit $x_i \to \varphi_i^\nu$ for $i\in \bar\lambda$, we have
	$$
	\left. \Stab' (\lambda) \right|_\nu = \epsilon (\lambda) \, \left. \Theta (\widetilde N^{', -}_\lambda) \right|_\nu \cdot \sum_{\sigma \in \fS_{\nu \backslash \lambda} , \bar\ft} \frac{\cN^\sigma_{\bar\lambda}}{\cD^\sigma_{\bar\lambda}} \cR^\sigma (\bar\ft) \cW^\sigma (\bar\ft) \bigg|_\nu,
	$$
	where
	$$
	\epsilon (\lambda) := (-1)^{m(\lambda)} \prod_{\substack{ c_I + 1 = c_J \\ (I\leftrightarrow J) \not\in \Gamma_\lambda, \, I,J \in \lambda }} (-1) . 
	$$
	As a corollary, we have the following identity in elliptic cohomology:
	\be \label{refined}
	\begin{array}{|c|}\hline
		\\
		\ \ \ \Stab' (\lambda)  = \epsilon (\lambda) \, \Theta (\widetilde N^{', -}_\lambda) \sum\limits_{\sigma \in \fS_{\bar\lambda} , \bar\ft} \dfrac{\cN^\sigma_{\bar\lambda}}{\cD^\sigma_{\bar\lambda}} \cR^\sigma (\bar\ft) \cW^\sigma (\bar\ft) .  \ \ \\\
		\\
		\hline
	\end{array}
	\ee }
\end{Proposition}

\begin{proof}
	Computations above show that
	$$
	\left. \Stab' (\lambda) \right|_\nu = (-1)^{m(\lambda)} \frac{\cN_\lambda}{\cD_\lambda} \bigg|_\nu \cdot \Theta (\widetilde N^{', -}_\lambda) \bigg|_\nu \cdot \sum_{\sigma \in \fS_{\nu \backslash \lambda} , \bar\ft} \frac{\cN^\sigma_{\bar\lambda}}{\cD^\sigma_{\bar\lambda}} \cR^\sigma (\bar\ft) \cW^\sigma (\bar\ft) \bigg|_\nu. 
	$$
	The refined formula is proved by the following lemma. 
\end{proof}

\begin{Lemma}
	$$
	\frac{\cN_\lambda}{\cD_\lambda} \bigg|_\nu =  \prod_{\substack{ c_I + 1 = c_J \\ (I\leftrightarrow J) \not\in \Gamma_\lambda, \, I,J \in \lambda }} (-1).
	$$
\end{Lemma}

\begin{proof}
	Let $t_1 = \hbar^{-1}$, $t_2 = 1$, $x_I \mapsto x_I / a_1$ in Proposition 10 of \cite{EllHilb}. We have
	$$
	\frac{\cN_\lambda}{\cD_\lambda} \bigg|_\nu =    \frac{\cN_\lambda}{\cD_\lambda} \bigg|_\lambda =   \prod_{\substack{ c_I + 1 = c_J, \, \rho_I < \rho_J \\ (I \leftrightarrow J) \not\in \Gamma_\lambda, \, I,J \in \lambda }} (-1) \cdot \prod_{\substack{ c_I + 1 = c_J, \, \rho_I > \rho_J \\ (I \leftrightarrow J) \not\in \Gamma_\lambda, \, I,J \in \lambda }} (-1). 
	$$
\end{proof}


\section{The Mother function \label{identsec}}
\subsection{Bijection on fixed points \label{bijsec}}
Recall that the set $X^{\bT}$ consists of $n!/((n-k)! k!)$ fixed points corresponding to $k$-subsets ${\bf p}=\{{\bf p}_1,\dots,{\bf p}_k\}$ in the set ${\bf n}=\{1,2,\dots,n\}$. 
On the dual side, the set $(X')^{\bT'}$ consists of the same number of fixed points, labeled by Young diagrams $\lambda$ which fit into the rectangle $\textsf{R}_{n,k}$ with dimensions $(n-k)\times k$. There is a natural bijection
\be 
\label{pointbij}
\textsf{bj}:(X')^{\bT'} \xrightarrow{\sim} X^{\bT}
\ee
defined in the following way. 

Let $\lambda \in (X')^{\bT'}$ be a fixed point.  The boundary of the Young diagram $\lambda$ is the graph of a piecewise linear function with exactly $n$-segments. 
Clearly, we have exactly $k$-segments where this graph has slope $-1$. This way we obtain a $k$-subset in ${\bf p}\subset \{1,2,\dots,n\}$ which defines a fixed point in $X^{\bT}$. For example, consider a Young diagram $\lambda=[4,4,4,3,3,2]$ in $\textsf{R}_{10,4}$ as in the Fig.\ref{bij}. Clearly, the boundary of $\lambda$ has negative slope  at segments $4,7,9,10$, thus ${\bf p}=\{4,7,9,10\}$. 

\vspace{6mm}

\begin{figure}[h]
	\hskip 50mm
	\begin{tikzpicture}[draw=blue]
	\draw [line width=1pt] (-7,0) -- (-9,2);
	
	\draw [line width=1pt] (-9,2) -- (-6,5);
	\draw [line width=1pt] (-8,1) -- (-5,4);
	\draw [line width=1pt] (-8.5,1.5) -- (-5.5,4.5);

	\draw [line width=1pt,green] (-7,3) -- (-6,4);
	\draw [line width=1pt,green] (-9,2) -- (-7.5,3.5);
	\draw [line width=1pt,green] (-5.5,3.5) -- (-5,4);
	
	\draw [line width=1pt] (-7.5,0.5) -- (-4.5,3.5);
	\draw [line width=1pt] (-6.5,0.5) -- (-4,3);
	\draw [line width=1pt] (-7,0) -- (-5.5,1.5);

	\draw [line width=1pt] (-6.5,0.5) -- (-8.5,2.5);
	\draw [line width=1pt] (-6,1) -- (-8,3);
	\draw [line width=1pt] (-5.5,1.5) -- (-7.5,3.5);
	\draw [line width=1pt] (-5,2) -- (-7,4);
	
	\draw [line width=1pt] (-4,3) -- (-6,5);
	\draw [line width=1pt] (-4.5,2.5) -- (-6.5,4.5);
	
	\draw [line width=1pt, green] (-6,4) -- (-5.5,3.5);
	\draw [line width=1pt,green] (-4,3) -- (-5,4);
	\draw [line width=1pt,green] (-7,3) -- (-7.5,3.5);

	\draw [line width=1pt,dashed,red] (-7.25,3.25) -- (-7.25,0);
	\draw [line width=1pt,dashed,red] (-5.75,3.75) -- (-5.75,0);
	\draw [line width=1pt,dashed,red] (-4.75,3.75) -- (-4.75,0);
	\draw [line width=1pt,dashed,red] (-4.25,3.25) -- (-4.25,0);
	\draw [->] (-9,0) -- (-9,5);
	\draw [->] (-9,0) -- (-3,0);
	
	\draw [blue,fill=black] (-8.75,0) circle (0.5ex);
	\draw [blue,fill=black] (-8.25,0) circle (0.5ex); 
	\draw [blue,fill=black] (-7.75,0) circle (0.5ex);  
	\draw [blue,fill=red] (-7.25,0) circle (0.5ex);
	\draw [blue,fill=black] (-6.75,0) circle (0.5ex);	
	\draw [blue,fill=black] (-6.25,0) circle (0.5ex);
	\draw [blue,fill=red] (-5.75,0) circle (0.5ex);
	\draw [blue,fill=black] (-5.25,0) circle (0.5ex);
	\draw [blue,fill=red] (-4.75,0) circle (0.5ex);	
	\draw [blue,fill=red] (-4.25,0) circle (0.5ex);
	
	\node [left] at (-8.60,-0.25) {\begin{tiny} $1$ \end{tiny}};
	\node [left] at (-8.10,-0.25) {\begin{tiny} $2$ \end{tiny}};
	\node [left] at (-7.60,-0.25) {\begin{tiny} $3$ \end{tiny}};
	\node [left] at (-7.10,-0.25) {\begin{tiny} $4$ \end{tiny}};
	\node [left] at (-6.60,-0.25) {\begin{tiny} $5$ \end{tiny}};
	\node [left] at (-6.10,-0.25) {\begin{tiny} $6$ \end{tiny}};
	\node [left] at (-5.60,-0.25) {\begin{tiny} $7$ \end{tiny}};
	\node [left] at (-5.10,-0.25) {\begin{tiny} $8$ \end{tiny}};
	\node [left] at (-4.60,-0.25) {\begin{tiny} $9$ \end{tiny}};
	\node [left] at (-4.10,-0.25) {\begin{tiny} $10$ \end{tiny}};
	
	\end{tikzpicture}
	\caption{The point $\lambda=[4,4,4,3,3,2]\subset \textsf{R}_{10,4}$ corresponds to ${\bf p}=\{4,7,9,10\} \subset \{1,2,\dots,10\}$. \label{bij}}
\end{figure}

We note that this bijection preserves the standard dominant ordering on the set of fixed points. For instance in the case $n=4,k=2$ the fixed points on $X$ are labeled by $2$-subsets in $\{1,2,3,4\}$, which are ordered as:
$$
X^{\bT} = \{ \{1,2\} , \{1,3\} , \{1,4\} , \{2,3\} , \{2,4\} , \{3,4\} \}.
$$
The fixed points on $X'$ correspond to Young diagrams  which fit into
$2\times 2$ rectangle. The bijection above gives the following ordered list of fixed points in $X'$:
$$
(X')^{\bT'} = \{ \emptyset,[1],[1,1],[2],[2,1],[2,2] \}.
$$

\subsection{Identification of equivariant and K\"ahler parameters}
Recall that the coordinates on  the abelian variety $\widehat{\Or}_{\bf p}=\cE_{\bT}\times \cE_{\textrm{Pic}(X)}$ are the equivariant parameters $u_i/u_{i+1},\hbar$ and the K\"ahler parameter $z$. The coordinates on
$\widehat{\Or}'_{\lambda}=\cE_{\bT'}\times \cE_{\textrm{Pic}(X')}$ are the equivariant parameters $a_1/a_2,\hbar$ and K\"ahler parameters $z_1,\dots,z_{n-1}$.   Let us consider an isomorphism identifying the equivariant and K\"ahler tori on the dual sides
$$
\kappa: \bT' \rightarrow \bK , \ \  \bK' \rightarrow \bT,
$$
defined explicitly by
\be
\label{parident}
\frac{a_1}{a_2} \mapsto z \hbar^{k-1}, \qquad \hbar \mapsto \hbar^{-1}, \qquad z_i \mapsto \left\{\begin{array}{ll} 
	\dfrac{u_{i} \hbar}{u_{i+1}},&  i<k,\\ 
	\dfrac{u_{i}}{u_{i+1}},&  k \leq i\leq n-k,\\
	\dfrac{u_{i} }{u_{i+1}\hbar},&  i>n-k.   \end{array}\right.
\ee
Recall that the stability and chamber parameters for $X$ 
are defined by the following vectors:
$$
\sigma =(1,2,\dots,n) \in \Lie_\matR (\bA), \qquad \theta = (-1) \in \Lie_\matR (\bK).
$$
Using the map (\ref{parident}) we find that:
$$
d\kappa^{-1} (\sigma) = (-1,\dots, -1) = -\theta', \qquad d\kappa^{-1} (\theta) = (1) = -\sigma'.  
$$
We see that the isomorphisms $\kappa$ is chosen such that the stability parameters are matched to chamber parameters on the dual side. 

\subsection{Mother function and 3d mirror symmetry}
For the $(\bT\times \bT')$-variety $X\times X'$ we consider
equivariant embeddings defined by fixed points:
\be \label{eqma}
X=X\times\{\lambda\} \stackrel{i_{\lambda}}{ \longrightarrow} X\times X' 
\stackrel{i_{{\bf p}}}{ \longleftarrow} \{{\bf p}\}\times X' =X'
\ee
We consider $X\times\{\lambda\}$ as a $\bT\times \bT^{'}$ variety with
trivial action on the second component. This gives  
$$
\textrm{Ell}_{\bT\times \bT^{'}}(X\times\{\lambda\})=\textrm{Ell}_{\bT}(X) \times \cE_{\bT^{'}}=\textsf{E}_{\bT}(X),
$$
where in the last equality we used the isomorphism $\kappa$ to identify $\cE_{\textrm{Pic}(X)}=\cE_{\bT^{'}}$.
Similarly, 
$$
\textrm{Ell}_{\bT^{'}}(\{{\bf p}\}\times X')=\textsf{E}_{\bT^{'}}(X'). 
$$
We conclude that $\bT\times \bT^{'}$-equivariant embeddings (\ref{eqma}) induce the following maps of extended elliptic cohomologies:
$$
\textsf{E}_{\bT}(X)\stackrel{i^{*}_{\lambda}}{\longrightarrow} \textrm{Ell}_{\bT\times \bT'}(X\times X') \stackrel{i^{*}_{\bf p}}{\longleftarrow} \textsf{E}_{\bT'}(X').
$$
Here is our main result.
\begin{Theorem} \label{lbthm}
\noindent	
	
	{\it 	
\begin{itemize}	
	\item	
		There exists a line bundle ${\frak{M}}$ on $\mathrm{Ell}_{\bT\times \bT'}(X\times X')$ such that 
		$$
		(i^{*}_{\lambda})^{*}( {{\frak{M}}}) ={{\frak{M}}}({\bf p}), \qquad (i^{*}_{{\bf p}})^{*}( {{\frak{M}}})={{\frak{M}}}'({\lambda}).
		$$
		where ${\bf p}=\mathsf{bj}(\lambda)$.	
		\item 
		There exists a holomorphic section ${\frak{m}}$ (\underline{the Mother function}) of ${\frak{M}}$,  such that
		$$
		(i^{*}_{\lambda})^*({\frak{m}})={\bf Stab}({\bf p}), \qquad (i^{*}_{{\bf p}})^{*}({\frak{m}})={\bf Stab}' (\lambda),
		$$
		where ${\bf p} = \mathsf{bj} (\lambda)$
\end{itemize}		
	}		
\end{Theorem}
We will prove this theorem in Section \ref{mainproof}.  This theorem implies that 
(up to normalization by diagonal elements) the restriction matrices of elliptic stable envelopes of $X$ and $X'$ are related by transposition: 

(Similarly with notations in Definition \ref{dfelgr}, we denote $T'_{\lambda, \mu} := \left. Stab' (\lambda) \right|_{\widehat\Or'_\mu}$; we also use the simplified notation $(-)|_{\bf p}$ for $(-)|_{\widehat\Or_{\bf p}}$.)
\begin{Corollary} \label{corth}
	{ \it The restriction matrices of the elliptic stable envelopes for $X$ and $X'$
		in the basis of fixed points are related by:
		\be \label{coinc}
		T_{\bf p,p} T'_{\lambda, \mu} = T'_{\mu,\mu} T_{\bf q,p}
		\ee	
		where ${\bf p}=\mathsf{bj}(\lambda)$, ${\bf q}=\mathsf{bj}(\mu)$ and parameters are identified by (\ref{parident}). }
\end{Corollary}
\begin{proof}
	For fixed points $\lambda,\mu \in (X')^{\bT'}$, let  ${\bf p}=\textsf{bj}(\lambda)$, ${\bf q}=\textsf{bj}(\mu)$ denote the corresponding fixed points in $X^{\bT}$.	
	Note that $(X\times X')^{\bT\times \bT'}=X^{\bT}\times (X')^{\bT'}$. Let us consider the point $({\bf p},\mu)$ from this set. By Theorem \ref{lbthm}  we have
	$$
	\left.{\bf Stab}({\bf q})\right|_{\bf p}=\left.\frak{m}\right|_{({\bf p},\mu)}=\left.{\bf Stab}' (\lambda)\right|_{\bf \mu}
	$$	
	By definition (\ref{hnormx}) (\ref{holnd}) we have $\left.{\bf Stab}' (\lambda)\right|_{\bf \mu}=\Theta'_{\lambda}  T^{'}_{\lambda,\mu}$, $\left.{\bf Stab}({\bf q})\right|_{\bf p}=\Theta_{\bf q} T_{\bf q,p}$. In the standard normalization of elliptic stable envelope, the diagonal elements of the restriction matrix are given by normal bundles of repelling part of the normal bundles:
$$
T_{\bf p,p}=\Theta(N^{-}_{\bf p}), \qquad T'_{\lambda,\lambda}=\Theta(N^{'-}_{\lambda}),
$$	
with $N^{-}_{\bf p}$ and $N^{'-}_{\lambda}$ as in (\ref{repat}), (\ref{nchar}). We see that $\Theta'_{\lambda}=T_{\bf p,p}$, 	$\Theta_{\bf q}=T'_{\mu,\mu}$. 
\end{proof}
As we will see in Section \ref{proofnab}, the equality (\ref{coinc}) encodes certain infinite family of highly nontrivial identities for theta function.

\section{The Mother function in case $k=1$ \label{proofab}} 
Before we prove the Theorem \ref{lbthm} in general, it might be very instructive 
to check its prediction in the  case $k=1$. In this case the formulas for stable envelopes for $X$ and $X'$ are simple enough to compute the Mother function explicitly.  

\subsection{Explicit formula for the mother function}
In the case $k=1$ both $X$ and $X'$ are hypertoric, $X=T^{*} \mathbb{P}^{n-1}$ and $X'$ is isomorphic to the $A_{n-1}$ surface (resolution of singularity $\matC^{2}/\matZ_{n}$). The map $\kappa$ has the following form:
\be
\label{iden}
\begin{array}{l}
	\kappa: \qquad \hbar \mapsto \dfrac{1}{\hbar}, \quad \dfrac{a_1}{a_2} \mapsto z, \quad  z_1 \mapsto \dfrac{u_1}{u_2}, \quad \cdots, \quad z_{n-1} \mapsto \dfrac{u_{n-1}}{u_n}.
\end{array}
\ee
We denote by $y=y_1$ the Chern root of the tautological bundle on $X$ 
and by $x_i=x_{i,1}$, $i=1,\cdots,n-1$ the Chern roots of tautological bundles on $X'$. For symmetry, we also denote by $x_0=a_1$ and $x_n=a_2$. 
In these notations we have:

\begin{Theorem} \label{thkone}
	{\it	In the case $k=1$, the Mother function equals:
		\be \label{mothone}
		{\frak{m}}=(-1)^n \prod\limits_{i=0}^{n} \theta\Big(\frac{x_i \hbar}{x_{i-1}} u_i y\Big).
		\ee	
	}
\end{Theorem}

\subsection{Stable envelope for $X'$}
First, let us consider the elliptic stable envelopes of the fixed points in $X'$. In the case $k=1$ the fixed points on the variety $X'$ are labeled by Young diagrams inside the $1 \times (n-1)$ rectangle. There are exactly $n$ such Young diagrams $\lambda_m=[\underbrace{1,1,\dots, 1}_{m-1}]$ with $m=0,\cdots, n-1$. To compute the stable envelope of $\lambda_m$,
we need to consider trees in $\lambda_m$ and $\bar{\lambda}_{m}$. Obviously, there is only one possible tree in this case, see Fig.\ref{abdia}: 

\begin{figure}[h!]
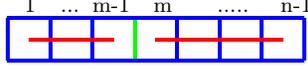

	$$
	\hspace{-30mm} \vspace{-10mm} \exona 
	$$
	\caption{\label{abdia}The tree for the fixed point representing  $[1^{m-1}]\subset \textsf{R}_{n,1}$} 
\end{figure}

\noindent
For (\ref{shenpart}) we obtain:
$$
\textbf{S}^{n,1}_{\lambda_m}=(-1)^{n-1} \theta \Big( \frac{x_1}{a_1} \Big) \theta \Big( \frac{a_2 \hbar}{x_{n-1}} \Big) \prod\limits_{i=1}^{m-2} \theta \Big( \frac{x_{i}}{x_{i+1}} \Big) \times \theta \Big( \frac{x_{m} \hbar}{x_{m-1}} \Big) \times \prod\limits_{i=m}^{n-2} \theta \Big( \frac{x_{i+1} \hbar}{x_i} \Big).
$$
To compute the K\"ahler part of the stable envelope (\ref{wpartell})
we note that $\beta^{(2)}_{\lambda_m}=0$ for all boxes of $ \textsf{R}_{n,1}$  and $\beta^{(1)}_{\lambda}$ is equal to zero for all boxes except the box $(m-1,1)$ where it is equal to $1$. Thus
$\beta_{\lambda}((i,1))=\delta_{i,m-1}$.
\ben
&& \textbf{W}^{Ell}\Big( \exonab \hspace{28mm} \vspace{-7mm} \Big) \\
&=&
\textbf{W}^{Ell}\Big(\exonabl \hspace{28mm}  \Big)\times \textbf{W}^{Ell}\Big(\exonabr \hspace{28mm} \vspace{-23mm}\Big) \\
&=& 
\phi \Big( \frac{a_1}{x_{1}},\hbar^{-1} \prod\limits_{i=1}^{m-1} z_{i}^{-1} \Big)   \prod\limits_{i=1}^{m-2} \phi \Big( \frac{x_i}{x_{i+1}},\hbar^{-1} \prod\limits_{j=i+1}^{m-1} z_{j}^{-1} \Big) \times  \phi \Big( \frac{a_2\hbar}{ x_{n-1}},\prod\limits_{i=m}^{n-1} z_{i}^{-1} \Big) 
\prod\limits_{i=m}^{n-2} \phi \Big( \dfrac{x_{i+1}\hbar}{x_{i}},\prod\limits_{j=i}^{m} z_{j}^{-1} \Big).
\een
We conclude that: 
\be \label{stabad}
\Stab' (\lambda_m) = \textbf{S}^{n,1}_{\lambda_m} \textbf{W}^{Ell}_{\lambda_m} = (-1)^n \dfrac{\prod\limits_{i=1}^{m-1} \theta \Big( \dfrac{x_{i}}{x_{i-1}}\hbar \prod\limits_{j=i}^{m-1} z_{j} \Big)
	\times \theta \Big( \dfrac{x_m}{x_{m-1}} \hbar \Big) \times  \prod\limits_{i=m+1}^{n} \theta \Big( \dfrac{x_{i}}{x_{i-1}}\hbar \prod\limits_{j=m}^{i-1} z_{j}^{-1} \Big) }
{\prod\limits_{i=1}^{m-1} \theta \Big( \hbar \prod\limits_{j=i}^{m-1} z_{j} \Big)   \times  \prod\limits_{i=m+1}^{n} \theta \Big( \prod\limits_{j=m}^{i-1} z_{j}^{-1} \Big) }
\ee
where we denote $x_0=a_1$ and $x_n=a_2$.
The restriction of stable envelope to fixed points is given by evaluation of Chern roots (\ref{fpsubs}). In this case the restriction to to $m$-th fixed point is given by:
\be
\label{fpsub}
\{ x_1=a_1, \quad \cdots,  \quad x_{m-1}=a_1, \quad  x_m=a_2 \hbar^{n-m}, \quad \cdots, \quad x_{n-1}=a_2 \hbar \}
\ee
Thus, for the diagonal matrix elements of restriction matrix we obtain:
$$
T'_{\lambda_m, \lambda_m} = \left. \Stab' (\lambda_m)\right|_{\lambda_m}= (-1)^n \theta \Big( \dfrac{a_2}{a_1} \hbar^{n-m+1} \Big).
$$
Finally, the stable envelope written in terms of parameters of $X$, i.e., all with the parameters substituted by (\ref{iden}), equals:
\be
\label{Stil}
\Stab' (\lambda_m)= (-1)^n \dfrac{\prod\limits_{i=1}^{n} \theta \Big( \dfrac{x_i}{x_{i-1} \hbar} \dfrac{u_i}{u_m} \Big) }{\prod\limits_{i=1}^{m-1} \theta \Big( \dfrac{u_i}{u_m \hbar} \Big) \prod\limits_{i=m+1}^{n} \theta \Big( \dfrac{u_i}{u_m} \Big)},
\ee
with diagonal elements of the restriction matrix:
\be
\label{Ttil}
T'_{\lambda_m, \lambda_m} =(-1)^n \theta(z^{-1} \hbar^{-n+m-1} ).
\ee

\subsection{Stable envelope for $X$}
Under the bijection of fixed points we have $\textsf{bj}(\lambda_m) = \{m\} \subset {\bf n}$. From (\ref{stabgr}) for the stable envelope of $X$ in the case $k=1$ we obtain:
\be
\label{stbx}
\Stab( m )= \prod\limits_{i=1}^{m-1} \theta \Big( \dfrac{y u_i}{\hbar} \Big) \times 
\dfrac{\theta(y u_m z^{-1} \hbar^{-n+m-1})}{\theta(z^{-1} \hbar^{-n+m-1})}\times \prod\limits_{i=m+1}^{n} \theta(y u_i).
\ee
The restriction to the $m$-th fixed point is given by substitution 
$y=u_m^{-1}$. Thus, for diagonal of restriction matrix we obtain:
\be
\label{tdia}
T_{m,m}=  \left. \Stab ( m ) \right|_{m}=  \prod\limits_{i=1}^{m-1}\theta(\frac{u_i}{u_m\hbar}) \prod\limits_{i=m+1}^{n} \theta(\frac{u_i}{u_m})
\ee

\subsection{Stable envelopes are restrictions of the Mother functions} 
We are now ready to check Theorem \ref{thkone} in the $k=1$ case. Note that (\ref{tdia}) 
gives exactly the denominator of (\ref{Stil}) and we obtain:
$$
{\bf Stab}' (\lambda_m)=T_{m,m} \Stab' (\lambda_m)= (-1)^n \prod\limits_{i=1}^{n} \theta(\frac{x_i}{x_{i-1} \hbar} \frac{u_i}{u_m} )=  \left.{\bf {\frak{m}}}\right|_m
$$
where ${\frak{m}}$ is defined by (\ref{mothone}) by $\left.{\frak{m}} \right|_m$ we denotes the restriction of this class to the $m$-th fixed point on $X$, i.e. the evaluation $y=u_m^{-1}$.  Similarly, we note that (\ref{Ttil}) is exactly the denominator of (\ref{stbx}) and we obtain:
$$
{\bf Stab}(m)=T'_{\lambda_m, \lambda_m} \Stab(m)  = (-1)^n \prod\limits_{i=1}^{m-1} \theta(\frac{y u_i}{ \hbar})\times  \theta( y u_m z^{-1} \hbar^{-n+m-1}) \times\prod\limits_{i=m+1}^{n} \theta(y u_i)= \left.{\frak{m}}\right|_{\lambda_m}
$$
where $\left.{\frak{m}}\right|_{\lambda_m}$ denoted the restriction to $\lambda_m$ on $X'$, i.e. the substitution (\ref{fpsub}) (one should not forget to substitute $\hbar\to \hbar^{-1}$ in (\ref{fpsub}), as all formulas written in terms of the  parameters of $X$).  Theorem~\ref{thkone} for $k=1$ is proven.


\section{Simplest non-abelian case $n=4, k=2$ \label{proofnab}}

\subsection{Identification of parameters and fixed points}
In the case $k=1$ considered in the previous section, the matrix elements of restriction matrices $T'_{\lambda, \mu}$ and $T_{\bf p, q}$ factorize into a product of theta functions and Theorem \ref{lbthm} can be proved by explicit computation. In contrast, when $k\geq2$ the matrix elements are much more complicated. 
In particular, Theorem \ref{lbthm} (and Corollary \ref{corth})  gives a set of  very non-trivial identities satisfied by the theta functions. In this section we consider the simplest example with $n=4$ and $k=2$.  
In this case the fixed points on $X$ are labeled by $2$-subsets in $\{1,2,3,4\}$. We consider the basis ordered as:
$$
X^{\bT} = \{ \{ 1,2 \} , \{ 1,3 \} , \{  1,4 \} , \{ 2,3 \} , \{ 2,4 \} , \{ 3,4 \} \}.
$$
The fixed points on $X'$ correspond to Young diagrams  which fit into a $2\times 2$ square. The bijection on the fixed points described in the Section \ref{bijsec} gives the corresponding points on $X'$ (in the same order):
$$
(X')^{\bT'} =\{ \emptyset ,[1],[1,1],[2],[2,1],[2,2] \}. 
$$
The identification of K\"ahler and equivariant parameters (\ref{parident}) in this case reads:
\be \label{identwo}
\kappa: \qquad \dfrac{a_1}{a_2} \mapsto z \hbar , \quad \hbar \mapsto {\hbar}^{-1}, \quad z_{{1}} \mapsto {\frac {u_{{1}}h}{u_{{2}}}}, \quad z_{{2}} \mapsto {
	\frac {u_{{2}}}{u_{{3}}}}, \quad z_{{3}} \mapsto {\frac {u_{{3}}}{u_{{4}}h}}.
\ee
We will denote a fixed point simply by  its number $m = 1,\cdots,6$. For example,
$T_{2,3}$ will denote the coefficient of the restriction matrix for $X$ given by
$T_{\{1,3\},\{1,4\}}$. Similarly, $T'_{1,4}$ denotes $T'_{\emptyset ,[2]}$ on the dual side $X'$.

\subsection{Explicit expressions for stable envelopes}
Using (\ref{stabgr}),(\ref{hnormx}),(\ref{ellipticenvelope}), (\ref{holnd}) one can compute explicit expressions for stable envelopes. We list two of them here for example (after applying $\kappa$ (\ref{identwo}) ):
$$
{\bf Stab}(6)= \dfrac{
\theta \Big( {\dfrac {y_{{1}}u_{{1}}}{\hbar}} \Big) \theta \Big( {
		\dfrac{y_{{1}}u_{{2}}}{\hbar} } \Big) \theta \Big( {\dfrac {y_{{1}}u_
			{{3}}}{z \hbar}}  \Big) \theta ( y_1 u_4 ) \theta \Big( 
		\dfrac {y_{{2}}u_{{1}}}{\hbar} \Big) \theta \Big( \dfrac{y_2 u_2}{\hbar}  \Big) \theta \Big( {\dfrac {y_{{2}}u_{{3}}}{\hbar} } \Big) 
	\theta \Big( {\dfrac {y_{{2}}u_{{4}}}{z\hbar^2} } \Big)
}{
\theta \Big( {\dfrac {y_{{1}}}{y_{{2}}}} \Big) \theta \Big( {\dfrac{y_2}{y_1 \hbar} } \Big) } 
+ ( y_1\leftrightarrow y_2 )
$$ 
\ben
{\bf Stab}' (1) &=& \dfrac{\theta \Big( \dfrac{u_4}{u_1} \Big) \theta \Big( \dfrac{u_4}{u_2} \Big) \theta \Big( \dfrac{a_2 u_3 u_4}{\hbar x_{{2,2}}u_{{2}}u_{
				{1}}} \Big) \theta \Big( \dfrac{x_{2,2} u_3}{x_{1,2} \hbar u_1} \Big) \theta \Big( \dfrac{x_{2,2} u_4}{x_{2,1} u_3} \Big) \theta \Big( \dfrac{x_{1,2} u_3}{x_{1,1} u_2} \Big) \theta \Big( \dfrac{x_{2,1}}{x_{1,1} \hbar} \Big) \theta \Big( {\dfrac {\hbar a_1}{x_{{1,1}}}} \Big) 
	\theta \Big( {\dfrac {\hbar a_1}{x_{{2,2}}}} \Big) \theta \Big( \dfrac{a_2}{x_{1,1} \hbar} \Big) 
}{
\theta \Big( \dfrac{u_3 u_4}{u_1 u_2} \Big) \theta
	\Big( {\dfrac {u_{{4}}}{u_{{3}}}} \Big) \theta \Big( {\dfrac {x_{{1
					,1}}}{x_{{2,2}}}} \Big) \theta \Big( {\dfrac {\hbar x_{{1,1}}}{x_{{2,2}}}
	} \Big)} \\
&& + \frac{\theta \Big( \dfrac{u_3}{u_1} \Big) \theta \Big( \dfrac{u_3}{u_2} \Big) \theta \Big( \dfrac{a_2 u_3 u_4}{\hbar x_{2,2} u_1 u_2}  \Big) \theta \Big( \dfrac{x_{2,2} u_4}{x_{1,2} \hbar u_1} \Big) \theta \Big( \dfrac{x_{1,2} u_4}{x_{1,1} u_2} \Big) \theta \Big( \dfrac {x_{{1,1}}
			u_{{4}}\hbar}{x_{{2,1}}u_{{3}}} \Big) \theta \Big( {\dfrac {x_{{2,2}}}{
			x_{{2,1}}}} \Big) \theta \Big( {\dfrac {\hbar a_1}{x_{{1,1}}}}
	\Big) \theta \Big( {\dfrac {\hbar a_1}{x_{{2,2}}}} \Big) \theta
	\Big( \dfrac {a_2}{\hbar x_{{1,1}}} \Big) 
}{
\theta \Big( \dfrac{u_3 u_4}{u_1 u_2}  \Big) \theta
	\Big( {\dfrac {u_{{4}}}{u_{{3}}}} \Big) \theta \Big( {\dfrac {x_{{1
					,1}}}{x_{{2,2}}}} \Big) \theta \Big( {\dfrac {\hbar x_{{1,1}}}{x_{{2,2}}}
	} \Big)} \\
&& +  \left( {x_{1,1}\leftrightarrow x_{2,2}} \right), 
\een
where we denote $x_0=a_1,x_m=a_2$.

\subsection{Theorem \ref{lbthm} in case $n=4,k=2$}
The Corollary \ref{corth1} means that the functions above are related by the following identities:
$$
\left.{\bf Stab}(a)\right|_{b}=\left.{\bf Stab}' (b)\right|_{a}, 
$$
where the restriction to the fixed points on $X$ is given by substitution of 
variables $y_i$ (\ref{eval}). The restrictions to the fixed points on $X'$
are defined by (\ref{fpsubs}) (together with identification of parameters (\ref{identwo})!). We only compute non-zero restrictions and only those $\left.{\bf Stab}(a)\right|_{b}$ with $a\neq b$ (the case $a=b$ is trivial). 
\\

For example: 
\begin{small}	
	
\noindent 
\underline{Case $a=2, b=1$:}	
$$
\left.{\bf Stab}' (1)\right|_{2}= \theta \left( z{\hbar}^{3} \right) \theta \left( \hbar \right) \theta \left( {
	\frac {u_{{1}}}{u_{{3}}}} \right) \theta \left( {\frac {zu_{{2}}{\hbar}^{3
	}}{u_{{3}}}} \right) \theta \left( {\frac {u_{{1}}}{u_{{4}}}} \right) 
\theta \left( {\frac {u_{{2}}}{u_{{4}}}} \right),
$$
$$
\left.{\bf Stab}(2)\right|_{1}= \theta \left( z{\hbar}^{3} \right) \theta \left( \hbar \right) \theta \left( {
	\frac {u_{{1}}}{u_{{3}}}} \right) \theta \left( {\frac {zu_{{2}}{\hbar}^{3
	}}{u_{{3}}}} \right) \theta \left( {\frac {u_{{1}}}{u_{{4}}}} \right) 
\theta \left( {\frac {u_{{2}}}{u_{{4}}}} \right).
$$
\end{small}

We see that for $(a,b) = (2,1)$ the two are trivially equal as product of theta functions, which also happens in cases $(a,b) = (3,2), (4,2), (5,2), (5,3), (4,3), (6,3), (5,4), (6,5)$. However, the identity is nontrivial for the remaining cases $(a,b)=(3,1),(4,1),(5,1),(6,1),(6,2),(6,4)$. 
\\

\begin{small}

\noindent 
\underline{Case $a=3, b=1$:}

$$
\left.{\bf Stab}' (1)\right|_{3}=\frac{\theta \left( \hbar \right) \theta \left( {\frac {u_{{1}}}{u_{{4}}}}
	\right) \theta \left( {\frac {u_{{1}}}{u_{{3}}}} \right)  \left( 
	\theta \left( {\frac {zu_{{3}}{\hbar}^{2}}{u_{{4}}}} \right) \theta
	\left( {\frac {z u_{{2}}{\hbar}^{3}}{u_{{3}}}} \right) \theta \left( {
		\frac {u_{{2}}}{u_{{4}}}} \right) \theta \left( \hbar \right) -\theta
	\left( {\frac {u_{{2}}}{u_{{3}}}} \right) \theta \left( {\frac {hu_{{
					4}}}{u_{{3}}}} \right) \theta \left( z{\hbar}^{2} \right) \theta \left( {
		\frac {z u_{{2}}{\hbar}^{3}}{u_{{4}}}} \right)  \right)}{\theta \left( {\frac {u_{{3}}}{u_{{4}}}} \right)},
$$
$$
\left.{\bf Stab}(3)\right|_{1}=\theta \left( {\frac {u_{{1}}}{u_{{3}}}} \right) \theta \left( {\frac 
	{u_{{1}}}{u_{{4}}}} \right) \theta \left( h \right) \theta \left( {
	\frac {\hbar u_{{2}}}{u_{{3}}}} \right) \theta \left( {\frac {z{\hbar}^{2}u_{{2
	}}}{u_{{4}}}} \right) \theta \left( z{\hbar}^{3} \right).
$$
\noindent 
\underline{Case $a=4, b=1$:}
$$
\left.{\bf Stab}' (1)\right|_{4}=\theta \left( {\frac {u_{{1}}}{u_{{4}}}} \right) \theta \left( {\frac 
	{u_{{2}}}{u_{{4}}}} \right) \theta \left( \hbar \right) \theta \left( {
	\frac {\hbar u_{{2}}}{u_{{3}}}} \right) \theta \left( z{\hbar}^{3} \right) 
\theta \left( {\frac {z{\hbar}^{2}u_{{1}}}{u_{{3}}}} \right),
$$
$$
\left.{\bf Stab}(4)\right|_{1}=\frac{\theta \left( \hbar \right) \theta \left( {\frac {u_{{1}}}{u_{{4}}}}
	\right) \theta \left( {\frac {u_{{2}}}{u_{{4}}}} \right)  \left( 
	\theta \left( {\frac {u_{{1}}}{u_{{3}}}} \right) \theta \left( {\frac 
		{z{\hbar}^{2}u_{{1}}}{u_{{2}}}} \right) \theta \left( \hbar \right) \theta
	\left( {\frac {zu_{{2}}{\hbar}^{3}}{u_{{3}}}} \right) -\theta \left( {
		\frac {u_{{2}}}{u_{{3}}}} \right) \theta \left( {\frac {\hbar u_{{2}}}{u_{{
					1}}}} \right) \theta \left( {\frac {zu_{{1}}{\hbar}^{3}}{u_{{3}}}}
	\right) \theta \left( z{\hbar}^{2} \right)  \right)
}{\theta \left( {\frac {u_{{1}}}{u_{{2}}}} \right)}.
$$
\noindent 
\underline{Case $a=5, b=1$:}
$$
\left.{\bf Stab}' (1)\right|_{5}= \frac{\theta \left( \hbar \right)  \left( -\theta \left( {\frac {z{\hbar}^{2}u_{{1}}
		}{u_{{4}}}} \right) \theta \left( {\frac {\hbar u_{{2}}}{u_{{4}}}} \right) 
	\theta \left( {\frac {\hbar u_{{4}}}{u_{{3}}}} \right) \theta \left( z{\hbar}^{
		2} \right) \theta \left( {\frac {u_{{1}}}{u_{{3}}}} \right) \theta
	\left( {\frac {u_{{2}}}{u_{{3}}}} \right) +\theta \left( {\frac {z{h}
			^{2}u_{{1}}}{u_{{3}}}} \right) \theta \left( {\frac {\hbar u_{{2}}}{u_{{3}}
	}} \right) \theta \left( {\frac {zu_{{3}}{\hbar}^{2}}{u_{{4}}}} \right) 
	\theta \left( {\frac {u_{{1}}}{u_{{4}}}} \right) \theta \left( {\frac 
		{u_{{2}}}{u_{{4}}}} \right) \theta \left( \hbar \right)  \right) 
}{\theta \left( {\frac {u_{{3}}}{u_{{4}}}} \right) },
$$
$$
\left.{\bf Stab}(5)\right|_{1}= \frac{\theta \left( \hbar \right)  \left( \theta \left( {\frac {u_{{1}}}{u_{{3}}
	}} \right) \theta \left( {\frac {u_{{1}}}{u_{{4}}}} \right) \theta
	\left( {\frac {z{h}^{2}u_{{1}}}{u_{{2}}}} \right) \theta \left( {
		\frac {\hbar u_{{2}}}{u_{{3}}}} \right) \theta \left( {\frac {z{\hbar}^{2}u_{{2
		}}}{u_{{4}}}} \right) \theta \left( \hbar \right) -\theta \left( {\frac {u
			_{{2}}}{u_{{3}}}} \right) \theta \left( {\frac {u_{{2}}}{u_{{4}}}}
	\right) \theta \left( {\frac {\hbar u_{{1}}}{u_{{3}}}} \right) \theta
	\left( {\frac {z{\hbar}^{2}u_{{1}}}{u_{{4}}}} \right) \theta \left( {
		\frac {\hbar u_{{2}}}{u_{{1}}}} \right) \theta \left( z{\hbar}^{2} \right) 
	\right)
}{\theta \left( {\frac {u_{{1}}}{u_{{2}}}} \right)}.
$$
\noindent 
\underline{Case $a=6, b=1$:}
\ben
\left.{\bf Stab}' (1)\right|_{6} &=& \dfrac{1}{\theta \left( {\frac {u_{{1}}u_{{2}}}{u_{{3}}u_{{4}}}} \right) \theta
	\left( {\frac {u_{{3}}}{u_{{4}}}} \right)}\left(  
\theta \left( {\frac {\hbar u_{{2}}}{u_{{3}}}} \right) \theta
\left( {
	\frac {\hbar u_{{1}}}{u_{{3}}}} \right) \theta \left( {\frac {u_{{3}}\hbar}{u_{
			{4}}}} \right) \theta \left( z{\hbar}^{2} \right) \theta \left( {\frac {z
		u_{{1}}u_{{2}}}{u_{{3}}u_{{4}}}} \right) \theta \left( {\frac {u_{{1}}
	}{u_{{4}}}} \right) \theta \left( {\frac {u_{{2}}}{u_{{4}}}} \right) 
\theta \left( \hbar \right)\right. \\
&& - \theta \left( {\frac {\hbar u_{{2}}}{u_{{4}}}}
\right) \theta \left( {\frac {\hbar u_{{1}}}{u_{{4}}}} \right) \theta
\left( {\frac {\hbar u_{{4}}}{u_{{3}}}} \right) \theta \left( z{\hbar}^{2}
\right) \theta \left( \hbar \right) \theta \left( {\frac {zhu_{{1}}u_{{2}
	}}{u_{{3}}u_{{4}}}} \right) \theta \left( {\frac {u_{{1}}}{u_{{3}}}}
\right) \theta \left( {\frac {u_{{2}}}{u_{{3}}}} \right) \\ 
&& - \left. \theta
\left( {\frac {u_{{2}}}{u_{{3}}}} \right) \theta \left( {\frac {z{\hbar}^
		{2}u_{{2}}u_{{1}}}{u_{{3}}u_{{4}}}} \right) \theta \left( {\hbar}^{2}
\right) \theta \left( {\frac {u_{{1}}}{u_{{3}}}} \right) \theta
\left( {\frac {u_{{3}}}{u_{{4}}}} \right) \theta \left( z \hbar \right) 
\theta \left( {\frac {u_{{1}}}{u_{{4}}}} \right) \theta \left( {\frac 
	{u_{{2}}}{u_{{4}}}} \right) 
\right),
\een
$$
\left.{\bf Stab}(6)\right|_{1}=\frac{  \theta \left( \hbar \right)^{2} \left( \theta \left( {
		\frac {u_{{1}}}{u_{{4}}}} \right) \theta \left( {\frac {\hbar u_{{1}}}{u_{{
					2}}}} \right) \theta \left( {\frac {z\hbar u_{{1}}}{u_{{3}}}} \right) 
	\theta \left( {\frac {\hbar u_{{2}}}{u_{{3}}}} \right) \theta \left( {
		\frac {z{\hbar}^{2}u_{{2}}}{u_{{4}}}} \right) -\theta \left( {\frac {u_{{2
		}}}{u_{{4}}}} \right) \theta \left( {\frac {z\hbar u_{{2}}}{u_{{3}}}}
	\right) \theta \left( {\frac {\hbar u_{{1}}}{u_{{3}}}} \right) \theta
	\left( {\frac {z{\hbar}^{2}u_{{1}}}{u_{{4}}}} \right) \theta \left( {
		\frac {\hbar u_{{2}}}{u_{{1}}}} \right)  \right)
}{\theta \left( {\frac {u_{{1}}}{u_{{2}}}} \right)}.
$$

\noindent 
\underline{Case $a=6, b=2$:}
$$
\left.{\bf Stab}' (2)\right|_{6}=\theta \left( {\frac {u_{{3}}\hbar}{u_{{2}}}} \right) \theta \left( \hbar
\right) \theta \left( {\frac {\hbar u_{{1}}}{u_{{2}}}} \right) \theta
\left( {\frac {z\hbar u_{{1}}}{u_{{4}}}} \right) \theta \left( z{\hbar}^{2}
\right) \theta \left( {\frac {u_{{3}}\hbar}{u_{{4}}}} \right),
$$
$$
\left.{\bf Stab}(6)\right|_{2}=\frac{\theta \left( \hbar \right) \theta \left( {\frac {\hbar u_{{1}}}{u_{{2}}}}
	\right) \theta \left( {\frac {u_{{3}}\hbar}{u_{{2}}}} \right)  \left( 
	\theta \left( {\frac {u_{{1}}}{u_{{4}}}} \right) \theta \left( {\frac 
		{z\hbar u_{{1}}}{u_{{3}}}} \right) \theta \left( {\frac {zu_{{3}}{\hbar}^{2}}{u
			_{{4}}}} \right) \theta \left( \hbar \right) -\theta \left( z\hbar \right) 
	\theta \left( {\frac {z{\hbar}^{2}u_{{1}}}{u_{{4}}}} \right) \theta
	\left( {\frac {u_{{3}}\hbar}{u_{{1}}}} \right) \theta \left( {\frac {u_{{
					3}}}{u_{{4}}}} \right)  \right)
}{\theta \left( {\frac {u_{{1}}}{u_{{3}}}} \right).} 
$$

\noindent 
\underline{Case $a=6, b=4$:}
$$
\left.{\bf Stab}' (4)\right|_{6}=\theta \left( {\frac {u_{{3}}\hbar}{u_{{1}}}} \right) \theta \left( {
	\frac {\hbar u_{{2}}}{u_{{1}}}} \right) \theta \left( {\frac {z\hbar u_{{2}}}{u_
		{{4}}}} \right) \theta \left( {\frac {u_{{3}}\hbar}{u_{{4}}}} \right) 
\theta \left( \hbar \right) \theta \left( z{\hbar}^{2} \right),
$$
$$
\left.{\bf Stab}(6)\right|_{4}=\frac{\theta \left( \hbar \right) \theta \left( {\frac {\hbar u_{{2}}}{u_{{1}}}}
	\right) \theta \left( {\frac {u_{{3}}\hbar}{u_{{1}}}} \right)  \left( 
	\theta \left( {\frac {u_{{2}}}{u_{{4}}}} \right) \theta \left( {\frac 
		{z\hbar u_{{2}}}{u_{{3}}}} \right) \theta \left( {\frac {zu_{{3}}{\hbar}^{2}}{u
			_{{4}}}} \right) \theta \left( \hbar \right) -\theta \left( z\hbar \right) 
	\theta \left( {\frac {z{\hbar}^{2}u_{{2}}}{u_{{4}}}} \right) \theta
	\left( {\frac {u_{{3}}\hbar}{u_{{2}}}} \right) \theta \left( {\frac {u_{{
					3}}}{u_{{4}}}} \right)  \right)}{
	\theta \left( {\frac {u_{{2}}}{u_{{3}}}} \right). }
$$
\end{small}

\subsection{Identities for theta functions}

In all these cases the identity follows from the well-known $3$-term identity
\be \label{treeeterm}
\theta \Big( \dfrac{a y_1}{x} \Big) \theta \Big( \dfrac{h y_2}{x} \Big) \theta \Big( \dfrac{h y_1}{y_2} \Big) \theta (a)=\theta \Big( \dfrac{a h y_2}{x} \Big) \theta \Big( \dfrac{y_2}{x} \Big) \theta \Big( \dfrac{y_1}{y_2} \Big) \theta \Big( \dfrac{a}{h} \Big)+\theta \Big( \dfrac{h y_2}{x} \Big) \theta \Big( \dfrac{a y_2}{x} \Big) \theta (h)
\theta \Big( \dfrac{a y_1}{y_2} \Big), 
\ee
and $4$-term identity for theta functions:
\be \label{fourtermr}
\begin{array}{l}
	\theta \left( h \right) \theta \Big( {\dfrac {y_{{1}}}{y_{{2}}}}
	\Big) \theta \Big( {\dfrac {hy_{{1}}}{x_{{1}}}} \Big) \theta
	\Big( {\dfrac {a_{{2}}hy_{{2}}}{x_{{1}}}} \Big) \theta \Big( {
		\dfrac {a_{{1}}a_{{2}}hy_{{1}}}{x_{{2}}}} \Big) \theta \Big( {\dfrac 
		{x_{{2}}}{y_{{2}}}} \Big) \theta \Big( {\dfrac {a_{{1}}x_{{2}}}{x_{{
					1}}}} \Big)\\
	\\ -\theta \Big( {\dfrac {a_{{1}}a_{{2}}hy_{{1}}}{x_{{1}}}}
	\Big) \theta \Big( {\dfrac {x_{{1}}}{y_{{2}}}} \Big) \theta
	\Big( {\dfrac {hy_{{1}}}{x_{{2}}}} \Big) \theta \Big( {\dfrac {a_{{
					2}}hy_{{2}}}{x_{{2}}}} \Big) \theta \Big( {\dfrac {hx_{{2}}}{x_{{1}}
	}} \Big) \theta \Big( {\dfrac {y_{{1}}}{y_{{2}}}} \Big) \theta
	\left( a_{{1}} \right) \\
	\\ 
	=-\theta \left( h \right) \theta \Big( {\dfrac {x_{{1}}}{x_{{2}}}}
	\Big) \theta \Big( {\dfrac {a_{{1}}a_{{2}}hy_{{2}}}{x_{{1}}}}
	\Big) \theta \Big( {\dfrac {a_{{2}}hy_{{1}}}{x_{{2}}}} \Big) 
	\theta \Big( {\dfrac {a_{{1}}y_{{1}}}{y_{{2}}}} \Big) \theta \Big( 
	{\dfrac {hy_{{1}}}{x_{{1}}}} \Big) \theta \Big( {\dfrac {x_{{2}}}{y_{
				{2}}}} \Big) \\
	\\
	+\theta \Big( {\dfrac {hy_{{2}}}{x_{{1}}}} \Big) 
	\theta \Big( {\dfrac {x_{{2}}}{y_{{1}}}} \Big) \theta \Big( {\dfrac 
		{x_{{1}}}{x_{{2}}}} \Big) \theta \Big( {\dfrac {hy_{{1}}}{y_{{2}}}}
	\Big) \theta \Big( {\dfrac {a_{{1}}a_{{2}}hy_{{1}}}{x_{{1}}}}
	\Big) \theta \Big( {\dfrac {a_{{2}}hy_{{2}}}{x_{{2}}}} \Big) 
	\theta \left( a_{{1}} \right)
\end{array}
\ee
Let us check the identity for the most complicated case $a=6,b=1$. The other cases are analyzed in the same manner.  First, we specialize the parameters in the 4-term relation
(\ref{fourtermr}) to the following values:
$$
\left\{ a_{{1}}={\hbar}^{-1}, \quad a_{{2}}=z \hbar , \quad x_{{1}}=u_{{3}}, \quad x_{{2}}=u_{{4}}, \quad y_{{1}}=u_{{2}}, \quad y_{{2}}=u_{{1}}, \quad h=\hbar \right\}. 
$$
After this substitution the above 4-term (up to a common multiple $\theta(\hbar)$) takes the form:
\be
\label{fourterm}
\begin{array}{l}
	-\theta \left( {\frac {u_{{1}}}{u_{{4}}}} \right) \theta \left( {
		\frac {\hbar u_{{3}}}{u_{{4}}}} \right) \theta \left( {\frac {u_{{1}}}{u_{{
					2}}}} \right) \theta \left( {\frac {zu_{{1}}{\hbar}^{2}}{u_{{3}}}}
	\right) \theta \left( {\frac {z\hbar u_{{2}}}{u_{{4}}}} \right) \theta
	\left( {\frac {\hbar u_{{2}}}{u_{{3}}}} \right) +\theta \left( {\frac {u_{
				{1}}}{u_{{2}}}} \right) \theta \left( {\frac {\hbar u_{{2}}}{u_{{4}}}}
	\right) \theta \left( {\frac {\hbar u_{{4}}}{u_{{3}}}} \right) \theta
	\left( {\frac {u_{{1}}}{u_{{3}}}} \right) \theta \left( {\frac {zhu_{
				{2}}}{u_{{3}}}} \right) \theta \left( {\frac {zu_{{1}}{\hbar}^{2}}{u_{{4}}
	}} \right) \\
	\\
	= -\theta \left( {\frac {u_{{1}}}{u_{{4}}}} \right) \theta \left( {
		\frac {u_{{3}}}{u_{{4}}}} \right) \theta \left( {\frac {\hbar u_{{1}}}{u_{{
					2}}}} \right) \theta \left( {\frac {z\hbar u_{{1}}}{u_{{3}}}} \right) 
	\theta \left( {\frac {z{\hbar}^{2}u_{{2}}}{u_{{4}}}} \right) \theta
	\left( {\frac {\hbar u_{{2}}}{u_{{3}}}} \right) +\theta \left( {\frac {u_{
				{3}}}{u_{{4}}}} \right) \theta \left( {\frac {u_{{2}}}{u_{{4}}}}
	\right) \theta \left( {\frac {\hbar u_{{1}}}{u_{{3}}}} \right) \theta
	\left( {\frac {\hbar u_{{2}}}{u_{{1}}}} \right) \theta \left( {\frac {zhu_
			{{2}}}{u_{{3}}}} \right) \theta \left( {\frac {zu_{{1}}{\hbar }^{2}}{u_{{4}
	}}} \right)
\end{array}
\ee
Now, the identity for $a=6,b=1$ has the form:
$$
A_1+A_2+A_3=B_1+B_2
$$
where the terms have the following explicit form (after clearing the denominators):
$$
\begin{array}{l}
A_1=\theta \left( {\frac {u_{{1}}}{u_{{2}}}} \right) \theta \left( {\frac 
	{\hbar u_{{2}}}{u_{{3}}}} \right) \theta \left( {\frac {\hbar u_{{1}}}{u_{{3}}}}
\right) \theta \left( {\frac {\hbar u_{{3}}}{u_{{4}}}} \right) \theta
\left( z{h}^{2} \right) \theta \left( {\frac {z\hbar u_{{1}}u_{{2}}}{u_{{3
		}}u_{{4}}}} \right) \theta \left( {\frac {u_{{1}}}{u_{{4}}}} \right) 
\theta \left( {\frac {u_{{2}}}{u_{{4}}}} \right) \theta \left( \hbar
\right)\\
\\
A_2=-\theta \left( {\frac {u_{{1}}}{u_{{2}}}} \right) \theta \left( {
	\frac {\hbar u_{{2}}}{u_{{4}}}} \right) \theta \left( {\frac {\hbar u_{{1}}}{u_{
			{4}}}} \right) \theta \left( {\frac {\hbar u_{{4}}}{u_{{3}}}} \right) 
\theta \left( z{\hbar}^{2} \right) \theta \left( \hbar \right) \theta \left( {
	\frac {z\hbar u_{{1}}u_{{2}}}{u_{{3}}u_{{4}}}} \right) \theta \left( {
	\frac {u_{{1}}}{u_{{3}}}} \right) \theta \left( {\frac {u_{{2}}}{u_{{3
}}}} \right)
\\
\\
A_3=-\theta \left( {\frac {u_{{1}}}{u_{{2}}}} \right) \theta \left( {
	\frac {u_{{2}}}{u_{{3}}}} \right) \theta \left( {\frac {z{\hbar}^{2}u_{{2}
		}u_{{1}}}{u_{{3}}u_{{4}}}} \right) \theta \left( {\hbar}^{2} \right) 
\theta \left( {\frac {u_{{1}}}{u_{{3}}}} \right) \theta \left( {\frac 
	{u_{{3}}}{u_{{4}}}} \right) \theta \left( z\hbar \right) \theta \left( {
	\frac {u_{{1}}}{u_{{4}}}} \right) \theta \left( {\frac {u_{{2}}}{u_{{4
}}}} \right)\\
\\
B_1=\theta \left( \hbar \right) ^{2}\theta \left( {\frac {u_{
			{1}}u_{{2}}}{u_{{3}}u_{{4}}}} \right) \theta \left( {\frac {u_{{3}}}{u
		_{{4}}}} \right) \theta \left( {\frac {u_{{1}}}{u_{{4}}}} \right) 
\theta \left( {\frac {\hbar u_{{1}}}{u_{{2}}}} \right) \theta \left( {
	\frac {z\hbar u_{{1}}}{u_{{3}}}} \right) \theta \left( {\frac {\hbar u_{{2}}}{u_
		{{3}}}} \right) \theta \left( {\frac {z{\hbar}^{2}u_{{2}}}{u_{{4}}}}
\right) \\
\\
B_2=-\theta \left( \hbar \right) ^{2}\theta \left( {\frac {u_
		{{1}}u_{{2}}}{u_{{3}}u_{{4}}}} \right) \theta \left( {\frac {u_{{3}}}{
		u_{{4}}}} \right) \theta \left( {\frac {u_{{2}}}{u_{{4}}}} \right) 
\theta \left( {\frac {z\hbar u_{{2}}}{u_{{3}}}} \right) \theta \left( {
	\frac {\hbar u_{{1}}}{u_{{3}}}} \right) \theta \left( {\frac {zu_{{1}}{\hbar }^{
			2}}{u_{{4}}}} \right) \theta \left( {\frac {\hbar u_{{2}}}{u_{{1}}}}
\right) 
\end{array}
$$
For some values of the parameters the three term relation (\ref{treeeterm}) can be written in 
the form:
$$
\begin{array}{l}
\theta \left( z{\hbar}^{2} \right) \theta \left( {\frac {u_{{2}}}{u_{{4}}}
} \right) \theta \left( {\frac {\hbar u_{{1}}}{u_{{3}}}} \right) \theta
\left( {\frac {z\hbar u_{{1}}u_{{2}}}{u_{{3}}u_{{4}}}} \right) =-\theta
\left( {\frac {\hbar u_{{4}}}{u_{{2}}}} \right) \theta \left( {\frac {z{h}
		^{2}u_{{2}}u_{{1}}}{u_{{3}}u_{{4}}}} \right) \theta \left( {\frac {u_{
			{1}}}{u_{{3}}}} \right) \theta \left( z\hbar \right) +\theta \left( {
	\frac {u_{{1}}u_{{2}}}{u_{{3}}u_{{4}}}} \right) \theta \left( h
\right) \theta \left( {\frac {z\hbar u_{{2}}}{u_{{4}}}} \right) \theta
\left( {\frac {zu_{{1}}{\hbar}^{2}}{u_{{3}}}} \right)
\end{array}
$$
and thus for $A_1$ we can write:
\begin{small}
\ben
A_1 &=& -\theta \left( {\frac {\hbar u_{{4}}}{u_{{2}}}} \right) \theta \left( {
	\frac {u_{{1}}}{u_{{2}}}} \right) \theta \left( {\frac {\hbar u_{{2}}}{u_{{
				3}}}} \right) \theta \left( {\frac {\hbar u_{{3}}}{u_{{4}}}} \right) \theta
\left( {\frac {u_{{1}}}{u_{{4}}}} \right) \theta \left( \hbar \right) 
\theta \left( {\frac {u_{{1}}}{u_{{3}}}} \right) \theta \left( {\frac 
	{z{\hbar}^{2}u_{{2}}u_{{1}}}{u_{{3}}u_{{4}}}} \right) \theta \left( z\hbar
\right) \\
&& + \theta \left( {\frac {zu_{{1}}{\hbar}^{2}}{u_{{3}}}} \right) 
\theta \left( {\frac {z\hbar u_{{2}}}{u_{{4}}}} \right) \theta \left( {
	\frac {\hbar u_{{2}}}{u_{{3}}}} \right) \theta \left( {\frac {u_{{1}}}{u_{{
				4}}}} \right) \theta \left( \hbar \right)^{2}\theta
\left( {\frac {u_{{1}}}{u_{{2}}}} \right) \theta \left( {\frac {u_{{1
		}}u_{{2}}}{u_{{3}}u_{{4}}}} \right) \theta \left( {\frac {\hbar u_{{3}}}{u_
		{{4}}}} \right) . 
\een
\end{small}
Similarly we can write the 3-term relation as:
$$
\begin{array}{l}
\theta \left( z{\hbar}^{2} \right) \theta \left( {\frac {u_{{2}}}{u_{{3}}}
} \right) \theta \left( {\frac {\hbar u_{{1}}}{u_{{4}}}} \right) \theta
\left( {\frac {z\hbar u_{{1}}u_{{2}}}{u_{{3}}u_{{4}}}} \right) =-\theta
\left( {\frac {\hbar u_{{3}}}{u_{{2}}}} \right) \theta \left( {\frac {z{\hbar}
		^{2}u_{{2}}u_{{1}}}{u_{{3}}u_{{4}}}} \right) \theta \left( {\frac {u_{
			{1}}}{u_{{4}}}} \right) \theta \left( zh \right) +\theta \left( {
	\frac {u_{{1}}u_{{2}}}{u_{{3}}u_{{4}}}} \right) \theta \left( h
\right) \theta \left( {\frac {z\hbar u_{{2}}}{u_{{3}}}} \right) \theta
\left( {\frac {zu_{{1}}{\hbar}^{2}}{u_{{4}}}} \right)
\end{array}
$$
and thus:
\begin{small}
\ben
A_2 &=& \theta \left( {\frac {\hbar u_{{3}}}{u_{{2}}}} \right) \theta \left( {
	\frac {u_{{1}}}{u_{{2}}}} \right) \theta \left( {\frac {u_{{1}}}{u_{{4
}}}} \right) \theta \left( \hbar \right) \theta \left( {\frac {\hbar u_{{2}}}{u
		_{{4}}}} \right) \theta \left( {\frac {\hbar u_{{4}}}{u_{{3}}}} \right) 
\theta \left( {\frac {u_{{1}}}{u_{{3}}}} \right) \theta \left( {\frac 
	{z{\hbar}^{2}u_{{2}}u_{{1}}}{u_{{3}}u_{{4}}}} \right) \theta \left( zh
\right) \\
&&  -\theta \left( {\frac {u_{{1}}}{u_{{2}}}} \right)  
\theta \left( \hbar \right)^{2}\theta \left( {\frac {\hbar u_{{2}}}{u
		_{{4}}}} \right) \theta \left( {\frac {\hbar u_{{4}}}{u_{{3}}}} \right) 
\theta \left( {\frac {u_{{1}}}{u_{{3}}}} \right) \theta \left( {\frac 
	{u_{{1}}u_{{2}}}{u_{{3}}u_{{4}}}} \right) \theta \left( {\frac {z\hbar u_{{
				2}}}{u_{{3}}}} \right) \theta \left( {\frac {zu_{{1}}{\hbar }^{2}}{u_{{4}}}
} \right)  
\een
\end{small}
Finally,
$$
\begin{array}{l}
\theta \left( {\hbar}^{2} \right) \theta \left( {\frac {u_{{2}}}{u_{{3}}}}
\right) \theta \left( {\frac {u_{{2}}}{u_{{4}}}} \right) \theta
\left( {\frac {u_{{3}}}{u_{{4}}}} \right) =\theta \left( \hbar \right) 
\theta \left( {\frac {\hbar u_{{2}}}{u_{{4}}}} \right) \theta \left( {
	\frac {\hbar u_{{4}}}{u_{{3}}}} \right) \theta \left( {\frac {\hbar u_{{3}}}{u_{
			{2}}}} \right) -\theta \left( {\frac {\hbar u_{{4}}}{u_{{2}}}} \right) 
\theta \left( {\frac {\hbar u_{{2}}}{u_{{3}}}} \right) \theta \left( \hbar
\right) \theta \left( {\frac {\hbar u_{{3}}}{u_{{4}}}} \right) 
\end{array}
$$
which gives:
\begin{small}
\ben
A_3 &=& \theta \left( {\frac {\hbar u_{{4}}}{u_{{2}}}} \right) \theta \left( {
	\frac {u_{{1}}}{u_{{2}}}} \right) \theta \left( {\frac {\hbar u_{{2}}}{u_{{
				3}}}} \right) \theta \left( {\frac {\hbar u_{{3}}}{u_{{4}}}} \right) \theta
\left( {\frac {u_{{1}}}{u_{{4}}}} \right) \theta \left( \hbar \right) 
\theta \left( {\frac {u_{{1}}}{u_{{3}}}} \right) \theta \left( {\frac 
	{z{\hbar}^{2}u_{{2}}u_{{1}}}{u_{{3}}u_{{4}}}} \right) \theta \left( zh
\right) \\
&& - \theta \left( {\frac {\hbar u_{{3}}}{u_{{2}}}} \right) \theta
\left( {\frac {u_{{1}}}{u_{{2}}}} \right) \theta \left( {\frac {u_{{1
	}}}{u_{{4}}}} \right) \theta \left( \hbar \right) \theta \left( {\frac {\hbar u
		_{{2}}}{u_{{4}}}} \right) \theta \left( {\frac {\hbar u_{{4}}}{u_{{3}}}}
\right) \theta \left( {\frac {u_{{1}}}{u_{{3}}}} \right) \theta
\left( {\frac {z{\hbar}^{2}u_{{2}}u_{{1}}}{u_{{3}}u_{{4}}}} \right) 
\theta \left( z\hbar \right). 
\een
\end{small}
Several terms in the sum $A_1+A_2+A_3$ cancels and we obtain:
\begin{small}
\ben
A_1+A_2+A_3 &=& \theta \left( \hbar \right) ^{2}\theta \left( {\frac {u_{
			{1}}}{u_{{2}}}} \right) \theta \left( {\frac {u_{{1}}u_{{2}}}{u_{{3}}u
		_{{4}}}} \right)  \left( \theta \left( {\frac {zu_{{1}}{\hbar}^{2}}{u_{{3}
}}} \right) \theta \left( {\frac {z\hbar u_{{2}}}{u_{{4}}}} \right) \theta
\left( {\frac {\hbar u_{{2}}}{u_{{3}}}} \right) \theta \left( {\frac {hu_{
			{3}}}{u_{{4}}}} \right) \theta \left( {\frac {u_{{1}}}{u_{{4}}}}
\right) \right. \\
&& \left.  -\theta \left( {\frac {\hbar u_{{2}}}{u_{{4}}}} \right) \theta
\left( {\frac {\hbar u_{{4}}}{u_{{3}}}} \right) \theta \left( {\frac {zhu_
		{{2}}}{u_{{3}}}} \right) \theta \left( {\frac {zu_{{1}}{\hbar}^{2}}{u_{{4}
}}} \right) \theta \left( {\frac {u_{{1}}}{u_{{3}}}} \right)  \right) 
\een
\end{small}
Now, modulo a common multiple $\theta \left( \hbar \right) ^{2} \theta \left( {\frac {u_{{1}}u_{{2}}}{u_{{3}}u_{{4}}}} \right)$ the relation $A_1+A_2+A_3=B_1+B_2$
is exactly the 4-term relation (\ref{fourterm}). 

\section{Proof of Theorem \ref{lbthm} \label{mainproof}}
Let us first discuss the idea of the proof. We denote the restriction matrices for the elliptic stable envelopes in (holomorphic normalization) by:
$$
{\bf{T}}_{\bf q, p}=\left.\bStab(\bf q)\right|_{\widehat{\Or}_{\bf p}}, \ \ \ {\bf{T}}'_{\lambda, \mu}=\left.\bStab' (\lambda)\right|_{\widehat{\Or}' _{\mu}}.
$$
Recall that the isomorphism $\kappa$ induces an isomorphism of extended orbits $\widehat{\Or}'_{\mu}\cong \widehat{\Or}_{\bf p}\cong \cE_{\bT\times \bT^{'}}$. First, we show that under this isomorphism we have the following identity
\be \label{matid}
{\bf{T}}'_{\lambda, \mu}={\bf{T}}_{\bf q, p}, \ \ \ \textrm{for} \ \ \ {\bf p}=
\textsf{bj}(\lambda), \ \ {\bf q}=
\textsf{bj}(\mu).
\ee
By Theorem \ref{uqth}, to prove this identity it is enough to check that the matrix elements ${\bf{T}}'_{\lambda, \mu}$ satisfies the conditions 1),2), 3).   

The condition 1) says that for fixed $\mu$ the set of functions ${\bf{T}}'_{\lambda, \mu}$ is a section of the line bundle $\frak{M}(\bf{q})$ see (\ref{twlin}). By Proposition \ref{progkm},
to check this property it is enough to show that ${\bf{T}}'_{\lambda, \mu}$ has the same quasiperiods in equivariant and K\"ahler variables as sections of $\left.\frak{M}(\bf{q})\right|_{\widehat{\Or}_{\bf p}}$ and that it satisfies the GKM conditions:
\be \label{gkmcon}
\left.{\bf{T}}'_{\lambda, \mu}\right|_{u_i=u_j} = \left.{\bf{T}}'_{\nu, \mu}\right|_{u_i=u_j},
\ee
if the fixed points ${\bf p}=\textsf{bj}(\lambda)$ and ${\bf s}=\textsf{bj}(\nu)$ are connected by equivariant curve, i.e., 
if ${\bf p}={\bf s} \setminus \{i\} \cup \{j\}$ as $k$-sets. 
We recall that the quasiperiods of ${\bf{T}}'_{\lambda, \mu}$ are the same as of function (\ref{ufund}) multiplied by (\ref{thptefdual}). The quesiperiods of sections of $\left.\frak{M}(\bf{q})\right|_{\widehat{\Or}_{\bf p}}$ are the same as of function (\ref{univ}) multiplied by (\ref{prefxor}). Both resulting functions are explicit product of theta functions and the quasiperiods are determined immediately from (\ref{thettrans}). A long but straightforward calculation then shows that the quasiperiods coincide. To check (\ref{gkmcon}) is however less trivial,we prove it in the next Subsection~\ref{gkmsec}.

The condition 2) is trivial and follows from our choice of holomorphic normalization. 

The condition 3) says that ${\bf{T}}'_{\lambda, \mu}$ must be divisible by some explicit product of theta functions and the result of division is a holomorphic function in variables $u_i$. We will refer to these properties as {\it divisibility} and {\it holomorphicity}. These properties of the matrix ${\bf{T}}'_{\lambda, \mu}$ will be proven in Subsections \ref{divsec} and \ref{holsec} respectively.

Let us consider the following scheme:
\be \label{sch}
\textsf{S}(X,X'):=\cE_{\bT\times\bT'} \times S^{k}(E) \times \prod\limits_{i=1}^{n-1} S^{\textsf{v}_i}(E).
\ee
Here $S^{k}(E)$  denotes $k$-th symmetric power of the elliptic curve $E$. We assume that coordinates on $S^{k}(E)$ are given by symmetric functions on Chern roots of tautological bundle on $X$. Similarly, $S^{\textsf{v}_i}(E)$
denotes the scheme with coordinates given by Chern roots of $i$-th tautological bundle on $X'$, i.e., symmetric functions in $x_{\Box}$ with $c_{\Box}=i$, see Section \ref{tbxp} for the notations. 

Recall that the stable envelopes $\bStab(\bf q)$ and $\bStab' (\lambda)$ are defined explicitly by (\ref{stabgr}) and (\ref{ellipticenvelope}). In particular, they are symmetric functions in the Chern roots of tautological bundles. This means that the function defined by\footnote{Note that ${\bf{T}}_{\bf q, p}$ it triangular matrix with non-vanishing diagonal, thus it is invertible and the sum in (\ref{mtil}) is well defined.
Note also that in (\ref{mtil}) we assume that ${\bf q}=\textsf{bj}(\mu)$ and the second sum over $\mu$ is the same as sum over ${\bf q} \in X^{\bT}$.
}
\be \label{mtil}
\tilde{\frak{m}}:=\sum\limits_{{\bf p,q} \in X^{\bT}}  ({\bf{T}})^{-1}_{\bf q, p}  \,\bStab(\bf p) \, \bStab' (\textsf{bj}^{-1}({\bf q}))
\ee
can be considered as a meromorphic section of certain line bundle on $\textsf{S}(X,X')$.  We denote this line bundle by $\tilde{\frak{M}}$.

Let us consider the map 
$$
\tilde{\bf c} : \textrm{Ell}_{\bT\times\bT'}(X\times X') \rightarrow S(X,X')
$$
which is defined as follows: the component of $\tilde{\bf c}$ mapping to the first factor of (\ref{sch}) is the projection to the base. The components of the map $\tilde{\bf c}$ to $S^{k}(E)$ and to $S^{\textsf{v}_l}(E)$ are given by the elliptic Chern classes of the corresponding tautological classes. For the definition of elliptic Chern classes see Section 1.8 in \cite{GKV} or Section 5 in \cite{ell1}.
It is known that $\tilde{\bf c}$ is an embedding \cite{kirv}, see also Section 2.4 in \cite{AOelliptic} for discussion.

Finally,  the line bundle and the section of the Theorem \ref{lbthm} can be defined as 
$\frak{M}=\tilde{\bf c}^{*}\tilde{\frak{M}}$ and $\frak{m}=\tilde{\bf c}^{*}(\tilde{\frak{m}})$.
Indeed, from the very definition (\ref{mtil}) and (\ref{matid}) it is obvious that
$$
(i^{*}_{\lambda})^*({\frak{m}})=\left.\tilde{\frak{m}}\right|_{\lambda}={\bf Stab}({\bf p}), \qquad (i^{*}_{{\bf p}})^{*}({\frak{m}})=
\left.\tilde{\frak{m}}\right|_{\bf p}={\bf Stab}' (\lambda). 
$$
i.e., the section $\frak{m}$ is the Mother function.

\subsection{Cancellation of trees}

Before checking Conditions 1)-3), we need a key lemma which describes that under specialization of some $u_i$ parameters, the contributions from trees cancel out with each other and  the summation simplifies  dramatically. 

Define the \emph{boundary} of $\bar\lambda$ to be the set 
$$
\{ (i,j) \in \bar\lambda \mid (i-1, j-1) \not\in \bar\lambda \}. 
$$
Define the \emph{upper boundary} of $\bar\lambda$ to be the set
$$
U := \{ (i,j) \in \bar\lambda \mid j=k \}. 
$$
Consider a $2\times 2$ square in $\bar\lambda$, consisting of $(c, d), (c+1, d), (c, d+1), (c+1, d+1)$, where $(c+1, d)$ is in the $a$-th diagonal. Let $\bar\ft$ be a tree, which contains the edge $(c+1, d+1) \to (c+1, d)$. 

The \emph{involution} of $\bar\ft$ at the box $(c+1, d)$ is defined to be the tree \footnote{The involution sends a tree to another tree. In fact, by definition, $\bar\ft$ contains the edge $\delta_1: (c+1, d+1) \to (c+1, d)$, but does not contain the edge $\delta_2: (c, d) \to (c+1, d)$. It also follows from definition that the subtree $\bar\fs:= [(c+1, d), \bar\ft]$ does not contain the other 3 boxes in the $2\times 2$ square (otherwise $\bar\ft$ would contain either a loop or a $\reflectbox{\textsf{L}}$-shape). It is then easy to see that $\inv(\bar\ft)$ is still a tree. } $\inv (\bar\ft, (c+1, d))$ obtained by removing $(c+1, d+1) \to (c+1, d)$ from $\bar\ft$ and adding the edge $(c,d) \to (c+1, d)$. We abbreviate the notation as $\inv(\bar\ft)$ if there's no confusion. Define $\inv (\inv (\bar\ft)) = \bar\ft$. Involution is a well-defined operation on all trees at all boxes that are not in $U$ or the boundary of $\bar\lambda$. 

Let $\bar\fs$ be the subtree 
$$
\bar\fs := [(c+1, d), \bar\ft] = [(c+1, d), \inv (\bar\ft)]. 
$$
The $u$-parameter contributed from $\bar\fs$ is
$$
u(\bar\fs) := \prod_{I \in \bar\fs} \frac{u_{c_I+1}}{u_{c_I}}.
$$

\begin{Lemma}[Cancellation lemma] \label{lemma-cancel}
	$$
	\frac{\cR(\bar\ft) \cW(\bar\ft)}{\cR(\inv (\bar\ft)) \cW (\inv (\bar\ft))} \bigg|_{u(\bar\fs) = 1} = -1 . 
	$$
	As a corollary,
	$$
	\sum_{\sigma \in \fS_{\bar\lambda}} \frac{\cN^\sigma_{\bar\lambda}}{\cD^\sigma_{\bar\lambda}} \cR^\sigma (\bar\ft) \cW^\sigma (\bar\ft) = - \sum_{\sigma \in \fS_{\bar\lambda}} \frac{\cN^\sigma_{\bar\lambda}}{\cD^\sigma_{\bar\lambda}} \cR^\sigma (\inv(\bar\ft)) \cW^\sigma (\inv(\bar\ft)). 
	$$
\end{Lemma}

\begin{proof}
	Direct computation shows that
	\be
	\frac{\cR(\bar\ft)}{\cR(\inv(\bar\ft))} &=& \frac{\theta \Big( \dfrac{x_{c+1,d} \hbar}{x_{c, d}} \Big) }{\theta \Big( \dfrac{x_{c+1, d+1}}{x_{c+1, d}} \Big) } .
	\ee
	The quotient $\cW (\bar\ft) / \cW (\inv(\bar\ft)) $ has contribution from an edge $e$ if the subtree $[h(e), \bar\ft]$ or $[h(e), \inv(\bar\ft)]$ contains $(c+1, d+1)$ or $(c,d)$. Those contributions are all of the form
	$$
	\frac{ \theta \Big( \dfrac{x_{t(e)} \varphi^\lambda_{h(e)} }{x_{h(e)} \varphi^\lambda_{t(e)} } \prod\limits_{I\in [h(e), \inv(\bar\ft)]} \dfrac{u_{c_I+1}}{u_{c_I}} \cdot u(\bar\fs) \Big) }{ \theta \Big( \dfrac{x_{t(e)} \varphi^\lambda_{h(e)} }{x_{h(e)} \varphi^\lambda_{t(e)} } \prod\limits_{I \in [h(e), \inv(\bar\ft)]} \dfrac{u_{c_I+1}}{u_{c_I}}  \Big) } 
	 \qquad \text{ or } \qquad 
	 \frac{ \theta \Big( \dfrac{x_{t(e)} \varphi^\lambda_{h(e)} }{x_{h(e)} \varphi^\lambda_{t(e)} } \prod\limits_{I\in [h(e), \bar\ft]} \dfrac{u_{c_I+1}}{u_{c_I}}  \Big) }{ \theta \Big( \dfrac{x_{t(e)} \varphi^\lambda_{h(e)} }{x_{h(e)} \varphi^\lambda_{t(e)} } \prod\limits_{I \in [h(e), \bar\ft]} \dfrac{u_{c_I +1}}{u_{c_I}}  u(\bar\fs) \Big) } 
	$$
	which are both $1$ under $u(\bar\fs) = 1$, and the only remaining factor comes from the edges $(c+1, d+1) \to (c+1, d)$ and $(c,d) \to (c+1,d)$: 
	$$
	\left. \frac{\theta \Big( \dfrac{x_{c+1, d+1}}{x_{c+1, d}} \dfrac{\varphi^\lambda_{c+1, d}}{\varphi^\lambda_{c+1, d+1}} u(\bar\fs) \Big) }{\theta \Big( \dfrac{x_{c,d}}{x_{c+1, d}} \dfrac{\varphi^\lambda_{c+1, d}}{\varphi^\lambda_{c,d}} u(\bar\fs) \Big)} \right|_{u(\bar\fs) = 1} = \frac{\theta \Big( \dfrac{x_{c+1, d+1}}{x_{c+1, d}} \Big)}{\theta \Big( \dfrac{x_{c,d}}{x_{c+1, d}} \hbar^{-1} \Big)} .
	$$
	The lemma follows. 
\end{proof}

\subsection{GKM conditions \label{gkmsec}}

The goal of this section is to prove that the elliptic stable envelopes $\bStab'(\lambda)$ satisfy the GKM condition (\ref{gkmcon}). For simplicity we assume that $(1,k) \in \bar\lambda$; in other words, $\bar\lambda$ starts with diagonal $1$. The general case can be easily reduced to this. 

A subtree of $\bar\ft$ is called a \emph{strip} if  it contains at most one box in each diagonal. We will also abuse the name \emph{strip} for a connected subset in a partition that contains at most one box in each diagonal. We call a strip that starts from diagonal $i$ to $j-1$ an \emph{$(i,j)$-strip}. 

Let $\lambda$ and $\mu$ be two partitions, and ${\bf p} = \textsf{bj} (\lambda)$, ${\bf q} = \textsf{bj} (\mu)$. Suppose that as fixed points in $X$, ${\bf p}$ and ${\bf q}$ are connected by a torus-invariant curve, which means that 
$$
{\bf q} = {\bf p} \setminus \{i\} \cup \{j\},
$$
for some $1\leq i, j\leq n$ (assume $i<j$). On the dual side, that means $\mu\supset \lambda$, and $\mu\backslash \lambda$ is an $(i,j)$-strip,  lying the boundary of $\bar\lambda$. 

Recall the GKM condition:
\begin{Proposition} \label{GKM-thm}
{\it	For partitions $\lambda$ and $\mu$ as above, 
	$$
	\left. \bStab' (\lambda) \right|_{u_i = u_j} = \left. \bStab' (\mu) \right|_{u_i = u_j}.
	$$}
\end{Proposition}

By localization and the triangular property of stable envelopes, it suffices to show that for any partition $\nu \supset \lambda$, 
$$
\left. \bStab' (\lambda) \right|_{\nu, u_i = u_j} = \left. \bStab' (\mu) \right|_{\nu, u_i = u_j}.
$$
Before proving the GKM condition, we need some analysis on the specialization of the stable envelopes under $u_i = u_j$.

\subsubsection{Specialization of $\bStab'(\lambda)$ under $u_i = u_j$}

Recall that ${\bf p} \subset {\bf n}$ and $i \in {\bf p}$, $j \not\in {\bf p}$, $i<j$. We would like to study the specialization $u_i = u_j$. 

By definition 
$$
\bStab' (\lambda) = T_{\bf p,p} \cdot \Stab'(\lambda),
$$
where 
$$
T_{\bf p,p} = (-1)^{(n-k)k} \prod_{i\in {\bf p}, \, j\in {\bf n} \backslash {\bf p}, \,  i<j} \theta \Big( \frac{u_i}{u_j} \Big) \prod_{i\in {\bf p}, \, j \in {\bf n} \backslash {\bf p}, \, i>j} \theta \Big( \frac{u_j}{u_i \hbar} \Big). 
$$
In particular, $T_{\bf p,p}$ contains a factor $\theta \Big( \dfrac{u_i}{u_j} \Big)$. 

For any tree $\bar\ft$ in $\bar\lambda$, consider all subtrees of $\bar\ft$ that are $(i,j)$-strips
$$
B = \{ B_i, B_{i+1}, \cdots, B_{j-1} \},
$$
where $B_l$ is the box in the $l$-th diagonal. We define $B(\bar\ft, i, j)$ to be one whose $B_i$ has the \emph{smallest height}. If $\bar\ft$ does not contain any $(i,j)$-stripes as subtrees, define $B (\bar\ft, i, j) = \emptyset$. A tree $\bar\ft$ in $\bar\lambda$ is called \emph{distinguished}, if its strip $B(\bar\ft, i, j)\neq \emptyset$, and lies in the boundary of $\bar\lambda$. 

A simple observation is that, for the contribution from $\bar\ft$ to $\bStab'(\lambda)$ to be nonzero under $u_i = u_j$, $B(\bar\ft, i, j)$ has to be nonempty.

\begin{Lemma} \label{root} {\it
	Let $B$ be an $(i,j)$-strip in $\bar\ft$ which is a subtree. Let $B_U$ be the box in $B\cap U$ with largest content. We have
	\begin{itemize}
		
		\item if $B_i \not\in U$, then $B_i$ is the root of $B$;
		
		\item if $B_i \in U$, then $B_U$ is the root of $B$. 
		
	\end{itemize}
}
\end{Lemma}

\begin{proof}
	If $B_i \not\in U$, and the root of $B$ is some box other than $B_i$. Then the unique path from $B_i$ to $U$ has a box $\Box$ in its interior with local maximal content. $\Box$ must be connected to both the boxes to the left and above it, which is not allowed. 
	
	If $B_i \in U$, then every box in $B$ from $B_i$ to $B_U$ is in $U$. It is clear that the root of $B$ is $B_U$. 
\end{proof}

\begin{Lemma} \label{bd} {\it
	Let $B$ be an $(i,j)$-strip in $\bar\ft$ which is a subtree. If $B_i$ lies in the boundary of $\bar\lambda$, then $B$ lies entirely in the boundary of $\bar\lambda$; in other words, $B(\bar\ft, i, j) = B$. }
\end{Lemma}

\begin{proof}
	Suppose $B_i$ lies in the boundary, but $B$ does not. Then there exists a box in the boundary of $\bar\lambda$, not in $B$, but in a diagonal less than $j-1$. Since $\bar\ft$ is a tree, there is a unique path from that box to some box in $U$. This path would contain a box with local maximal content in its interior. Contradiction. 
\end{proof}

\begin{Lemma} \label{dist}
{\it	Under the specialization $u_i = u_j$,
	$$
	\left. \bStab' (\lambda) \right|_{u_i = u_j} = T_{\bf p,p} \cdot  \epsilon(\lambda) \, \Theta (\widetilde N^{', -}_\lambda) \sum_{\substack{ \sigma \in \fS_{\bar\lambda} \\ \bar\ft \text{ disinguished} }}  \frac{\cN^\sigma_{\bar\lambda}}{\cD^\sigma_{\bar\lambda}} \cR^\sigma (\bar\ft) \cW^\sigma (\bar\ft) \bigg|_{u_a = u_b} .
	$$ }
\end{Lemma}

\begin{proof}
	Let $B = B(\bar\ft, i, j)$. Since $T_{\bf p, p}$ contains a zero $u_i / u_j$, if $ B= \emptyset$, it is clear that the stable envelope will vanish. Now assume $B\neq \emptyset$. 
	
	If $i = 1$, then $B_i = (1,k)$. By Lemma \ref{bd} $B$ lies in the boundary and $\bar\ft$ is distinguished. 
	
	If $i \neq 1$, it is easy to see that $B_i \not\in U$ (otherwise as a subtree $B$ must contain $(1,k)$). If moreover $B_i$ is not in the boundary, then one can construct its involution $\inv (\bar\ft)$. By Lemma \ref{lemma-cancel}, the contributions from $\bar\ft$ and $\inv(\bar\ft)$ cancel with each other. Therefore, in the summation over trees, we are left with those $\bar\ft$ whose $B_i$ lies in the boundary of $\bar\lambda$, which by Lemma \ref{bd} are distinguished. 
\end{proof}

Fix a distinguished tree $\bar\ft$, and $B = B(\bar\ft, i, j)$. Let's consider the restriction of $\bStab'(\lambda)$ to a certain fixed point $\nu \supset \lambda$. For an individual contribution from given $\bar\ft$ and $\sigma$, we take the following limit, called \emph{B-column limit} for $\nu \backslash \lambda$: first, for each pair of $I,J \in \nu\backslash \lambda$ such that $I$ is above $J$ and $I, J \not\in B$, take
$$
x_I = x_J \hbar;
$$
for any $I \in B$, take $x_I = \varphi^\nu_I$; finally take any well-defined evaluation of the remaining variables. Note that this limit only depends on the partition $\lambda$ and the pair $i, j$, and does not depend on $\bar\ft$. 

\begin{Lemma} \label{sigma-fixes-B}
{\it	The restriction
	$$
	\frac{\cN^\sigma_{\bar\lambda}}{\cD^\sigma_{\bar\lambda}} \cR^\sigma (\bar\ft) \cW^\sigma (\bar\ft) \bigg|_\nu
	$$
	under the $B$-column limit vanishes unless $\sigma$ fixes $B$.} 
\end{Lemma}

\begin{proof}
	Suppose that the restriction does not vanish under the chosen limit. Recall that $B_l$, $i \leq l\leq j-1$ is the box in the $l$-th diagonal of $B$. We use induction on $l$, from $j-1$ to $i$. Recall that by the refined formula, $\sigma$ lies in $\fS_{\nu\backslash \lambda}$. 
	
	First we show that $B_{j-1}$ is fixed by $\sigma$. Let $Y_1, Y_2, \cdots$ be the boxes in the $j$-th diagonal of $\nu\backslash \lambda$, such that the heights of $Y_m$'s are increasing. Since $j \not\in {\bf p}$, $Y_1$ is the box to the right of $B_{j-1}$. Hence we have the theta factors
	$$
	\prod_{m\geq 1} \theta \Big( \frac{x_{\sigma ( B_{j-1})}}{x_{Y_m} \hbar} \Big),
	$$
	as $\rho_{B_{j-1}} > \rho_{Y_m}$ and $B_{j-1}$ is not connected to $Y_1$. Under the $B$-column limit for $\nu\backslash\lambda$, this product vanishes unless $\sigma (B_{j-1})$ has no box below it, which implies $\sigma (B_{j-1}) = B_{j-1}$. 
	
	Next, suppose that $B_{l+1}$ is fixed by $\sigma$, consider $B_l$. Let $e$ be the edge connecting $B_l$ and $B_{l+1}$. Let $X_1 = B_l, X_2, \cdots$ and $Y_1 = B_{l+1}, Y_2, \cdots$ be respectively the boxes in the $l$-th and $(l+1)$-th diagonals of $\nu\backslash \lambda$. 
	
	If $e$ is \emph{horizontal}, then we have factors
	$$
	\prod_{m\geq 2} \theta \Big( \frac{x_{\sigma (X_m)}}{x_{B_{l+1} }} \Big),
	$$
	since we know $\rho_{X_m} < \rho_{B_{l+1}}$ and $X_2$ is not connected to $B_{l+1}$. If $\sigma (X_m) = X_1 = B_l$ for some $m\neq 1$, then the factor $\theta \Big( \dfrac{x_{\sigma (X_m)}}{x_{B_{l+1} }} \Big) = \theta \Big( \dfrac{x_{B_l}}{x_{B_{l+1} }} \Big)$ vanishes under the $B$-column ordering. Hence $\sigma (B_l) = B_l$. 
	
	If $e$ is \emph{vertical}, then we have factors
	$$
	\prod_{m\geq 2} \theta \Big( \frac{ x_{\sigma (B_l)}}{x_{\sigma (Y_m)} \hbar } \Big) = \prod_{m\geq 2} \theta \Big( \frac{ x_{\sigma (B_l)}}{x_{Y_m} \hbar} \Big) ,
	$$
	since we know $\rho_{B_l} > \rho_{Y_m}$ and $B_l$ is not connected to $Y_2$. If $\sigma (B_l) = X_m$ for some $m\neq 1$, then the factor $\theta \Big( \dfrac{ x_{\sigma (B_l)}}{x_{Y_m} \hbar} \Big) = \theta \Big( \dfrac{ x_{X_m}}{x_{Y_m} \hbar} \Big)$ vanishes under the $B$-column ordering, since $X_m$ is the box above $Y_m$ and they are not in $B$ for $m\geq 2$. Hence $\sigma$ fixes $B_l$.  
\end{proof}

In summary, after restriction to $\nu$ in the $B$-column limit for $\nu\backslash \lambda$, only contributions from distinguished $\bar\ft$ and permutations $\sigma$ that fix $B$ survive. We are now ready to prove Proposition \ref{GKM-thm}.

\subsubsection{Proof of Proposition \ref{GKM-thm}: $\mu$ is not contained in $\nu$}

In this case, the strip $\mu\backslash \lambda$ is not entirely contained in $\nu\backslash \lambda$. Clearly we have $\bStab'(\mu) \big|_\nu = 0$. 

\begin{Lemma} \label{not-containd}
	$$
	\left. \bStab' (\lambda) \right|_{\nu, u_i = u_j} = 0. 
	$$
\end{Lemma}

\begin{proof}
	By Lemma \ref{dist}, only distinguished trees $\bar\ft$, with strip $B = B(\bar\ft, i, j)$ contributes. Let $B_i, \cdots, B_{j-1}$ be boxes in $B$, and $X$ be the first box in $B$ that does not lie in $\nu\backslash \lambda$. For restriction to $\nu$ of an individual contribution by given $\bar\ft$ and $\sigma$, we take the \emph{column limit} for $\bar\nu$, i.e. first let $x_I = x_J$ for any $I, J \in \bar\nu$ in the same column, and then take any limit for the remaining variables. 
	
	If $X\neq B_i$, then there's a box $Y$ above it, which also lies in $\bar\nu$. Since $Y\not\in B$, the edge connecting $X$ and $Y$ is not in $\bar\ft$. The contribution from $\bar\ft$ then contains a factor $\theta (X/Y) $, which vanishes under the column limit. 
	
	If $X = B_i$, then either $i\neq 1$, or $i = 1$, and the entire $B\cap U$, and in particular $B_U$, lie in $\bar\nu$. By Lemma \ref{root} we know the root $r_B = B_i$ or $B_U$ respectively. Denote the box not in $B$ and  connected to $r$ by $C$. The factor in $\cW^{\sigma = 1} (\bar\ft)$ that contributes the pole $u_i / u_j$ is
	$$
	 \frac{ \left.\theta \Big( \dfrac{x_C \varphi^\lambda_{r_B} }{x_{r_B} \varphi^\lambda_C } \dfrac{u_j}{u_i} \Big) \right|_\nu}{ \theta \Big( \dfrac{u_j}{u_i} \Big)}  = 1. 
	$$
	$\bStab'(\lambda) \big|_\nu = 0$ under $u_i = u_j$ because of the zero $u_i / u_j$ in $T_{\bf p,p}$. 
\end{proof}

\subsubsection{Proof of Proposition \ref{GKM-thm}: $\mu \subset \nu$}

In this case $B$ is contained entirely in $\nu\backslash \lambda$; in other words, $\lambda \subset \mu \subset \nu$. Let $r_B$ be the root of $B$, which if $i=1$, is $B_U$; and if $i \neq 1$, is $B_a$. 

If $(n-k, k) \not\in B$, let $C\in \bar\ft \backslash B$ be the box connected to $r_B$. $C$ could be in or not in $\nu\backslash \lambda$. If $(n-k, k) \in B$, we denote by convention that $x_C / \varphi_C^\lambda = 1$. Then
\ben
&& \left. \theta \Big( \frac{u_j}{u_i} \Big) \Stab' (\lambda) \right|_{\nu, u_i = u_j} \\
&=& \theta \Big( \frac{u_j}{u_i} \Big) \epsilon (\lambda) \left. \Theta (\widetilde N^{', -}_\lambda) \right|_\nu \cdot \sum_{\sigma \in \fS_{\nu \backslash \mu}, \bar\ft \cap \bar\mu }  \frac{\cN^\sigma_{\bar\mu}}{\cD^\sigma_{\bar\mu}} \cR^\sigma (\bar\ft \cap \bar\mu) \cW^\sigma (\bar\ft \cap \bar\mu)  \bigg|_\nu  \cdot \prod_{\substack{ c_I = n-k, \, I\neq (n-k, k) \\ I\in \mu\backslash \lambda}} \left. \theta \Big( \frac{a_2 \hbar}{x_I} \Big) \right|_\nu \\
&&  \prod_{\substack{c_I + 1 = c_J, \,  \rho_I > \rho_J \\ ( I \leftrightarrow J ) \not\in \bar\ft, \,  I \text{ or } J \in \mu \backslash \lambda }} \left. \theta \Big( \frac{x_{\sigma(J)} \hbar}{x_{\sigma(I)}} \Big) \right|_\nu 
\prod_{\substack{ c_I + 1 = c_J, \,  \rho_I < \rho_J \\ ( I \leftrightarrow J ) \not\in  \bar\ft, \, I \text{ or } J \in \mu \backslash \lambda  }} \left. \theta \Big( \frac{x_{\sigma(I)}}{x_{\sigma(J)}} \Big) \right|_\nu 
\cdot \prod_{\substack{ c_I = c_J,\,  \rho_I > \rho_J \\ I \text{ or } J \in \mu \backslash \lambda }} \theta \Big( \frac{x_{\sigma(I)}}{x_{\sigma(J)}} \Big)^{-1} \theta \Big( \frac{x_{\sigma(I)}}{x_{\sigma(J)} \hbar} \Big)^{-1}  \bigg|_\nu \\
&&  \cdot \frac{ \left. \theta \Big( \dfrac{x_C \varphi^\lambda_{r_B} }{x_{r_B} \varphi^\lambda_C } \dfrac{u_j}{u_i} \Big) \right|_\nu}{ \theta \Big( \dfrac{u_j}{u_i} \Big) }  \prod_{e\in B \backslash U} \frac{\left.  \theta \Big( \dfrac{x_{t(e)} \varphi^\lambda_{h(e)} }{x_{h(e)} \varphi^\lambda_{t(e)} } \dfrac{u_j}{u_{h(e)}} \Big) \right|_\nu }{ \theta \Big( \frac{u_j}{u_{h(e)}} \Big)} 
\prod_{e\in B \cap U} \frac{ \left. \theta \Big( \dfrac{x_{t(e)} \varphi^\lambda_{h(e)} }{x_{h(e)} \varphi^\lambda_{t(e)} } \dfrac{u_{t(e)}}{u_i} \Big) \right|_\nu }{ \theta \Big( \dfrac{u_{t(e)}}{u_i} \Big)}  ,
\een
which can be compared with $\Stab' (\mu)$. Direct computation shows that
\ben
\theta \Big( \frac{u_j}{u_i} \Big) \frac{\Stab'(\lambda) }{\Stab' (\mu) } \bigg|_{\nu,u_i = u_j} &=&  (-1)^{j-i} \theta (\hbar^{-1})  \prod_{e\in B \backslash U} \frac{ \left. \theta \Big( \dfrac{x_{t(e)} \varphi^\lambda_{h(e)} }{x_{h(e)} \varphi^\lambda_{t(e)} } \dfrac{u_j}{u_{h(e)}} \Big) \right|_{\nu, u_i = u_j} }{ \theta \Big( \dfrac{u_j}{u_{h(e)}} \Big)} 
\prod_{e\in B \cap U} \frac{ \left.  \theta \Big( \dfrac{x_{t(e)} \varphi^\lambda_{h(e)} }{x_{h(e)} \varphi^\lambda_{t(e)} } \dfrac{u_{t(e)}}{u_i} \Big) \right|_{\nu, u_i = u_j} }{ \theta \Big( \dfrac{u_{t(e)}}{u_i} \Big)}  \\
&=& (-1)^{j-i} \theta (\hbar^{-1}) \prod_{\substack{i< m <j \\ m \in {\bf p} }} \frac{\theta \Big( \dfrac{\hbar u_j}{u_m} \Big)}{\theta \Big( \dfrac{u_j}{u_m} \Big) } \prod_{\substack{i<m<j \\ m \in {\bf n} \backslash {\bf p} } } \dfrac{\theta \Big( \dfrac{u_j}{u_m \hbar } \Big)}{\theta \Big( \dfrac{u_j}{u_m} \Big)} ,
\een
where the last equality is because for $e\in B \backslash U$,
$$
\frac{ \left. \theta \Big( \dfrac{x_{t(e)} \varphi^\lambda_{h(e)} }{x_{h(e)} \varphi^\lambda_{t(e)} } \dfrac{u_j}{u_{h(e)}} \Big) \right|_\nu }{ \theta \Big( \dfrac{u_j}{u_{h(e)}} \Big)}  = \left\{ \begin{aligned}
& \frac{ \theta \Big( \hbar \dfrac{u_j}{u_{h(e)}} \Big) }{ \theta \Big( \dfrac{u_j}{u_{h(e)}} \Big)} , && \qquad h(e) \in p, \ e\not\in U \\
& \frac{ \theta \Big( \hbar^{-1} \dfrac{u_j}{u_{h(e)}} \Big) }{ \theta \Big( \dfrac{u_j}{u_{h(e)}} \Big)} , && \qquad h(e) \not\in p, \ e\not\in U  
\end{aligned} \right. 
$$
and for $e\in B\cap U$,
$$
\frac{ \left. \theta \Big( \dfrac{x_{t(e)} \varphi^\lambda_{h(e)} }{x_{h(e)} \varphi^\lambda_{t(e)} } \dfrac{u_{t(e)}}{u_i} \Big) \right|_{\nu, u_i = u_j} }{ \theta \Big( \dfrac{u_{t(e)}}{u_i} \Big)}  = \dfrac{ \theta \Big( \hbar \dfrac{u_{t(e)}}{u_j} \Big) }{ \theta \Big( \dfrac{u_{t(e)}}{u_j} \Big)} .
$$

The proposition is proved by making the change of variable $\hbar \mapsto \hbar^{-1}$ in the above result, and compare with the following lemma. 

\begin{Lemma}
	$$
	\left. \theta \Big( \dfrac{u_i}{u_j} \Big)^{-1} \frac{T_{\bf p,p}}{T_{\bf q,q}} \right|_{u_i = u_j}
	= \theta (\hbar^{-1})^{-1} \prod_{\substack{i<m<j \\ m \in {\bf n} \backslash {\bf p}}} \theta \Big( \dfrac{u_j}{u_m} \Big) \prod_{\substack{ i<m<j \\ m \in {\bf p}}} \theta \Big( \frac{u_m}{u_j} \Big)  \prod_{\substack{i<m<j \\ m \in {\bf n} \backslash {\bf p}}} \theta \Big( \frac{u_m}{u_j \hbar} \Big)^{-1} \prod_{\substack{i<m<j \\ m \in {\bf p}}} \theta \Big( \frac{u_j}{u_m \hbar} \Big)^{-1} 
	$$
\end{Lemma}

\begin{proof}
	Straightforward computation. 
\end{proof}

\subsection{Divisibility \label{divsec}}

In this subsection we aim to prove the following divisibility result. Let ${\bf p} = \textsf{bj} (\lambda)$, ${\bf q} = \textsf{bj} (\mu) \in X^T$ be two fixed points. 

\begin{Proposition}
{\it	The function $\dfrac{T_{\bf p, p}}{T'_{\mu, \mu}} \cdot T'_{\lambda, \mu}$ is of the form
	$$
	f_{\mu, \lambda} \cdot \prod_{\substack{i\in {\bf p}, \, j \in {\bf n} \backslash {\bf p} \\ i>j} } \theta \Big( \frac{u_j}{u_i \hbar} \Big),
	$$
	where $f_{\mu, \lambda}$ is holomorphic in parameters $u_i$. }
\end{Proposition}

\begin{proof}
	Recall that
	$$
	T_{\bf p,p} = (-1)^{k(n-k)} \prod_{i\in {\bf p}, \, j\in {\bf n} \backslash {\bf p}, \,  i<j} \theta \Big( \frac{u_i}{u_j} \Big) \prod_{i\in {\bf p}, \, j \in {\bf n} \backslash {\bf p}, \, i>j} \theta \Big( \frac{u_j}{u_i \hbar} \Big), 
	$$
	and $T'_{\mu, \mu}$ does not depend on $u_i$'s. By formula (\ref{refined}), we can see that all possible poles of $T'_{\lambda,\mu}$ take the form $u_i / u_j$. Therefore, all possible poles of the function $f_{\mu, \lambda}$  in the proposition are of the form $u_i / u_j$. Moreover, by the proof of holomorphicity (Theorem \ref{holo-thm} below), they have no poles at $u_i / u_j$. We conclude that $f_{\mu, \lambda}$ is holomorphic in $u_i$. 
\end{proof}

\subsection{Holomorphicity \label{holsec}}

In this subsection we will prove the holomorphicity, i.e., the normalized restriction matrices of stable envelopes on $X'$ are holomorphic in $u_i$'s. The idea is to apply  general results for $q$-difference equations associated to Nakajima quiver varieties. 

\subsubsection{Quantum differential equations}
Let $X$ be a Nakajima variety. For the cone of effective curves in
$H_{2}(X,\matZ)$ we consider the semigroup algebra which is spanned by monomials $z^{d}$ with $d\in H_{2}(X,\matZ)_{\textrm{eff}}$. It has a natural completion which we denote by $\matC[[z^d]]$. The cup product in the equivariant cohomology $H^{\bullet}_{\bT}(X)$ has a natural commutative deformation, parametrized by $z$:
\be \label{qmul}
\alpha \star \beta = \alpha \cup \beta + O(z)
\ee
known as the quantum product. 

The quantum multiplication defines a remarkable flat connection
on the trivial $H^{\bullet}_{\bT}(X)$-bundle over $\textrm{Spec}(\matC[[z^d]])$. Flat sections $\Psi (z)$ of this connection, considered as $H^{\bullet}_{\bT}(X)$-valued functions, are defined by the following system of differential equations (known as the quantum differential equation or Dubrovin connection):
$$
\varepsilon \dfrac{d}{d\lambda} \Psi(z)= \lambda \star \Psi(z), \qquad  \Psi(z)\in H^{\bullet}_{\bT}(X)[[z]],
$$
where $\lambda \in H^{2}(X,\matC)$ and the differential operator is defined by 
\be \label{zmon}
\dfrac{d}{d\lambda } z^d = (\lambda, d) z^{d}. 
\ee
\subsubsection{Quantum multiplication by divisor}
The equivariant cohomology of Nakajima varieties are equipped with a natural action of certain Yangian $Y_{\hbar}(\frak{g}_{X})$ \cite{Varan}. 
In the case of Nakajima varieties associated to quivers of ADE type
this algebra coincides with the Yangian of the corresponding Lie algebra
(but in general can be substantially larger). 

The Lie algebra $\frak{g}_{X}$ has a root decomposition:
$$
\frak{g}_{X}=\frak{h} \oplus \bigoplus_{\alpha} \frak{g}_{\alpha}
$$
in which $\frak{h}=H^{2}(X,\matC)\oplus \textrm{center}$, and $\alpha \in H_{2}(X,\matZ)_{\textrm{eff}}$. All root subspaces $\frak{g}_{\alpha}$ are finite dimensional and $\frak{g}_{-\alpha}=\frak{g}^{*}_{\alpha}$ with respect to the symmetric nondegenerate invariant form.

The quantum multiplication (\ref{qmul}) for Nakajima varieties 
can be universally described in therms of the corresponding Yangians:

\begin{Theorem}[Theorem 10.2.1 in \cite{MO}]
	{ \it The quantum multiplication by a class $\lambda \in H^{2}(X)$ is given by:
		\be \label{qml}
		\lambda\star  = \lambda \cup  + \hbar \sum\limits_{(\theta, \alpha) >0} \alpha(\lambda) \dfrac{z^{\alpha}}{1-z^{\alpha}} e_{\alpha} e_{-\alpha} +\cdots
		\ee
		where $\theta\in H^{2}(X,\matR)$ is a vector in the ample cone 
		(i.e., in the summation, $\theta$ selects the effective representative from each 
		$\pm \alpha$ pair) and $\cdots$ denotes a diagonal term, which can be fixed by the condition $\lambda \star 1 = \lambda$.
	}	
\end{Theorem}

Let $z_i$ with $i=1,\cdots,n-1$ denote the K\"ahler parameters of the Nakajima variety $X'$ from Section \ref{elxd}. 

\begin{Corollary} \label{difs}
	{ \it The quantum connection associated with the Nakajima variety 
		$X'$ is a connection with regular singularities supported on the hyperplanes
		$$
		z_i z_{i+1}\dots z_{j}=1, \qquad 1 \leq i<j \leq n-1.
		$$	
	}
\end{Corollary}

\begin{proof}
	The variety $X'$ is a Nakajima quiver variety associated with the $A_{n-1}$-quiver. Thus the corresponding Lie algebra $\frak{g}_X\cong \frak{sl}_{n}$. The K\"ahler parameters $z_i$ associated to the tautological line bundles on $X'$ correspond to the simple roots of this algebra. In other words, in the notation of (\ref{zmon}) they correspond to $z_i=z^{\alpha_i}$, where $\alpha_i$, $i=1,\cdots,n-1$ 
	are the simple roots of $\frak{sl}_{n}$ (more precisely, simple roots with respect to positive Weyl chamber $(\theta',\alpha_i)>0$ where $\theta'$ is the choice stability parameters for $X'$).

	By (\ref{qml}), the singularities of quantum differential equation of $X'$ 
	are located at
	$$
	z^{\alpha}=1 
	$$  
	for positive roots $\alpha$. All positive roots of $\frak{sl}_{n}$ are of the form $\alpha=\alpha_i+\alpha_{i+1}+\dots+\alpha_{j}$ with $1\leq i<j\leq n-1$. Thus, the singularities are at
	$$
	z^{\alpha}=z_i z_{i+1} \cdots z_{j}=1.
	$$
\end{proof}

\subsubsection{Quantum difference equation}
In the equivariant K-theory, the differential equation is substituted by its $q$-difference version:
\be \label{qde}
\Psi(z q^{\mathcal{L}}) \mathcal{L} = {\bf M}_{\mathcal{L}}(z) \Psi(z)  
\ee 
where $\mathcal{L} \in \Pic (X)$ is a line bundle and $q=e^{\varepsilon}$ and $\Psi(z) \in K_{\bT}(X)[[z]]$. The theory of quantum difference equations for Nakajima varieties was developed in \cite{OS}. In particular, the operators ${\bf M}_{\mathcal{L}}(z) \in \End (K_{\bT}(X))$ were constructed for an arbitrary line bundle $\mathcal{L}$. These operators are the $q$-deformations of (\ref{qml}), i.e., in the cohomological limit they behave as:
$$
{\bf M}_{\mathcal{L}}(z)=1+ \lambda \star +\cdots
$$
where $\dots$ stands for the terms vanishing the the cohomological limit and $\lambda=c_{1}(\mathcal{L})$. 

In K-theory the sum over roots in (\ref{qml}) is substituted by a product:
$$
{\bf M}_{\mathcal{L}}(z)=\mathcal{L} \prod\limits_{w} {\bf B}_{w}(z)
$$
over certain set of affine root hyperplanes  of an affine algebra $\widehat{\frak{g}}_{X}$.

The singularities of the quantum difference equations, i.e., the singularities of matrix ${\bf M}_{\mathcal{L}}(z)$ are located in the union of singularities of ${\bf B}_{w}(z)$.  The wall crossing operators ${\bf B}_{w}(z)$ are constructed in Section 5.3 of \cite{OS}. 
In particular, if $z^{\alpha}=1$ are the singularities of the quantum differential equation in cohomology then the singularities of (\ref{qde}) can only be located at  $z^{\alpha} q^p \hbar^s=1$ for some integral $p,s$. This, together with Corollary \ref{difs} gives:
\begin{Proposition} \label{prsl}
	{\it The singularities of the quantum difference equation associated with the Nakajima variety $X'$ are located at
		$$
		z_i z_{i+1}\cdots z_{j} q^p \hbar^s=1, \qquad 1\leq i<j\leq n, \  p,s\in \matZ.
		$$	}
\end{Proposition}
\subsubsection{Pole subtraction theorem} 
The elliptic stable envelopes describes the monodromy of $q$-difference equations.  More precisely, the $q$-difference equation (\ref{qde}) has two distinguished fundamental solution matrices, indexed by fixed points $X^{\bT}$. The $z$-solutions $\Psi^{z}$ form a basis of solutions, which are holomorphic in the K\"ahler parameters in the neighborhood $|z_i|<1$. Similarly, the $a$-solutions $\Psi^{a}$ for a basis of solutions holomorphic in  $|a|<1$. By general theory of $q$-difference equations, every two bases of solutions are related by a transition matrix:\footnote{Let us clarify the meaning of terms in (\ref{trm}): here $\Psi^{z}$ denotes the fundamental solution matrix - the $|X^{\bT}| \times |X^{\bT}|$ - dimensional matrix with columns of $\Psi^{z}$ satisfying the quantum difference equation. The set of $|X^{\bT}|$ columns of $\Psi^{z}$ forms a basis in the space of solutions.   The elements of this basis are holomorphic in variables $z$. Similarly, $\Psi^{a}$ is a matrix whose columns form a basis of solutions, which are holomorphic in parameters $a$. The theorem above says that $W(z)=\{W(z)_{\lambda,\mu}\}_{\lambda,\mu\in X^{\bT}}$ coincides with $\Psi^{a} (\Psi^{z})^{-1}$.} 
\be \label{trm}
\Psi^{a} = W(z) \Psi^{z},
\ee
known as the monodromy matrix from the solutions $\Psi^{z}$ to  $\Psi^{a}$. The central result of \cite{AOelliptic} (in the case when $X^{\bT}$ is finite)  is the following: 
\begin{Theorem}[Theorem 5 in \cite{AOelliptic}]
	{ \it Let $X$ be a Nakajima variety and let
		$$
		T_{\lambda,\mu}(z)=\left. \Stab(\lambda)\right|_{\mu} , \qquad \lambda, \mu \in X^\bT
		$$	
		be the restriction matrix for elliptic stable envelope in the basis of fixed points. Then, the matrix $W(z)=\{W(z)_{\lambda,\mu}\}_{\lambda,\mu\in X^{\bT}}$ from (\ref{trm}) in takes the form:
		$$
		W(z)_{\lambda,\mu} =\dfrac{T_{\lambda,\mu}(z)}{\Theta(\left.T^{1/2} X\right|_{\mu})}
		$$
		where $\Theta(\left.T^{1/2} X\right|_{\mu})$ is given by (\ref{thetathom}) applied to Laurent polynomial $\left.T^{1/2} X\right|_{\mu}$. In particular,  $\Theta(\left.T^{1/2} X\right|_{\mu})$ does not depend on the K\"ahler parameters.} 
\end{Theorem}

The singularities of solutions $\Psi^a$ and $\Psi^z$ are supported on the singularities of the corresponding $q$-difference equation.
It implies that the transition matrix also may have only these singularities (if $W(z)$ is singular at a hyperplane $h$, which is not a singularity of $q$-difference equation then, by (\ref{trm}) $\Psi^{a}$ is also singular along $h$ which is not possible).

In particular, combining the last Theorem with Proposition \ref{prsl} we obtain:
\begin{Corollary}  \label{cor1}
	{\it Let $T'_{\lambda,\mu}$ be the restriction matrix of the elliptic stable envelope for the Nakajima variety $X'$ in the basis of fixed points. Then, the singularities of 
		$T'_{\lambda,\mu}$ are supported to the set of hyperplanes:
		$$
		z_i z_{i+1} \cdots z_{j} q^p \hbar^s=1, \qquad 1\leq i<j\leq n, \  p,s\in \matZ.
		$$ }
\end{Corollary}
This  implies that the poles of the restriction matrix $T'_{\lambda,\mu}$ in the coordinates $u_i$ related to K\"ahler variables $(\ref{parident})$ are of the form:
	\be \label{pospols}
	 \frac{u_i}{u_j} \hbar^{s} q^{p}, \qquad i\neq j, \ \ p,s \in \matZ.
	\ee

\subsubsection{Holomorphicity of stable envelope} 
Let us return to the Nakajima varieties $X$ and $X'$ defined in Sections \ref{elx} and \ref{elxd} respectively. We identify the fixed points as in Section \ref{bijsec}, and identify the equivariant and K\"ahler parameters by (\ref{parident}). Let $T'_{\lambda,\mu}$  and $T_{\bf p, q}$ be the restriction matrices of the elliptic stable envelopes for the Nakajima varieties $X'$ and $X$ respectively.
\begin{Theorem} \label{holo-thm}
	{ \it The functions
		$$
		T_{\bf p, p} T'_{\lambda,\mu} 
		$$ 
		are holomorphic in parameters $u_i$. \it }
\end{Theorem}
\begin{proof}
	By (\ref{pospols}) we need to show that
	the denominators of functions $T_{\bf p, p} T'_{\lambda,\mu}$ do not contain factors of the form
	$$
	\prod\limits_{i\neq j} \theta \Big( \frac{u_i}{u_j} \hbar^{s_{ij}} \Big).
	$$
	On the other hand, by Proposition \ref{refin-formula},  the explicit formula for the elliptic stable envelope on $X'$ has the form:
	$$
	T'_{\lambda, \mu} = \epsilon (\lambda) \Theta (\widetilde N^{', -}_\lambda) \big|_\mu \cdot \sum_{\sigma \in \fS_{\mu \backslash \lambda}, \bar\ft}  \dfrac{\cN^\sigma_{\bar\lambda}}{\cD^\sigma_{\bar\lambda}} \cR^\sigma (\bar\ft) \calW^\sigma (\bar\ft) \big|_\mu, 
	$$
	where $\Theta (\widetilde N^{', -}_\lambda) \big|_\mu$, $\cN^\sigma_{\bar\lambda}$, $\cD^\sigma_{\bar\lambda}$, $\cR^\sigma (\bar\ft)$ are independent of $u_i$, and 
	$$
	\calW^\sigma (\bar\ft) =  \frac{\theta \Big( \dfrac{a_2 \hbar }{x_{\sigma(\bar r)}} \prod\limits_{I \in [\bar r, \bar\ft]} \dfrac{u_{c_I}}{u_{c_I+1}} \Big) }{ \theta \Big( \prod\limits_{I \in [\bar r, \bar\ft]} \dfrac{u_{c_I}}{u_{c_I+1}}  \Big)}  \prod_{e\in \bar\ft} \frac{ \theta \Big( \dfrac{x_{\sigma(t(e))} \varphi^\lambda_{h(e)} }{x_{\sigma(h(e))} \varphi^\lambda_{t(e)} } \prod\limits_{I \in [h(e), \bar\ft]} \dfrac{u_{c_I}}{u_{c_I+1}} \Big) }{ \theta \Big( \prod\limits_{I \in [h(e), \bar\ft]} \dfrac{u_{c_I}}{u_{c_I+1}} \Big)}.  
	$$
	Therefore, we conclude that among (\ref{pospols}) only factors with $s_{ij}=0$ may appear. To show that those are actually not poles, it suffices to prove that 
	$$
	\theta \Big( \frac{u_i}{u_j} \Big) T_{\bf p, p} T'_{\lambda, \mu} \bigg|_{u_i = u_j} = 0. 
	$$
	As discussed before, the only possible nontrivial terms of the LHS come from trees $\bar\ft$ which contains some $(i,j)$-strip $B$. 
	
	If $j \in {\bf p}$, one can see that $\bar\lambda\backslash B$ contains a path in $\bar\ft$ admitting a box with local maximal content, which is not allowed. In other words, contributions from all $\bar\ft$ are zero in this case. 
	
	If $i \in {\bf n} \backslash {\bf p}$, then the boxes above and to the left of the root of $B$ both lie in $\bar\lambda$, and the involution $\inv(\bar\ft)$ is also a tree in $\bar\lambda$. By the cancellation Lemma \ref{lemma-cancel}, contribution from $\bar\ft$ cancels with that from $\inv(\bar\ft)$. Sum over all $\bar\ft$ gives $0$. 
	
	If $i\in {\bf p}$ and $j \in {\bf n} \backslash {\bf p}$, then $T_{\bf p, p}$ contains a factor $\theta \Big( \dfrac{u_i}{u_j} \Big)$, and nontrivial terms come from trees $\bar\ft$ that contains at least two $(i,j)$-strips, e.g.,  $B_1$, $B_2$. At least one of them, say $B_1$, is not contained in the boundary of $\bar\lambda$ and hence the involution of $\bar\ft$ with respect to $B_1$ is well-defined. Contribution from $\bar\ft$ then cancels with that from $\inv (\bar\ft)$. Therefore, we exclude all possible poles $\theta (u_i / u_j)$, and $T_{\bf p, p} T'_{\lambda, \mu}$ is holomorphic in $u_i$. 
\end{proof}


\bibliographystyle{abbrv}
\bibliography{bib}

\newpage

\vspace{12 mm}

\noindent
Richárd Rimányi\\
Department of Mathematics,\\
University of North Carolina at Chapel Hill,\\
Chapel Hill, NC 27599-3250, USA\\
rimanyi@email.unc.edu

\vspace{3 mm}

\noindent
Andrey Smirnov\\
Department of Mathematics,\\
University of North Carolina at Chapel Hill,\\
Chapel Hill, NC 27599-3250, USA;\\
Institute for Problems of Information Transmission\\
Bolshoy Karetny 19, Moscow 127994, Russia\\
asmirnov@email.unc.edu

\vspace{3 mm}

\noindent
Zijun Zhou\\
Department of Mathematics,\\
Stanford University ,\\
450 Serra Mall, Stanford, CA 94305, USA\\
zz2224@stanford.edu

\vspace{3 mm}

\noindent
Alexander Varchenko\\
Department of Mathematics,\\
University of North Carolina at Chapel Hill,\\
Chapel Hill, NC 27599-3250, USA\\
Faculty of Mathematics and Mechanics\\
Lomonosov Moscow State
University,\\
Leninskiye Gory 1, 119991 \\
Moscow GSP-1, Russia,\\
anv@email.unc.edu,

\end{document}